\pdfoutput=1
\documentclass[english]{smfart}
\usepackage{hien-sabbah14}
\begin{document}
\frontmatter
\title[Local Laplace transform]{The local Laplace transform of\\ \hbox{an elementary irregular meromorphic connection}}

\author[M.~Hien]{Marco Hien}
\address{Lehrstuhl für Algebra und Zahlentheorie\\
Universitätsstraße 14\\
86159 Augsburg\\
Deutschland}
\email{marco.hien@math.uni-augsburg.de}
\urladdr{http://alg.math.uni-augsburg.de/mitarbeiter/mhien/}

\author[C.~Sabbah]{Claude Sabbah}
\address{UMR 7640 du CNRS\\
Centre de Mathématiques Laurent Schwartz\\
École polytechnique\\
F--91128 Palaiseau cedex\\
France}
\email{Claude.Sabbah@polytechnique.edu}
\urladdr{http://www.math.polytechnique.fr/~sabbah}

\thanks{M.H.: This research was supported by the grant DFG-HI 1475/2-1 of the Deutsche Forschungsgemeinschaft\\
C.S.: This research was supported by the grants ANR-08-BLAN-0317-01 and ANR-13-IS01-0001-01 of the Agence nationale de la recherche.}

\begin{abstract}
We give a definition of the topological local Laplace transformation for a Stokes-filtered local system on the complex affine line and we compute in a topological way the Stokes data of the Laplace transform of a differential system of elementary type.
\end{abstract}


\maketitle
\tableofcontents

\mainmatter

\section{Introduction}
\Subsection{Riemann-Hilbert correspondence and Laplace transformation}\label{subsec:general}
Let $M$ be a holonomic $\Clt$-module and let $\Fou M$ be its Laplace transform, which is a holonomic $\CC[\tau']\langle\partial_{\tau'}\rangle$-module through the correspondence $\tau'=\partial_t$ and $t=-\partial_{\tau'}$, that is, given by the kernel $\exp(-t\tau')$. According to the Riemann-Hilbert correspondence as stated by Deligne (see \cite{Deligne77}, \cite{Malgrange91} and \cite[Chap.\,5]{Bibi10}), the holonomic $\Clt$-module $M$ corresponds to a Stokes-perverse sheaf on the projective line $\PP^1_t$ with affine coordinate $t$, which is a Stokes-filtered local system in the neighbourhood of $t=\infty$. Similarly, $\Fou M$ corresponds to a Stokes-perverse sheaf on $\PP^1_{\tau'}$, which is a Stokes-filtered local system in the neighbourhood of $\tau'=\infty$. In the following, we will denote by $\tau$ the coordinate of $\PP^1_{\tau'}$ centered at $\wh\infty$ such that $\tau=1/\tau'$ on $\PP^1_{\tau'}\moins\{0,\wh\infty\}$, and we continue to set $\wh\infty=\{\tau=0\}$. The topological Laplace transformation is the corresponding transformation at the level of the categories of Stokes-perverse sheaves. We will use results of \cite{Mochizuki10} to make clear its definition. A topological Laplace transformation can also be defined in the setting of enhanced ind-sheaves considered in \cite{D-K13}, which corresponds, through the Riemann-Hilbert correspondence of \loccit, to the Laplace transformation of holonomic $\cD$-modules (see also \cite{K-S14}).

\Subsection{Riemann-Hilbert correspondence and local Laplace transformation}\label{subsec:local}
In this article, we will consider the local Laplace transformation $\ccF^{(0,\infty)}$ from the category of finite dimensional $\CC\lpb t\rpb$-vector spaces with connection to that of finite dimensional $\CC\lpb \tau\rpb$-vector spaces with connection. It is defined as follows. A~finite dimensional $\CC\lpb t\rpb$-vector space with connection $\ccM$ can be extended in a unique way as a holonomic $\Clt$-module $M$ with a regular singularity at infinity and no other singularity at finite distance than $t=0$. Then $\ccF^{(0,\infty)}(\ccM)$ is by definition the germ~$\Fou M_{\wh\infty}$ of $\Fou M$ at $\wh\infty:=\{\tau'=\infty\}$. We will regard $\ccF^{(0,\infty)}(\ccM)$ as a $\CC\lpb \tau\rpb$-vector space with connection. It is well-known that $\wh M$ has at most a regular singularity at $\tau'=0$ and no other singularity at finite distance. Therefore, giving $\ccF^{(0,\infty)}(\ccM)$ is equivalent to giving the localized module $\CC[\tau',\tau^{\prime-1}]\otimes_{\CC[\tau']}\wh M$.

The Deligne-Riemann-Hilbert correspondence associates to $\ccM$ a Stokes-filtered local system $(\cL,\cL_\bbullet)$ on $\SS^1_{t=0}$ (the circle with coordinate $\arg t$). Similarly, we will denote by $\ccF^{(0,\infty)}_\top(\cL,\cL_\bbullet)$ the Stokes-filtered local system on $\SS^1_{\tau=0}$ associated with $\ccF^{(0,\infty)}(\ccM)$. In this setting, we will address the following questions:
\begin{itemize}
\item
To make explicit the topological local Laplace transformation functor $\ccF^{(0,\infty)}_\top$.
\item
To use this topological definition to compute, in some examples, the Stokes structure of $\ccF^{(0,\infty)}(\ccM)$ in terms of that of $\ccM$.
\end{itemize}

\subsection{Statement of the results}
The first goal will be achieved in \S\ref{sec:topLaplace}. As for the second one, we will restrict to the family of examples consisting of elementary meromorphic connections, denoted by $\El(\rho,-\vi,\ccR)$ in \cite{Bibi07a}. Namely, $\rho:u\mto t=u^p$ is a ramification of order~$p$ of the variable~$t$, $\vi$ is a polynomial in $u^{-1}$ without constant term, and $\ccR$ is a finite dimensional $\CC\lpb u\rpb$-vector space with a regular connection. Setting $\ccE^{-\vi}=(\CC\lpb u\rpb,\rd-\nobreak\rd\vi)$, we define
\[
\ccM=\El(\rho,-\vi,\ccR):=\rho_+(\ccE^{-\vi}\otimes \ccR).
\]
If we extend $\ccR$ as a free $\CC[u,u^{-1}]$-module $R$ of finite rank with a connection having a regular singularity at $u=0$ and $u=\infty$ and no other singularity, and set $E^{-\vi}=(\CC[u,u^{-1}],\rd-\rd\vi)$, the extension $M$ of $\ccM$ considered above is nothing but the free $\CC[t,t^{-1}]$-module with connection
\[
M=\El(\rho,-\vi,R):=\rho_+(E^{-\vi}\otimes R).
\]
Let $q$ be the pole order of $\vi $. We also encode $R$ as a pair $(\bV,\bT)$ of a vector space with an automorphism (the formal monodromy). We call $\vi$ the \emph{exponential factor} of $\rho^+\ccM$ and set $\Exp=\Exp(\rho^+\ccM)=\{\vi\}$.

\begin{assumption}\label{ass:pqcoprime}
We will assume in this article that $p,q$ are coprime.
\end{assumption}

By the stationary phase formula of \cite{Fang07, Bibi07a} the formal Levelt-Turrittin type of the local Laplace transform $\ccF^{(0,\infty)}(\ccM)$ is that of $\El(\wh\rho,-\wvi,\wh R)$, where
\begin{itemize}
\item
the ramification $\wh\rho:\eta\mto\tau$ has order $\wh p:=p+q$, \ie $\wh\rho\in\eta^{p+q}(c(\rho,\vi)+\eta\CC\{\eta\})$ for some nonzero constant $c(\rho,\vi)$; writing $\rho$ and $\vi$ in the $\eta$ variable, it is expressed as $\wh\rho(\eta)=-\rho'(\eta)/\vi'(\eta)$;
\item
the exponential factor $\wvi(\eta)=\vi(\eta)-\vi'(\eta)\rho(\eta)/\rho'(\eta)$ has pole order $q$;
\item
the regular part $\wh R$ corresponds to the pair $(\wh\bV,\wh\bT)=(\bV,(-1)^q\bT)$.
\end{itemize}
Our goal is then to compute explicitly the Stokes structure of the non-ramified meromorphic connection $\wh\rho^+\ccF^{(0,\infty)}(\ccM)$, which is usually non trivial. It is of pure level~$q$. Stokes structures of pure level $q$ can be represented in various ways by generalized monodromy data. We make use of the following description (see Section \ref{sec:Stokesdata} for details).

To the formal data attached to $\wh\rho^+\ccF^{(0,\infty)}(\ccM)$, which consist of the family $\wh\Exp:=\Exp(\wh\rho^+\ccF^{(0,\infty)}(\ccM))=\{\wvi_{\wzeta}(\eta)\mid\wzeta\in\mu_{p+q}\}$ ($\mu_n$ is the group of $n$-th roots of unity and $\wvi_{\wzeta}=\wvi(\wzeta\eta(1+\ro(\eta))$) and the pair
$(\oplus_{\wzeta\in\mu_{p+q}}\bV,\oplus_{\wzeta\in\mu_{p+q}}(-1)^q\bT)$, we add:

\refstepcounter{equation}\label{eq:genericargument}
\noindent\eqref{eq:genericargument}\enspace
the choice of a generic argument $\vtn_o\in\SS^1_{\eta=0}$, giving rise to a total ordering of the set~$\wh\Exp$, hence an order preserving numbering $\{0,\dots,p+q-1\}\simeq\mu_{p+q}$,

\smallskip
\refstepcounter{equation}\label{eq:linearStokesData}
\noindent\eqref{eq:linearStokesData}\enspace
a set of \emph{unramified linear Stokes data $\bigl((\bLMH_\ell),(S_\ell^{\ell+1}), (F\bLMH_\ell)\bigr)_{\ell=0,\dots,2q-1}$ of pure level $q$}, which consists~of
\begin{itemize}
\item
a family $(\bLMH_\ell)_{\ell=0,\dots,2q-1}$ consisting of $2q$ $\CC$-vector spaces,
\item
isomorphisms $S_\ell^{\ell+1}: \bLMH_{\ell} \isom \bLMH_{\ell+1}$,
\item
for each $\mu\in\{0,\dots,q-1\}$, a finite exhaustive\footnote{\label{ftn:exhaustive}We say that an increasing filtration $F_\bbullet\bLMH$ indexed by a totally ordered finite set $O$ is \emph{exhaustive} if $F_{\max O}\bLMH=\bLMH$. We then set $F_{<\min O}\bLMH=0$. For a decreasing filtration, we have $F^{\min O}\bLMH=\bLMH$ and we set $F^{>\max O}\bLMH=0$.} increasing filtration $F_\bbullet\bLMH_{2\mu}$ (\resp decreasing filtration $F^\bbullet\bLMH_{2\mu+1}$) indexed by $\{0,\dots,p+q-1\}$, with the property that the filtrations are mutually opposite with respect to the isomorphisms $S_\ell^{\ell+1}$, \ie for any $\mu\in\{0,\dots,q-1\}$:
\begin{starequation}\label{eq:opposite1intro}
\begin{split}
\bLMH_{2 \mu}&=\bigoplus_{k=0}^{p+q-1}\bigl(F_k\bLMH_{2 \mu}\cap S_{2 \mu-1}^{2 \mu}(F^k\bLMH_{2 \mu-1})\bigr),\\
\bLMH_{2 \mu +1}&=\bigoplus_{k=0}^{p+q-1}\bigl(F^k\bLMH_{2 \mu +1}\cap S_{2 \mu}^{2 \mu +1}(F_k\bLMH_{2 \mu})\bigr),
\end{split}
\end{starequation}
with the convention that $\bLMH_{-1}=\bLMH_{2q-1}$.
\end{itemize}

Note that the opposite filtrations induce uniquely determined splittings, that is, filtered isomorphisms
\begin{equation}
\begin{cases}
\tau_{2 \mu}:\bLMH_{2 \mu} \isom\gr^F \bLMH_{2 \mu}=:\bigoplus_{k=0}^{p+q-1}\gr^F_k \bLMH_{2 \mu},\\
\tau_{2 \mu +1}:\bLMH_{2 \mu +1} \isom\gr_F \bLMH_{2 \mu +1}=:\bigoplus_{k=0}^{p+q-1}\gr_F^k\bLMH_{2 \mu +1},
\end{cases}
\end{equation}
where the right hand side carries its natural increasing (\resp decreasing) filtration. The resulting isomorphisms $\Sigma_\ell^{\ell+1} :=\tau_{\ell+1} \circ S_\ell^{\ell+1} \circ \tau_\ell^{-1}: \gr_F \bLMH_\ell\to\gr_F \bLMH_{\ell+1}$ are upper/lower block triangular -- depending on the parity of $\ell$ -- and their blocks $(S_\ell^{\ell+1})_{j, j}$ on the diagonal are isomorphisms. These matrices are commonly called the Stokes matrices or Stokes multipliers. We call the family
\begin{equation}\label{eq:fracSFM}
\fS\bigl(\wh\rho^+\ccF^{(0,\infty)}(\ccM)\bigr):=\bigl((\bLMH_\ell),(S_\ell^{\ell+1}), (F\bLMH_\ell)\bigr)_{\ell=0,\dots,2q-1}
\end{equation}
the linear part of the Stokes data of $\wh\rho^+\ccF^{(0,\infty)}(\ccM)$. The notion of morphism between such families is obvious.

On the other hand, we will define in \S\ref{subsec:modelStokes} standard linear Stokes data that we denote by $\fS^\std(\bV, \bT, p, q)$.

The main result regarding the explicit determination of the Stokes data of the local Laplace transform of $\El(\rho, - \vi, \ccR)$ can be stated as follows.

\begin{theoreme}\label{thm:mainintro}
With the previous notation and assumptions, for a suitable choice of~$\vtn_o$, the linear Stokes data $\fS\bigl(\wh\rho^+\ccF^{(0,\infty)}(\ccM)\bigr)$, for $\ccM=\El(\rho,-\vi,(\bV,\bT))$, are isomorphic to the standard linear Stokes data $\fS^\std(\bV, \bT, p, q)$.
\end{theoreme}

\begin{remarques}\mbox{}
\begin{enumerate}
\item
In order to descend from this result to the linear Stokes data of $\ccF^{(0,\infty)}(\ccM)$, we would need to identify the $\mu_{p+q}$-action on the standard linear Stokes data. This will not be achieved in this article.
\item
The arguments in the proof also lead to an explicit computation of the topological monodromy of $\ccF^{(0,\infty)}(\ccM)$ up to conjugation (see Proposition \ref{prop:topmonointro}).
\item
Similar results were obtained by T.\,Mochizuki in \cite[\S3]{Mochizuki09b}, by using explicit bases of homology cycles, while we use here cohomological methods. Also, the results of \loccit do not make explicit a standard model for the linear Stokes data, although it should be in principle possible to obtain such a model from these results.
\end{enumerate}
\end{remarques}

\subsection{The standard linear Stokes data}\label{subsec:modelStokes}
We will now define the set $\fS^\std(\bV,\bT,p,q)$ attached to a finite dimensional vector space $\bV$ equipped with an automorphism $\bT$, and to a pair of coprime integers $(p,q)\in(\NN^*)^2$. We refer to Section \ref{sec:standard} for a geometric motivation which leads to the definitions below. We start by defining two orderings on $\mu_{p+q}$.

\begin{definition}\label{def:standardorder}
For $\wzeta,\wzeta'\in\mu_{p+q}$, we set\enlargethispage{\baselineskip}%
\begin{starequation}\label{eq:standardorder}
\wzeta\lodd\wzeta'\text{ if }\reel(e^{-\ve i}(\wzeta^p-\wzeta^{\prime p}))<0
\quad\text{for some $\ve$ such that $0<\ve\ll1$}
\end{starequation}
and $\wzeta \lev \wzeta' \Leftrightarrow\wzeta' \lodd\wzeta$. We enumerate $\mu_{p+q}$ according to the even ordering, \ie
\begin{starstarequation}\label{eq:evenum}
\mu_{p+q}=\{\wzeta_k\mid k=0,\dots,p+q-1\}\quad\text{with }\wzeta_0\lev\cdots\lev\wzeta_{p+q-1}.
\end{starstarequation}%
\end{definition}

Since $(p,q)=1$, we also have $(p+q,q)=1$ and there exist $a,b\in\ZZ$ with $ap=1+b(p+q)$. Set $\xi:=\exp(\frac{2\pi i}{p+q})\in\mu_{p+q}$. Then the \emph{even ordering} on $\mu_{p+q}$ is expressed~by
\begin{equation}\label{eq:ordev}
1=\xi^0 \lev \xi^{a} \lev \xi^{-a} \lev \cdots\lev\xi^{ka} \lev \xi^{-ka} \lev\cdots\lev\ximaxev,
\end{equation}
with $k\in[1,\frac{p+q}2)$, with
\[
\ximaxev=\begin{cases}
\pm\exp(\frac{a\pi i}{p+q})&\text{if $p+q$ is odd and $a$ is even ($+$) or odd ($-$)},\\
\exp(a\pi i)=-1&\text{if $p+q$ is even}.
\end{cases}
\]
The \emph{odd ordering} is the reverse ordering.

For $k\in\ZZ$ we set
\begin{equation}\label{eq:evinout}
\left\{\begin{aligned}
\evin(k)&:=\left[\frac{qk}{p+q}+\frac{1}{2} \right],\\
\evout(k)&:=k-\evin(k)=\left\lceil\frac{pk}{p+q}-\frac{1}{2} \right\rceil.
\end{aligned}
\right.
\end{equation}
Both are increasing functions of $k$. If $k\in\{0,\dots, p+q-1\}$, we have $\evin(k)\in\{0,\dots,q\}$ and $\evout(k)\in\{0,\dots,p\}$. Furthermore, let
\begin{equation}\label{eq:minin}
\evminin(k):=\min \Big\{k' \ge -\frac{p+q}{2q} \mid \evin(j)= \evin(k) \text{ for all } j \in [k',k] \Big\}
\end{equation}
and
\begin{equation}\label{eq:maxout}
\evmaxout(k):=\max \Big\{k'' < p+q + \frac{p+q}{2p} \mid \evout(j) \!=\!\evout(k) \text{ for all } j \!\in\! [k,k''] \Big\}
\end{equation}
(see Remark \ref{rem:cev}\eqref{rem:cev3} for an explanation). In a similar way we set for $k\in\ZZ$:
\begin{equation}\label{eq:oddinout}
\left\{\begin{aligned}
\oddin(k)&:=\left[\frac{qk}{p+q}\right],\\
\oddout(k)&:=k-1-\oddin(k)=\left\lceil\frac{pk}{p+q}-1\right\rceil\in\{-1,\dots,p-1\},
\end{aligned}
\right.
\end{equation}
and if $k\in\{0,\dots,p+q-1\}$ we have $\oddin(k)\in\{0,\dots,q-1\}$ and $\oddout(k)\in\{-1,\dots,p-1\}$. We define $\oddminin(k)$ and $\oddmaxout(k)$ literally as in \eqref{eq:minin} and \eqref{eq:maxout} replacing ``ev'' by ``odd''. Note that
$$
\Big[\frac{qk}{p+q}-\frac{1}{2}, \frac{qk}{p+q}\Big]\cap\N=\emptyset \Longleftrightarrow \evin(k)=\oddin(k)+1.
$$

Let $\bV $ be a finite dimensional complex vector space together with an automorphism $\bT\in\Aut(\bV)$. Let us denote by ${\bf1}_k$ the complex vector space of rank one with chosen basis element $1_k$. Let us consider direct sum $\bV^{p+q}$ of $p+q$ copies of $\bV $ and let us keep track of the indices by writing
\begin{equation}\label{eq:bVpq}
\bV^{p+q}=\bigoplus_{k=0}^{p+q-1}\bV \otimes_{\C} {\bf1}_k.
\end{equation}
\begin{convention}\label{conv:alphaT}
For any $j\in\ZZ$, we will write
$$
v \otimes 1_j :=
\bT^{-s}(v) \otimes 1_{j+s(p+q)} \text{ for } s\in\Z \text{ such that } j+s(p+q)\in [0, p+q-1].
$$
\end{convention}

\pagebreak[2]
\begin{definition}\label{def:L}\mbox{}\par
\noindent$\cbbullet$ Assume $p+q \ge 3 $. Then:
\begin{enumerate}
\item
\label{def:L01}
The isomorphism $\sigma_{\ev}^{\odd}:\bigoplus_{k=0}^{p+q-1}\bV \otimes{\bf1}_k\to\bigoplus_{k=0}^{p+q-1}\bV \otimes{\bf1}_k$ is defined as
\begin{starequation}
\label{eq:defL01}
\sigma_{\ev}^{\odd}(v \otimes 1_k) :=
\begin{cases}
v \otimes 1_k, & \text{if } [\frac{qk}{p+q}-\frac{1}{2}, \frac{qk}{p+q}]\cap\N=\emptyset \\
v \otimes 1_{\oddminin(k)-1}\\
\hspace*{4mm}{}- v \otimes 1_{\oddmaxout(\oddminin(k))}\\
\hspace*{12mm}{}+v \otimes 1_{\oddmaxout(k)+1},
& \text{if } [\frac{qk}{p+q}-\frac{1}{2}, \frac{qk}{p+q}]\cap\N\neq\emptyset
\end{cases}
\end{starequation}

\item
\label{def:L12}
The isomorphism $\sigma^{\ev}_{\odd}:\bigoplus_{k=0}^{p+q-1}\bV \otimes{\bf1}_k\to\bigoplus_{k=0}^{p+q-1}\bV \otimes{\bf1}_k$ is defined as
\begin{starstarequation} \label{eq:defL12}
\sigma^{\ev}_{\odd}(v \otimes 1_k) :=
\begin{cases}
v \otimes 1_k,& \text{if } [\frac{qk}{p+q}-\frac{1}{2}, \frac{qk}{p+q}]\cap\N\neq\emptyset \\
v \otimes 1_{\evminin(k)-1}\\
\hspace*{4mm}{} - v \otimes 1_{\evmaxout(\evminin(k))}\\
\hspace*{12mm}{}+v \otimes 1_{\evmaxout(k)+1},& \text{if } [\frac{qk}{p+q}-\frac{1}{2}, \frac{qk}{p+q}]\cap\N=\emptyset
\end{cases}
\end{starstarequation}
\end{enumerate}

\noindent$\cbbullet$
If $p=q=1 $, the isomorphism $\sigma_{\ev}^{\odd}$ is defined as
\begin{equation}\label{eq:defpq1i}
\sigma_{\ev}^{\odd}(v \otimes 1_0) := - v \otimes 1_0+(\id+T^{-1}) v \otimes 1_1 \quad\text{and} \quad
\sigma_{\ev}^{\odd}(v \otimes 1_1) := v \otimes 1_1,
\end{equation}
and $\sigma_{\odd}^{\ev}$ is defined as
\begin{starequation}\label{eq:defpq1ii}
\sigma^{\ev}_{\odd}(v \otimes 1_0) := v \otimes 1_0 \quad\text{and} \quad \sigma^{\ev}_{\odd}(v \otimes 1_1):= (\id+T^{-1})v \otimes 1_0 - v \otimes 1_1.
\end{starequation}
\end{definition}

We are now ready to define the following particular set of Stokes data.

\begin{definition}\label{def:standstokesintro}
The standard linear Stokes data $\fS^\std(\bV, \bT,p,q)$ are given by
\begin{enumerate}
\item\label{def:standstokesintro1}
the vector spaces $\bLMH_\ell:=\bV^{p+q}$ for $\ell=0,\dots, 2q-1$,
\item\label{def:standstokesintro2}
for each $\mu =0,\dots, q-1$, the isomorphisms
\begin{align*}
S_{2 \mu-1}^{2 \mu}:=\sigma^{\ev}_{\odd}: \bLMH_{2 \mu-1} &\isom \bLMH_{2 \mu}
\quad\text{(\cf\eqref{eq:defL12})},\\
S^{2 \mu +1}_{2 \mu}:=\sigma_{\ev}^{\odd}: \bLMH_{2 \mu} &\isom \bLMH_{2 \mu +1}
\quad\text{(\cf\eqref{eq:defL01})}.
\end{align*}
\item\label{def:standstokesintro3}
the isomorphism $S_{2q-1}^{2q} :=\diag(\bT,\dots, \bT) \cdot \sigma^{\ev}_{\odd}: \bLMH_{2q-1} \isom \bLMH_0$.
\item\label{def:standstokesintro4}
the filtrations $F_k\bLMH_{2 \mu} :=\bigoplus_{k' \leq k}\bV\otimes{\bf1}_{k'}$ and $F^k\bLMH_{2 \mu +1} :=\bigoplus_{k'\geq k}\bV\otimes{\bf1}_{k'}$.
\end{enumerate}
\end{definition}

Note that the opposedness property \eqref{eq:opposite1intro} for the above data is not obvious. We will not give a direct proof since it follows a posteriori from our main result by which these data are the Stokes data attached to some Stokes structure. The arguments in the proof of Theorem \ref{thm:mainintro} also lead to the following explicit computation of the topological monodromy of $\wh\rho^+\ccF^{(0,\infty)}(\ccM)$, where we set
\begin{equation}\label{eq:ink}
\mathrm{in}(k):=\evin(k)=\left[\frac{qk}{p+q}+\frac12\right], \quad \max\nolimits_\mathrm{out}(k):=\evmaxout(k).
\end{equation}

\begin{proposition} \label{prop:topmonointro}
The topological monodromy of $\wh\rho^+\ccF^{(0,\infty)}(\ccM)$ is conjugate to the automorphism of the vector space $\bV^{p+q}=\bigoplus_{k=0}^{p+q-1}\bV \otimes_\C {\bf1}_k$ (using the notation \eqref{eq:bVpq} and Convention \ref{conv:alphaT}) given by
$$
\wh T_\top(v\otimes1_k)
=\begin{cases}
v\otimes1_{k+1}& \text{if }\mathrm{in}(k+1)=\mathrm{in}(k),\\
v\otimes(1_k-1_{\max_\mathrm{out}(k)}+1_{\max_\mathrm{out}(k)+1})& \text{if }\mathrm{in}(k+1)=\mathrm{in}(k)+1.
\end{cases}
$$
\end{proposition}

\begin{remarque}\label{rem:topmdrmy}
The topological monodromy can of course be deduced from the Stokes data by the formula$$
\Ttop :=S_{2q-1}^0 \circ S_{2q-2}^{2q-1} \circ \cdots \circ S_0^1\in\Aut(\bLMH_0).
$$
However, we will give a more direct computation in \S\ref{subsec:topmono}. This approach is less involved than the detour over the Stokes data and could perhaps be applied in more general situations without having to understand the full information of the Stokes data.
\end{remarque}

The article is organized as follows. Section \ref{sec:StData} recalls the notion of Stokes-filtered local system and explains the correspondence with that of linear Stokes data -- generalized monodromy data -- describing Stokes structures of pure level $q$. In Section~\ref{sec:topLaplace} the construction of the topological Laplace transformation is discussed in general, and the main tool (Th.\,\ref{th:pushforward}) is \cite[Cor\ptbl4.7.5]{Mochizuki10}. We then concentrate on the case of an elementary meromorphic connection. A descrip\-tion of the Stokes data for $\wh\rho^+\ccF^{(0,\infty)}(\ccM)$ in cohomological terms is proved in the subsequent section -- Theorem \ref{theo:Stokesdataabstract}.

The remaining sections provide the proof that these data are isomorphic to the standard linear Stokes data of Definition \ref{def:standstokesintro}. Section \ref{sec:support} identifies the filtered vector spaces $(\bLMH_\ell,F\bLMH_\ell)_\ell$ obtained by the cohomological computation of Theorem \ref{theo:Stokesdataabstract} with the filtered vector spaces entering in the standard linear Stokes data. This is done by using Morse theory. However, the Morse-theoretic description of the linear Stokes data expressed through Theorem \ref{theo:Stokesdataabstract} does not seem suitable to compute the Stokes matrices~$S_\ell^{\ell+1}$. This is why, in Section \ref{sec:standard}, we construct a simple Leray covering giving rise to a basis of the cohomology space $\bLMH_\ell$ for each $\ell$, which allows us to identify the matrices $S_\ell^{\ell+1}$ given by Theorem \ref{theo:Stokesdataabstract} to those given by the standard linear Stokes data.

It remains to compare the Stokes filtration obtained in Theorem \ref{theo:Stokesdataabstract} with that obtained from the standard linear Stokes data. The problem here is that the simple Leray covering constructed in Section \ref{sec:standard} may not be adapted to the Morse-theoretic computation. In Section \ref{finalsection} we change the construction of the Leray covering in order both to keep the same combinatorics to compute the matrices $S_\ell^{\ell+1}$ and to identify the filtration of Theorem \ref{theo:Stokesdataabstract} with the standard Stokes filtration \ref{def:standstokesintro}\eqref{def:standstokesintro4}. This is the contents of Corollary \ref{cor:isofil}, which concludes the proof of Theorem \ref{thm:mainintro}.

\subsubsection*{Acknowledgements}
The first author thanks Giovanni Morando for useful discussions. The second author thanks Takuro Mochizuki for many discussions on the subject. Both authors thank the referee for having carefully read the manuscript and for having proposed various improvements.

\newpage
\section{Stokes-filtered local systems and Stokes data} \label{sec:StData}

\subsection{Reminder on Stokes-filtered local systems}\label{subsec:Stokesfilteredls}
We refer to \cite{Deligne77}, \cite{Malgrange91} and \cite[Chap\ptbl2]{Bibi10} for more details.

Let $\varpi:\wt\Afut\to\Afut$ be the real blowing-up of the origin in $\Afut$. Then $\wt\Afut=\Afut{}^*\cup\SS^1_{t=0}$ is homeomorphic to a semi-closed annulus, with $\SS^1_{t=0}:=\varpi^{-1}(0)\subset\wt\Afut$. Similarly, we consider $\wt\Afuu$, and $\rho$ extends as the $p$-fold covering $\wt\rho:\SS^1_{u=0}\to\SS^1_{t=0}$. For any $\vt\in\SS^1_{u=0}$, the order on $u^{-1}\CC[u^{-1}]$ at $\vt$ is the additive order defined by
\begin{multline}\label{eq:leqvt}
\vi\leqvt0\iff\\
\text{$\exp(\vi(u))$ has moderate growth in some open sector centered at $\vt$.}
\end{multline}
We set
\begin{equation}\label{eq:levt}
\vi\levt0\iff\vi\leqvt0\text{ and }\vi\neq0.
\end{equation}
This is equivalent to $\exp(\vi(u))$ having rapid decay in some open sector centered at~$\vt$.

\begin{definition}[{see \eg\cite[Chap.\,2]{Bibi10}}]\label{def:stokesfiltlocsyst}
A Stokes-filtered local system with ramification $\rho$ consists of a local system $\cL$ on $\SS^1_{t=0}$ such that $\wt\rho^{-1}\cL$ is equipped with a family of subsheaves $\cL_{\leq\vi}\subset\wt\rho^{-1}\cL$ indexed by a finite subset $\Exp\subset u^{-1}\CC[u^{-1}]$ with the following properties:
\begin{enumerate}
\item\label{def:stokesfiltlocsyst1}
For each $\vt\in\SS^1_{u=0}$, the germs $\cL_{\leq\vi,\vt}$ ($\vi\in\Exp$) form an increasing filtration of $(\wt\rho^{-1}\cL)_\vt=\cL_{\wt\rho(\vt)}$.
\item\label{def:stokesfiltlocsyst2}
Set $\cL_{<\vi,\vt}=\sum_{\psi\levt\vi}\cL_{\leq\psi,\vt}$, where the sum is taken in $\cL_{\wt\rho(\vt)}$ (and the sum indexed by the empty set is zero). Then $\cL_{<\vi,\vt}$ is the germ at $\vt$ of a subsheaf $\cL_{<\vi}$ of~$\cL_{\leq\vi}$.
\item\label{def:stokesfiltlocsyst3}
The filtration is exhaustive, \ie $\bigcup_{\vi\in\Exp}\cL_{\leq\vi}=\wt\rho^{-1}\cL$ and $\bigcap_{\vi\in\Exp}\cL_{<\vi}=0$.
\item\label{def:stokesfiltlocsyst4}
Each quotient $\gr_\vi\cL:=\cL_{\leq\vi}/\cL_{<\vi}$ is a local system.
\item\label{def:stokesfiltlocsyst5}
$\sum_{\vi\in\Exp}\rk\gr_\vi\cL=\rk\cL$.
\item\label{def:stokesfiltlocsyst6}
For each $\zeta\in\mu_p$ and $\vi\in\Exp$, setting $\vi_\zeta(u):=\vi(\zeta u)$, we have $\cL_{\leq\vi_\zeta}\simeq\wt\zeta^{-1}\cL_{\leq\vi}$ through the equivariant isomorphism $\wt\rho^{-1}\cL\simeq\wt\zeta^{-1}\wt\rho^{-1}\cL$, where $\wt\zeta:\SS^1_{u=0}\to\SS^1_{u=0}$ is induced by the multiplication by $\zeta$.
\end{enumerate}
\end{definition}

\begin{remarque}
It is equivalent to start with a family of subsheaves $\cL_{<\vi}\subset\wt\rho^{-1}\cL$ such that $\cL_{<\vi,\vt}\subset\cL_{<\psi,\vt}$ if $\vi\leqvt\psi$, and to define $\cL_{\leq\vi}$ by the formula $\cL_{\leq\vi,\vt}=\bigcap_{\psi,\,\psi\levt\vi}\cL_{<\psi,\vt}$ with the convention that the intersection indexed by the empty set is~$\cL_\vt$. The family $\cL_{<\vi}$ will be simpler to obtain in our computation of Laplace transform.
\end{remarque}

\begin{remarque}[Extension of the index set]\label{rem:extensionindex}
It is useful to extend the indexing set of the filtration and to define the subsheaves $\cL_{\leq\psi}\subset\wt\rho^{-1}\cL$ for any $\psi\in u^{-1}\CC[u^{-1}]$. For such a $\psi$ and for any $\vi\in\Exp$, we denote by $\SS^1_{\vi\leq\psi}$ the open subset of $\SS^1_{u=0}$ defined~by
\[
\vt\in\SS^1_{\vi\leq\psi}\iff\vi\leqvt\psi.
\]
We shall abbreviate by $(\beta_{\vi\leq\psi})_!$ the functor on sheaves composed of the restriction to this open set and the extension by zero from this open set to $\SS^1_{u=0}$. Then $\cL_{\leq\psi}$ is defined by the formula
\begin{starequation}\label{eq:extindex}
\cL_{\leq\psi}=\sum_{\vi\in\Exp}(\beta_{\vi\leq\psi})_!\cL_{\leq\vi}.
\end{starequation}%
The germ of $\cL_{\leq\psi}$ at any $\vt\in\SS^1_{u=0}$ can also be written as
\begin{starstarequation}\label{eq:extindexx}
\cL_{\leq\psi,\vt}=\sum_{\substack{\vi\in\Exp\\\vi\leqvt\psi}}\cL_{\leq\vi,\vt},
\end{starstarequation}%
with the convention that the sum indexed by the empty set is zero. Note also that, if $\psi$ is not ramified, \ie $\psi\in t^{-1}\CC[t^{-1}]$, then $\cL_{\leq\psi}$ is a subsheaf of $\cL$. A similar definition holds for $\cL_{<\psi}$, but one checks that $\cL_{<\psi}=\cL_{\leq\psi}$ if $\psi\notin\Exp$. Of particular interest will be $\cL_{\leq0}\subset\cL$.
\end{remarque}

\subsection{Stokes data of pure level \texorpdfstring{$q$}{q}} \label{sec:Stokesdata}

In this section, we will define the notion of Stokes data attached to a given Stokes structure $(\cL, \cL_\bbullet)$, which -- after various choices to be fixed a priori -- describes the Stokes filtration in terms of opposite filtrations on a generic stalk. After choosing appropriate basis, the passage from one filtration to the other could be expressed by matrices. These matrices are usually referred to as the {\em Stokes matrices} or {\em Stokes multipliers} in different approaches to Stokes structures (see e.g. \cite{B-J-L79}). In level $q=1$, the following description has already been given in \cite{H-S09}, section 2.b.

We will restrict to the case of an unramified Stokes-filtered local system, \ie we assume that, in Definition \ref{def:stokesfiltlocsyst}, $\rho$ is equal to $\id$, and we will work with the variable~$u$. We will set $\SS^1=\SS^1_{u=0}$. We identify $\SS^1=\R/ 2\pi \Z $ and will call an element $\vt\in\SS^1$ an \angle.

The description of these data by combinatorial means will highly depend on the various pole orders of the factors $\vi\in\Exp$ and their differences $\vi-\psi $ as well. The situation we have in mind as an application allows to get rid of many of these difficulties, since we will be able to restrict to a single pole order:

\begin{definition}
We say that a finite subset $E \subset u^{-1} \C[u^{-1}] $ is of pure level~$q$~if
\begin{itemize}
\item
the pole order of each $\vi\in E$ as well as
\item
the pole order of each difference $\vi-\psi $ for $\vi,\psi\in E$, $\vi\neq \psi $,
\end{itemize}
equals $q$. We say that an unramified Stokes structure $(\cL, \cL_\bbullet)$ is pure of level $q$ if the associated set $\Exp$ of exponential factors is pure of level $q$.
\end{definition}

Let $E \subset u^{-1} \C[u^{-1}] $ be a finite subset of pure level $q$. In particular, there are exactly~$2q$ Stokes directions for any pair $\vi\neq \psi $ in $E$, where a Stokes direction is an \angle $\vt\in\SS^1$ at which the two factors $\vi, \psi $ do not satisfy one of the two inequalities $\vi \le_\vt \psi $ or $\psi \le_\vt \vi $. More precisely, if
$$
(\psi-\vi)(u)=u^{-q} \cdot g(u)
$$
with a polynomial $g(u)$ with $g(0)\neq 0$, the Stokes directions are
$$
\tSt(\vi,\psi) :=\Big\{\vt\in\SS^1 \mid -q \vt+\arg(g(0)) \equiv \frac{\pi}{2} \text{ or } \frac{3\pi}{2} \bmod 2\pi\Big\}.
$$
At these directions, the asymptotic behaviour of $\exp(\vi-\psi)$ changes from rapid decay to rapid growth and conversely.

Let $\StDir(E)$ denote the set of all Stokes directions $\StDir(E) :=\bigcup_{\vi\neq \psi\in E} \tSt(\vi,\psi)$.
An \angle $\vt_o\in\SS^1$ is said to be \emph{generic with respect to $E$} if
\begin{equation}\label{eq:genericwrtE}
\Big\{\vt_o+\frac{\ell\pi}{q} \mid \ell\in\Z\Big\}\cap\StDir(E)=\emptyset.
\end{equation}
Let us fix a generic $\vt_o$ and let us set $\vt_\ell:=\vt_o+\sfrac{\ell\pi}{q}$ ($\ell=0,\dots,2q-1$). Then each of the open intervals $(\vt_\ell, \vt_{\ell+1})$ contains exactly one Stokes direction for each pair $\vi\neq \psi $ in $E$. For $\ve >0$ small enough, the same holds for the intervals
$I_\ell :=(\vt_\ell-\ve, \vt_{\ell+1}+\ve)$ with which we cover the circle $\SS^1=\bigcup_{\ell=0}^{2q-1} I_\ell$.

The category of unramified linear Stokes data of pure level $q$ indexed by the finite set $\{1,\dots,r\}$ over the field $\CC$ is defined by \eqref{eq:linearStokesData}. A morphism of two linear Stokes data of this type is given by $\CC$\nobreakdash-linear maps between the vector spaces which commute with all the $S_\ell^{\ell+1}$ and respect the filtrations. We denote this category by $\Stdat(q,r)$. A similar definition can be given over any field $\fldk$.

\medskip
We now fix a finite set $E \subset u^{-1} \C[u^{-1}] $ of pure level $q$. Let $\Ststr(E)$ be the category of unramified Stokes structures $(\cL,\cL_\bbullet)$ on $\SS^1$ with $\Exp=E$. After choosing an appropriate \angle $\vt_o$ as above, we construct a functor
\begin{equation}\label{eq:Ststrdat}
\Phi_{E,\vt_o}: \Ststr(E)\to\Stdat(q,\#E)
\end{equation}
as follows.

We can arrange the elements in $E$ according to the ordering \eqref{eq:leqvt} with respect to the generic direction $\vt_o$, \ie we write $E=\{\vi_1 <_{\vt_o} \vi_2<_{\vt_o} \vi_3 <_{\vt_o}\dots <_{\vt_o} \vi_{\#E}\}$. We then have the same ordering at any $\vt_{2 \mu}$ with even index and the reverse order for $\vt_{2 \mu +1}$. This follows from the fact that each inequality $\vi_i <_\vt \vi_j $ becomes reversed whenever the \angle $\vt $ passes a Stokes direction for the pair $\vi_i $, $\vi_j $, and there is exactly one Stokes direction for each pair inside $(\vt_\ell, \vt_{\ell+1})$.

We invoke the fundamental result going back to Balser, Jurkat and Lutz:
\begin{proposition}\label{prop:BJL}
For each open interval $I \subsetneq \SS^1$ of width $\sfrac{\pi}{q}+2\ve $ for $\ve>0$ small enough, there is a unique splitting
$$
\tau: \cL|_I \cong\bigoplus_{\vi\in E} \gr_\vi \cL|_I
$$
compatible with the filtrations.
\end{proposition}
\begin{proof}
See \cite[Lem.\,5.1]{Malgrange83bb} or \cite[Th.\,A]{B-J-L79}.
\end{proof}

In particular, writing $\gr\cL :=\bigoplus_{\vi\in E} \gr_\vi \cL $ for the associated graded sheaf endowed with the Stokes filtration naturally induced by using the order \eqref{eq:leqvt}, we have unique filtered isomorphisms
\begin{equation}\label{eq:trivIk}
\tau_\ell: (\cL, \cL_\bbullet)|_{I_\ell} \isom (\gr\cL, (\gr\cL)_\bbullet)|_{I_\ell}
\end{equation}
over the intervals $I_\ell$ chosen above.

For each $\ell=0,\dots, 2q-1$, we let $L^{(\ell)}:=\Gamma(I_\ell, \cL)$ be the sections of the local system $\cL$ over $I_\ell$. We have canonical isomorphisms
$$
\fra_\ell: L^{(\ell)} \isom \cL_{\vt_\ell} \quad\text{and}\quad
\frb_\ell: L^{(\ell)} \isom \cL_{\vt_{\ell+1}}
$$
between $L^{(\ell)}$ and the corresponding stalks of $\cL$. Writing $\bLMH_\ell:=\cL_{\vt_\ell}$ for the stalk at~$\vt_\ell$, we define
$$
S_\ell^{\ell+1}:=\frb_\ell^{-1} \circ \fra_\ell: \bLMH_\ell \isom \bLMH_{\ell+1}.
$$
We call this isomorphism the clockwise analytic continuation. Additionally, we define, for $\mu =0,\dots,q-1$,
\[
\begin{cases}
F_j\bLMH_{2 \mu} :=(\cL_{\le \vi_j})_{\vt_{2 \mu}} \subset \bLMH_{2 \mu},\\
F^j\bLMH_{2 \mu +1} :=(\cL_{\le \vi_j})_{\vt_{2 \mu +1}} \subset \bLMH_{2 \mu +1}.
\end{cases}
\]
Then the data
$$
\bigl((\bLMH_\ell)_\ell, (S_\ell^{\ell+1})_\ell, (F\bLMH_\ell)_\ell\bigr)
$$
define a set of linear Stokes data as in \eqref{eq:linearStokesData}.

\begin{proposition}
Given a finite set $E$ of pure level $q$ and a generic $\vt_o\in\SS^1$, the construction above defines an equivalence of categories
$$
\Phi_{E, \vt_o}: \Ststr(E)\to\Stdat(q,\#E)
$$
between the Stokes structures with exponential factors $E$ and the Stokes data of pure level $q$ indexed by $\{1,\dots,\#E\}$.
\end{proposition}
\begin{proof}
In the case $q=1$, a similar statement is proved in \cite{H-S09} or \cite{Bibi13}. The same proof holds in our situation as well.
\end{proof}

\subsection{Reminder on the local Riemann-Hilbert correspondence}
Let $\ccM$ be a finite dimensional $\CC\lpb t\rpb$-vector space with connection, and let $\rho:u\mto u^p=t$ be a ramification such that $\rho^+\ccM$ has a formal Levelt-Turrittin decomposition
\[
\CC\lpr u\rpr\otimes_{\CC\lpb u\rpb}\rho^+\ccM\simeq\bigoplus_{\vi\in\Exp(\ccM)}(\ccE^{-\vi(u)}\otimes\ccR_\vi),
\]
with $\Exp:=\Exp(\ccM)\subset u^{-1}\CC[u^{-1}]$ (we use an opposite sign $-\vi(u)$ with respect to the usual notation as it will be more convenient for the Stokes filtration). Moreover, by the uniqueness of the Levelt-Turrittin decomposition, we have, for each $\zeta\in\mu_p$ ($p$-th root of the unity)
\begin{equation}\label{eq:equivariant}
\ccR_{\vi_\zeta}=\ccR_\vi,\quad\text{with }\vi_\zeta(\eta):=\vi(\zeta\eta).
\end{equation}
It may happen that $\vi_\zeta=\vi$ for some $\vi\in\Exp$ and $\zeta\in\mu_p$.

The Riemann-Hilbert correspondence $\ccM\mto(\cL,\cL_\bbullet)$ goes as follows (see \cite{Deligne77}, \cite{B-V89}, \cite{Malgrange91}, \cite[Chap.\,5]{Bibi10}). Let us denote by $\cF$ the local system attached to $\ccM$ on a punctured neighbourhood of $\{t=0\}$, that we regard as well as a local system on the open annulus $(\Afu_t)^*$ or on the semi-closed annulus $\wt{\Afu_t}$. Similarly, extending~$\rho$ as a covering map $(\Afu_u)^*\to(\Afu_t)^*$ or $\wt{\Afu_u}\to\wt{\Afu_t}$, we define $\rho^{-1}\cF$. We also use the notation
\[
\cL:=\cF_{|\SS^1_{t=0}},\quad \wt\rho^{-1}\cL:=\cF_{|\SS^1_{u=0}}.
\]
One can attach to $\ccM$ the family of subsheaves $\cL_{\leq\vi}:=\cH^0\DR^{\rmod0}(\rho^+\ccM\otimes\nobreak\ccE^\vi)_{|\SS^1_{u=0}}$ of $\wt\rho^{-1}\cL$ ($\vi\in\Exp$), where $\DR^{\rmod0}$ denotes the de~Rham complex with coefficients in the sheaf of holomorphic functions on $\wt{\Afu_u}\moins \SS^1_{u=0}$ having moderate growth along $\SS^1_{u=0}$. The equivariance property \ref{def:stokesfiltlocsyst}\eqref{def:stokesfiltlocsyst6} is obtained through the isomorphisms $\zeta^+\rho^+\ccM=\rho^+\ccM$. Similarly, we can consider the family of subsheaves $\cL_{<\vi}:=\cH^0\DR^{\rdc0}(\rho^+\ccM\otimes\ccE^\vi)$ of $\wt\rho^{-1}\cL$ ($\vi\in\Exp$), where $\DR^{\rdc0}$ denotes the de~Rham complex with coefficients in the sheaf of holomorphic functions on $\wt{\Afu_u}\moins \SS^1_{u=0}$ having rapid decay along $\SS^1_{u=0}$.

\begin{remarque}\label{rem:Stokessheaf}
The equivalence of categories $\ccM\longleftrightarrow M$ considered in \S\ref{subsec:local} reads as follows, by introducing the notion of Stokes-filtered sheaves. A ramified Stokes-filtered sheaf on~$\wt\Afut$ is defined (in~the present setting) as the pair formed by a local system~$\cF$ on~$\wt\Afut$ and a Stokes-filtration~$\cL_\bbullet$ of its restriction $\cL:=\cF_{|\SS^1_{t=0}}$. Since $\SS^1_{t=0}$ is a deformation retract of $\wt\Afut$, giving $\cF$ is equivalent to giving $\cL$, hence the equivalence.

With this notion of Stokes-filtered sheaf, we can define a subsheaf $\cF_{\leq0}$ of $\cF$, which coincides with $\cF$ away from $\SS^1_{t=0}$, and whose restriction to $\SS^1_{t=0}$ is $\cL_{\leq0}$ (see Remark \ref{rem:extensionindex}). In the following, we will also consider the real blow-up space~$\wt\PP^1_t$ of~$\PP^1_t$ at $t=0$ and $t=\infty$, which is homeomorphic to a closed annulus, and we will still denote by~$\cF$ (\resp$\cF_{\leq0}$) the push-forward of $\cF$ (\resp$\cF_{\leq0}$) by the open inclusion $\wt\Afut\hto\wt\PP^1_t$.
\end{remarque}

\section{The topological Laplace transformation \texorpdfstring{$\ccF^{(0,\infty)}_\top$}{Ftop}}\label{sec:topLaplace}
\subsection{Reminder on the local Laplace transformation \texorpdfstring{$\ccF^{(0,\infty)}$}{F}}\label{subsec:reminderlocalLaplace}
We keep the notation as in the introduction, and we will focus on the coordinate $\tau$, so that the Laplace kernel is now $\exp(-t/\tau)$. Moreover, since we only want to deal with $\ccF^{(0,\infty)}(\ccM)$, we only consider the localized Laplace transform $\CC[\tau',\tau^{\prime-1}]\otimes_{\CC[\tau']}\Fou M$, that we will denote by $\Fou M$ from now on. We consider the following diagram:
\[
\xymatrix@=5mm{
&\Afut\times\GG_{\rmm,\tau}\ar[dl]_\pi\ar[dr]^{\wh \pi}&\\
\Afut&&\GG_{\rmm,\tau}
}
\]
The localized Laplace transform $\Fou M$ is defined by the formula
\[
\Fou M=\wh\pi_+(\pi^+M\otimes E^{-t/\tau}).
\]
We can regard the $\CC[t,\tau,\tau^{-1}]\langle\partial_t,\partial_\tau\rangle$-module $\pi^+M\otimes E^{-t/\tau}$ both as a holonomic $\cD_{\PP^1_t\times\Afu_\tau}$-module and as a meromorphic bundle with connection on $\PP^1_t\times\Afu_\tau$ with poles along the divisor $D:=D_0\cup D_\infty\cup D_{\wh\infty}$ in $\PP^1_t\times\nobreak\Afu_\tau$ (with $D_0=\{0\}\times\Afu_\tau$, $D_\infty=\{\infty\}\times\Afu_\tau$ and $D_{\wh\infty}=\PP^1_t\times\{\wh\infty\}$). It has irregular singularities along $D$. As indicated in \S\ref{subsec:local}, the germ of $\Fou M$ at $\tau=0$ only depends on the germ $\ccM$ of $M$ at $t=0$, and it is denoted by $\ccF^{(0,\infty)}(\ccM)$.

Let $\wh\rho:\eta\mto \eta^{\wh p}=\tau$ be a ramification such that $\wh\rho^+\Fou M$ is non-ramified at infinity, that is,
\[
\CC\lpr\eta\rpr\otimes_{\CC\lpb\eta\rpb}\wh\rho^+\ccF^{(0,\infty)}(\ccM)\simeq\bigoplus_{\wh\psi\in\wh\Exp(\Fou M)}(\ccE^{-\wh\psi(\eta)}\otimes\wh\ccR_{\wh\psi}),
\]
with $\wh\Exp(\Fou M)\subset \eta^{-1}\CC[\eta^{-1}]$. Then $\wh\rho^+\ccF^{(0,\infty)}(\ccM)$ is obtained through the diagram
\[
\xymatrix@=5mm{
&\Afut\times\GG_{\rmm,\eta}\ar[dl]_\pi\ar[dr]^{\wh\pi}&\\
\Afut&&\GG_{\rmm,\eta}
}
\]
by the formula
\[
\wh\rho^+\ccF^{(0,\infty)}(\ccM)=\CC\lpb\eta\rpb\otimes_{\CC[\eta,\eta^{-1}]}\wh\pi_+(\pi^+M\otimes E^{-t/\wh\rho(\eta)}).
\]
Moreover, by the uniqueness of the Levelt-Turrittin decomposition, we have, for each $\wzeta\in\mu_{\wh p}$ ($\wh p$-th root of the unity)
\[
\wh\ccR_{\wh\psi_{\wzeta}}=\wh\ccR_{\wh\psi},\quad\text{with }\wh\psi_{\wzeta}(\eta):=\wh\psi(\wzeta^{-1}\eta).
\]
The stationary phase formula of \cite{Fang07,Bibi07a} makes explicit the correspondence $(p,\Exp,(\ccR_\vi)_{\vi\in\Exp})\mto(\wh p,\wh\Exp,(\wh\ccR_{\wvi})_{\wvi\in\wh\Exp})$.

\subsection{The monodromy of \texorpdfstring{$\ccF^{(0,\infty)}\ccM$}{FlocM}}
Let us denote by $\FcL$ the local system attached to $\ccF^{(0,\infty)}\ccM$ on~$\SS^1_{\tau=0}$. The local system $\FcL$ can equivalently be regarded as a local system $\FcF$ on~$\GG_{\rmm,\tau}^\an$. Let us denote by $T_\top$ the topological monodromy of $\ccM$ around the origin, and by $\wh T_\top$ the topological monodromy of $\ccF^{(0,\infty)}\ccM$ around $\tau=0$ (with the notation as in \S\ref{subsec:local}). Since $\Fou M$ has no singularity except at $0,\infty$, $\wh T_\top^{-1}$ is the topological monodromy of $\Fou M$ around $\tau'=0$. Therefore, there are various ways of calculating the topological monodromy $\wh T_\top$:
\begin{itemize}
\item
either by comparing it with the monodromy at $t=\infty$ of $M$, which is nothing, up to changing orientation, but the topological monodromy of $\ccM$ at $t=0$, since $M$ has no singular point on $(\Afuan_t)^*$, and its singular point at $t=\infty$ is regular; this will be done in Proposition \ref{prop:topmonodromyF};
\item
or by a direct topological computation at $\tau=\infty$; this is done in \S\ref{subsec:topmono};
\item
lastly, by obtaining it from the Stokes data at $\tau=\infty$, once we have computed them; this is however not the most economical way.
\end{itemize}

Given an automorphism~$T$, we will denote by $T_\lambda$ ($\lambda\in\CC^*$) the component of $T$ with eigenvalue $\lambda$. Since $\wh M$ has a regular singularity at $\tau=\infty$, with monodromy equal to $T_\top^{-1}$, classical results give:

\begin{proposition}\label{prop:topmonodromyF}
For $\lambda\neq1$, we have $\wh T_{\top,\lambda}=T_{\top,\lambda}$. Moreover, if $\ccM$ is a successive extension of germs of rank-one meromorphic connections, the Jordan blocks of size $k\geq2$ of $\wh T_{\top,1}$ are in one-to-one correspondence with the Jordan blocks of size $k-1$ of $T_{\top,1}$.
\end{proposition}

\begin{proof}[Sketch of proof]
Only for this proof, we denote by $\Fou M$ the non localized Laplace transform of $M$. It has a regular singularity at $\tau'=0$. The monodromy $T_\top^{-1}$ of $M$ at infinity is known to be equal to the monodromy of the vanishing cycles of $\Fou M$ at $\tau'=0$, while $\wh T_\top^{-1}$ is the monodromy of the nearby cycles of $\Fou M$ at $\tau'=0$ (\cf\eg\cite[Prop.\,4.1(iv)]{Bibi05b} with $z=1$). The first assertion follows. For the second assertion, note that, if $\ccM$ is a rank-one meromorphic connection, then~$M$ is irreducible as a $\Clt$-module, and thus $\Fou M$ is also irreducible, so is a minimal extension at $\tau'=0$, which implies the second assertion for $M$. In general, note that in a exact sequence $0\to\Fou M'\to\Fou M\to\Fou M''\to0$, if the extreme terms are minimal extensions at $\tau'=0$, then so is the middle term. This concludes the proof of the second assertion.
\end{proof}

As in Remark \ref{rem:Stokessheaf}, let us consider the real blowing-up map $\varpi:\wt\PP^1_t\to\PP^1_t$ of $\PP^1_t$ at $t=0$ and $t=\infty$ (so that $\wt\PP^1_t$ is a closed annulus) and let us denote by $\wtpi:\wt\PP^1_t\times\GG_{\rmm,\tau}^\an\to\GG_{\rmm,\tau}^\an$ the projection. Then we have
\begin{equation}\label{eq:FcF}
\FcF\simeq R^1\wtpi_*\DR^{\rmod(D_0\cup D_\infty)}(\pi^+M\otimes E^{-t/\tau})
\end{equation}
and all other $R^k\wtpi_*$ vanish. We note that, for $\tau\neq0$, twisting with $E^{-t/\tau}$ has no effect near $t=0$. Similarly, since $M$ is regular at $\infty$, $E^{-t/\tau}$ is the only important factor along $t=\infty$.

\subsubsection*{Computation}
Let us use the notation of Remark \ref{rem:Stokessheaf}. Let us fix $\tau_o\in\CC^*$ and let us compute the fibre $\FcF_{\tau_o}$. Recall that $\wt\PP^1_t$ is a closed annulus with boundary components $\SS^1_{t=0}$ and $\SS^1_{t=\infty}$. Let $\cF_{\leq0}$ be the sheaf $\cH^0\DR^{\rmod0}M$ on the semi-closed annulus $\wt\PP^1_t\moins\SS^1_{t=\infty}$. We denote by $(\wt\PP^1_{t})_{\leq_{\tau_o}\infty}$ the closed annulus $\wt\PP^1_t$ with the closed interval $\arg t-\arg\tau_o\in[-\pi/2,\pi/2]\bmod2\pi$ in~$\SS^1_{t=\infty}$ deleted. This is the domain where the function $\exp(t/\tau_o)$ has rapid decay. We consider the inclusions
\[
\wt\PP^1_t\moins\SS^1_{t=\infty}\Hto{\alpha^\infty_{\tau_o}}(\wt\PP^1_{t})_{\leq_{\tau_o}\infty}\Hto{\beta^\infty_{\tau_o}}\wt\PP^1_t.
\]
We then have (similarly to \cite[\S7.3]{Bibi10}):

\begin{proposition}
The fibre $\FcF_{\tau_o}$ is equal to $H^1(\wt\PP^1_t,\beta^\infty_{\tau_o,!}\alpha^\infty_{\tau_o,*}\cF_{\leq0})$ and the other $H^k$ vanish. It is isomorphic to the vanishing cycle space at $t=0$ of the perverse sheaf $\varpi_*\cF_{\leq0}$ on $\Afu_t$. The monodromy of $\FcF$ with respect to $\tau$ is obtained by rotating the interval $(\pi/2,3\pi/2)+\arg\tau_o\subset \SS^1_{t=\infty}$ in the counterclockwise direction with respect to~$\tau_o$.\qed
\end{proposition}

\begin{remarque}
We can similarly compute the fibre of $\wt{\wh\rho}{}^{-1}\FcF$ at any $\eta_o$ by replacing $\arg\tau_o$ with $\wh p\arg\eta_o$ in the formula above.
\end{remarque}

\subsection{The Stokes structure of \texorpdfstring{$\ccF^{(0,\infty)}\ccM$}{FlocM2}}\label{subsec:stokesFM}
Let us denote by $\FcL_\bbullet$ the Stokes filtration of $\FcL$, regarded as a family of subsheaves of $\wt{\wh\rho}{}^{-1}(\FcL)$.
Due to the ramification $\wh\rho$, we now work with the variable $\eta$. Let $\wt{\PP^1_t\times\Afu_\eta}$ be the real oriented blowing-up of $\PP^1_t\times\Afu_\eta$ along the components $D_0, D_\infty, D_{\wh\infty}$ of $D$ regarded now in $\PP^1_t\times\Afu_\eta$. We have a similar diagram (see \eg \cite[Chap.\,8]{Bibi10}):
\[
\xymatrix@=5mm{
&\wt{\PP^1_t\times\Afu_\eta}\ar[dl]_{\wpi}\ar[dr]^{\wtpi}&\\
\wt\PP^1_t&&\wt\Afu_\eta&\ar@{_{ (}->}[l]_-{\whj}\GG_{\rmm,\eta}
}
\]
We note that $\wt{\PP^1_t\times\Afu_\eta}$ is the product $\wt\PP^1_t\times\wt\Afu_\eta$ of the closed annulus $\wt\PP^1_t$ by the semi-closed annulus $\wt\Afu_\eta$. One defines on such a real blown-up space the sheaf of holomorphic functions with rapid decay near the boundary (\ie the inverse image by the real blow-up map of the divisor $D$). There is a corresponding de~Rham complex that we denote by $\DR^\modD$ (see \loccit).

\begin{theoreme}[Mochizuki, {\cite[Cor\ptbl4.7.5]{Mochizuki10}}]\label{th:pushforward}
For each $\wh\psi\in\eta^{-1}\CC[\eta^{-1}]$ the natural morphism
\enlargethispage{\baselineskip}%
\[
\DR^{\rmod\wh\infty}\bigl(\wh\rho^+\ccF^{(0,\infty)}(\ccM)\otimes\ccE^{\wh\psi(\eta)}\bigr)\to\bR\wtpi_*\DR^\modD(\pi^+M\otimes E^{\wh\psi(\eta)-t/\wh\rho(\eta)})[1]
\]
is a quasi-isomorphism and the natural morphism
\[
R^1\wtpi_*\DR^\modD(\pi^+M\otimes E^{\wh\psi(\eta)-t/\wh\rho(\eta)})\to\whj_*R^1\wtpi_*\DR^{\rmod(D_0\cup D_\infty)}(\pi^+M\otimes E^{-t/\wh\rho(\eta)})
\]
is injective.
\end{theoreme}

\begin{remarques}\mbox{}
\begin{enumerate}
\item
The result of \cite[Cor\ptbl4.7.5]{Mochizuki10} refers to the real blow-up $\PP^1_t\times\wt\Afu_\eta$ along~$D_{\wh\infty}$, not along all the components of $D$. The above statement is obtained by using that, with respect to $\varpi:\wt{\PP^1_t\times\Afu_\eta}\to\PP^1_t\times\wt\Afu_\eta$, we have
\[
\bR\varpi_*\DR^{\modD}(\pi^+M\otimes E^{-t/\wh\rho(\eta)})=\DR^{\modD_{\wh\infty}}(\pi^+M\otimes E^{-t/\wh\rho(\eta)}).
\]
This follows from the identification $\bR\varpi_*\cA^{\modD}_{\wt{\PP^1_t\times\Afu_\eta}}=\cA^{\modD_{\wh\infty}}_{\PP^1_t\times\wt\Afu_\eta}(*D_0\cup D_\infty)$, as in \cite[Prop.\,8.9]{Bibi10}.

\item
Set $\wh\rho(\eta)=c\eta^{\wh p}(1+\ro(\eta))$ and let us choose a $\wh p$-th root $c^{1/\wh p}(1+\ro(\eta))$ of $\wh\rho(\eta)/\eta^{\wh p}$. Then the $\wh p$-th roots of $\wh\rho(\eta)$ are written $\wh\rho^{1/\wh p}_{\wzeta}(\eta)=\wzeta\eta\cdot c^{1/\wh p}(1+\ro(\eta))$ ($\wzeta\in\mu_{\wh p}$). Now, for $\wh\psi\in\eta^{-1}\CC[\eta^{-1}]$, we set $\wh\psi_{\wzeta}(\eta):=\wh\psi\bigl(\wh\rho^{1/\wh p}_{\wzeta}(\eta)\bigr)$.

The $\mu_{\wh p}$-equivariance is then induced by the isomorphism
\[
\wzeta^+E^{\wh\psi(\eta)-t/\wh\rho(\eta)}=E^{\wh\psi_{\wzeta}(\eta)-t/\wh\rho(\eta)},
\]
where we still denote by $\wzeta$ the map $\eta\mto\wh\rho^{1/\wh p}_{\wzeta}(\eta)$.
\end{enumerate}
\end{remarques}

The theorem expresses therefore the Stokes filtration as the push-forward by $\wtpi$ of a family of complexes indexed by $\wh\Exp$. The questions stated in the introduction reduce now to
\begin{itemize}
\item
expressing, for any $\wh\psi\in\wh\Exp$, the complex $\DR^{\modD}(\pi^+M\otimes E^{\wh\psi(\eta)-t/\wh\rho(\eta)})$, together with the isomorphisms
\[
\wzeta^{-1}\DR^{\modD}(\pi^+M\otimes E^{\wh\psi(\eta)-t/\wh\rho(\eta)})\simeq\DR^{\modD}(\pi^+M\otimes E^{\wh\psi_{\wzeta}(\eta)-t/\wh\rho(\eta)}),
\]
in terms of the Stokes-filtered local system attached to $M$,
\item
computing in a topological way the push-forward by $\wtpi$ once this complex is well understood.
\end{itemize}

While $\cH^0\DR^{\modD}$ is easy to compute from the data $(\cF,\cL_\bbullet)$, it happens in general that the complex $\DR^{\modD}$ has higher cohomology, which is not easy to express in terms of $(\cF,\cL_\bbullet)$. The main reason is that the meromorphic connection $\pi^+M\otimes E^{\wh\psi(\eta)-t/\wh\rho(\eta)}$ may not be good (in the sense of \cite{Bibi97}) at $(0,\wh\infty)$, due to indeterminacy $0/0$ of the functions
\begin{equation}\label{eq:kappa}
\omega_{\vi,\wh\psi}(u,\eta):=\wh\psi(\eta)-\rho(u)/\wh\rho(\eta)-\vi(u)
\end{equation}
for $u\to0$ and $\eta\to0$ ($\vi\in\Exp$, $\wh\psi\in\wh\Exp$).

\begin{proposition}\label{prop:good}
There exists a projective modification $e:Z\to\PP^1_t\times\Afu_\eta$ which consists of a succession of point blowing-ups above $(0,\wh\infty)$, such that for each $\wh\psi\in\wh\Exp$, the pull-back $\pi^+M\otimes E^{\wh\psi(\eta)-t/\wh\rho(\eta)}$ is good along the normal crossing divisor $D_Z:=e^{-1}(D)$.\qed
\end{proposition}

This result is a particular case of a general result due to Kedlaya \cite{Kedlaya09} and Mochizuki \cite{Mochizuki07b} (see also \cite{Mochizuki09}), but can be proved in a much easier way in the present setting, since it essentially reduces to resolving the indeterminacy $0/0$ of the rational functions $\omega_{\vi,\wh\psi}(u,\eta)$ for $\vi\in\Exp(M)$ and $\wh\psi\in\Exp(\wh M)$, and to using \cite[Lem.\,III.1.3.3]{Bibi97}. A particular case will be made precise in \S\ref{subsec:simplifassumpt}. The results below do not depend on the choice of $Z$ satisfying the conclusion of Theorem \ref{th:Majima}. We decompose $D_Z$ as the union of the strict transforms of $D_0,D_\infty,D_{\wh\infty}$ that we denote by the same letters, and the exceptional divisor $E$.

The main tool is then the higher dimensional Hukuhara-Turrittin theorem (originally due to Majima \cite{Majima84}, see also \cite[Proof of Th.\,7.2]{Bibi93}, \cite[Chap.\,20]{Mochizuki08}, \cite[Th.\,12.5\,\&\,Cor.\,12.7]{Bibi10}). We consider the real blow-up space $\wt Z$ of $Z$ along the components of $D_Z$ and the natural lift $\wt e:\wt Z\to\wt{\PP^1_t\times\Afu_\eta}$ of $e:Z\to\PP^1_t\times\Afu_\eta$.

\begin{theoreme}\label{th:Majima}
If $Z$ is as in Proposition \ref{prop:good}, then $\DR^{\modD_Z}e^+(\pi^+M\otimes E^{\wh\psi(\eta)-t/\wh\rho(\eta)})$ has cohomology in degree zero only.\qed
\end{theoreme}

We can then use a particular case \cite[Prop.\,8.9]{Bibi10} of \cite[Cor\ptbl4.7.5]{Mochizuki10}, which is easier to prove because $e$ is a projective modification, to get:
\[
\bR\wt e_*\DR^{\modD_Z}e^+(\pi^+M\otimes E^{\wh\psi(\eta)-t/\wh\rho(\eta)})=\DR^\modD(\pi^+M\otimes E^{\wh\psi(\eta)-t/\wh\rho(\eta)}).
\]

We will apply Theorem \ref{th:pushforward} through its corollary below.

\begin{corollaire}[of Theorem \ref{th:pushforward}]\label{cor:pushforward}
Let $Z$ be as in Proposition \ref{prop:good}. Then, for each $\wh\psi\in\eta^{-1}\CC[\eta^{-1}]$ the natural morphism
\[
\DR^{\rmod\wh\infty}\bigl(\wh\rho^+\ccF^{(0,\infty)}(M)\otimes\ccE^{\wh\psi(\eta)}\bigr)\to R^1(\wtpi\circ\wt e)_*\cH^0\DR^{\modD_Z}e^+(\pi^+M\otimes E^{\wh\psi(\eta)-t/\wh\rho(\eta)})
\]
is a quasi-isomorphism and the natural morphism
\[
R^1(\wtpi\circ\wt e)_*\cH^0\DR^{\modD_Z}e^+(\pi^+M\otimes E^{\wh\psi(\eta)-t/\wh\rho(\eta)})
\to\wt{\wh\rho}{}^{-1}\FcF
\]
is injective.\qed
\end{corollaire}

The pull-back $(\wh\pi\circ e)^{-1}(\{\eta=0\})=e^{-1}(D_{\wh\infty})$ is the union $D_{\wh\infty}\cup E$. Its pull-back in~$\wt Z$, denoted by $\wD_{\wh\infty}\cup\wE$, is equal to $(\wtpi\circ\wt e)^{-1}(\SS^1_{\eta=0})$. We will denote by $\wthj$ the open inclusion $\wt Z\moins(\wD_{\wh\infty}\cup\wE)\hto\wt Z$.

Let $\rho:u \mapsto u^p=t $ be a ramification such that $\rho^+ \cM $ is non-ramified at $t=0 $. Consider the fibre product $\wt Z':= (\wt{\PP^1_u\times\Afu_\eta}) \times_{\wt{\PP^1_t\times\Afu_\eta}} \wt Z $ as a topological space:
\[
\xymatrix@C=15mm{
\wt{\PP^1_u\times\Afu_\eta}\ar[d]_{\wt{\rho\times 1}} & {\wt Z'}\ar[l]_-{\wt e\,'} \ar[d]^{\wt \sigma} \\
\wt{\PP^1_u\times\Afu_\eta} & \wt Z \ar[l]_-{\wt e}\\
}
\]
Let us write $\wt D_Z' := \wt \sigma^{-1}(\wt D_Z) $ and $\wthj \, {}':\wt Z' \smallsetminus \sigma^{-1}(\wD_{\wh \infty} \cup \wE) \hookrightarrow \wt Z' $. Note that the restriction of $\wt e\,' $ to $\wt Z' \moins \wt D_Z' $ is a homeomorphism.

Let $\varpi: \wt Z \to Z $ be the oriented real blow-up map. Below we use the convention of Remark \ref{rem:Stokessheaf}, and for $\wt z' \in\wt Z'$ we set
\[
\wt u=(\wt \pi'\circ\wt e\,')(\wt z\,')\in\wt\PP^1_u,\quad
\Exp_{\wt u}=
\begin{cases}
\Exp&\text{if }\wt u \in\SS^1_{u=0},\\
0&\text{if }\wt u\in\PP^1_u\moins \SS^1_{u=0}.
\end{cases}
\]
and, as usual, for a point $\wt z\,' \in \wt Z' $ over $z:= \varpi \circ \wt \sigma(\wt z\,') \in Z $ and a function $g(u,\eta) $:
\[
\begin{split}
e'^*\vi\leq_{\wt z'} e'^*g &\iff\reel\bigl(e'^*\vi-e'^*g\bigr)<0\text{ in some neighb.\ of }\wt z'\\
&\hspace*{2.2cm}\text{or }e'^*\vi-e'^*g \text{ is bounded in some neighb.\ of $(\varpi \circ \wt \sigma)^{-1}(z)$}.
\end{split}
\]
Recall that the Stokes filtration $(\cF', \cF'_{\le \wh{\psi}}) $ on the local system $\cF' := \wt{\rho}^{-1} \cF $ satsifies the equivariance condition $\zeta^{-1} \cF'_{\le \vi} = \cF'_{\le \vi_\zeta} $ inside $\zeta^{-1} \cF'=\cF' $ for each $\zeta \in \mu_p $.

\pagebreak[2]
\begin{definition}\mbox{}
\begin{enumerate}
\item
The sheaf $\cG'$ on $\wt Z'$ is defined as $\wthj\,{}'_*\cG^{\prime\circ}$, where $\cG^{\prime\circ}$ is the subsheaf of $(\wt{\pi'} \circ \wt{e'})^{-1} \cF' $ on $\wt Z'\moins(\wt\sigma^{-1}(\wD_{\wh\infty}\cup\wE))$ whose germ at $\wt z' \in\wt Z'\moins(\wt\sigma^{-1}(\wD_{\wh\infty}\cup\wE))$ is defined by
\[
\cG^{\prime\circ}_{\wt z'}:=\sum_{\substack{\vi\in\Exp_{\wt u}\\ e'^*\vi\leq_{\wt z'} e'^*(-\rho(u)/\wh\rho(\eta))}}\hspace*{-8mm}\cF'_{\leq\vi,\wt u}\subset\cF'_{\wt u}.
\]

\item
The sheaf $\cG'_{\leq\wh\psi}$ is the subsheaf of $\cG'$ which coincides with the latter away from~$\wt D_{Z'}$ and whose germ $\cG'_{\leq\wh\psi,\wt z'}$ at each $\wt z'\in\wt Z'$ is defined by the formula (same convention as in \eqref{eq:extindexx})
\[
\cG'_{\leq\wh\psi,\wt z'}:=\sum_{\substack{\vi\in\Exp_{\wt u}\\ e'^*\vi\leq_{\wt z'} e'^*(\wh\psi- \rho(u)/\wh\rho(\eta))}}\hspace*{-10mm}\cF'_{\leq\vi,\wt u}\subset\cF'_{\wt u}.
\]
\end{enumerate}
\end{definition}

(We note that the inclusion $\cG'_{\leq\wh\psi}\subset\cG'$ is obvious away from $\wt\sigma^{-1}(\wD_{\wh\infty}\cup\wE)$ since $\psi$ plays no role there, and therefore it is also clear on this set, by the definition of $\cG'$ there.)

\begin{lemdef}[Sheaves $\cG$ and $\cG_{\leq\wh\psi}$]\label{def:Gpsihat}\mbox{}
The sheaves $\cG' $ and $\cG'_{\leq\wh\psi} $ descend with respect to $\wt\sigma $ to subsheaves
$\cG $ and $\cG_{\leq\wh\psi} $ of $(\wt\pi \circ \wt e)^{-1} \cF $ on $\wt Z $.
\end{lemdef}

\begin{proof}
Since $\rho(u)=u^p $, the group $\mu_p $ acts transitively on the fibres of $\wt \sigma $. This action is free on the fibres over $\wt z \in \wt D_Z $. By definition, for any $\zeta\in\mu_p$,
\[
e'^\ast \vi \leq_{\zeta \wt z'} e'^\ast ( \wh\psi(\eta) -\rho(u)/\wh\rho(\eta)) \Longleftrightarrow
e'^\ast \vi_\zeta \leq_{\wt z'} e'^\ast ( \wh\psi(\eta) -\rho(u)/\wh\rho(\eta))
\]
since the right-hand side is invariant with respect to the $\mu_p $-action. By the equivariance property of $\cF'_{\leq} $, we deduce the equality of subspaces of $\cF_{\wt t}$:
\[
\cG'_{\leq\wh\psi,\zeta \wt z'} =
\sum_{\substack{\vi\in\Exp_{\wt u}\\ e'^*\vi\leq_{\zeta \wt z'} e'^*(\wh\psi- \rho(u)/\wh\rho(\eta))}} \hspace*{-10mm}\cF'_{\leq\vi,\zeta \wt u} = \\
\sum_{\substack{\vi\in\Exp_{\wt u}\\ e'^*\vi_\zeta \leq_{\wt z'} e'^*(\wh\psi- \rho(u)/\wh\rho(\eta))}} \hspace*{-10mm}\cF'_{\leq \vi_\zeta,\wt u} = \cG'_{\leq\wh\psi, \wt z'} \ ,
\]
where we recall that $\vi \in \Exp_{\wt u} \Leftrightarrow \vi_\zeta \in \Exp_{\wt u} $ for all $\zeta \in \mu_p $. Therefore, we can define the subsheaf $\cG_{\leq\wh\psi} \subset (\wt\pi \circ \wt e)^{-1} \cF $ by requiring its stalk at $\wt z $ to be
$\cG_{\leq\wh\psi,\wt z} := \cG'_{\leq\wh\psi, \wt z'} $ independent on the choice of the lift $\wt z' $ of $\wt z $. The same works for $\cG' $ descending to a subsheaf $\cG $ of $(\wt\pi \circ \wt e)^{-1} \cF $. Obviously, $\cG_{\leq\wh\psi} \subset \cG $.
\end{proof}

\begin{remarque}\label{rem:orderZ}
The indexing condition in the sum defining $\cG'_{\leq\wh\psi,\wt{z'}}$ can be written as $0\leq_{\wt z'} e'^*\omega_{\vi,\wh\psi}$, where $\omega_{\vi,\wh\psi}$ is defined by \eqref{eq:kappa}. Let $\wt\eta\in\wt\PP^1_\eta$ be the image of $\wt z\,'$ by $\wtpi\circ\wt e\,'$. We will be mainly interested in the case where $\wt\eta\in\SS^1_{\eta=0}$, in which case we denote it by $\vtn$. Then, when $\vi$ is fixed, the condition $e'^*\omega_{\vi,\wh\psi}\leq_{\wt z'} e'^*\omega_{\vi,\wh\psi'}$ is equivalent to $\wh\psi\leq_{\vtn}\wh\psi'$. Therefore, $\wh\psi\leq_{\vtn}\wh\psi'$ implies $\cG_{\leq\wh\psi,\wt z}\subset \cG_{\leq\wh\psi',\wt z}\subset\cG_{\wt z}$. The theorem below implies then that $(R^1(\wtpi\circ\wt e)_*\cG_{\leq\wh\psi})_{\vtn}$ is included in $(R^1(\wtpi\circ\wt e)_*\cG_{\leq\wh\psi'})_{\vtn}$ and defines a filtration of $(R^1(\wtpi\circ\wt e)_*\cG)_{\vtn}$.
\end{remarque}

\begin{thdef}\label{th:F0inftytop}
Let $(\cF,\cF_\bbullet)$ be a Stokes-filtered sheaf on $\wt\PP_t^1$, with trivial Stokes filtration away from $\SS^1_{t=0}$ and a Stokes filtration indexed by $\Exp$ on $\SS^1_{t=0}$ having ramification~$\rho$. Define $\wh\Exp$ and $\wh\rho$ according to the stationary phase formula, and $\cG,\cG_{\leq\wh\psi}$ as in Definition \ref{def:Gpsihat}.
\begin{itemize}
\item
Then, for each $\wh\psi\in\wh\Exp$, the complexes $\bR(\wtpi\circ\wt e)_*\cG$ and $\bR(\wtpi\circ\wt e)_*\cG_{\leq\wh\psi}$ have cohomology in degree one at most, and $R^1(\wtpi\circ\wt e)_*\cG_{\leq\wh\psi}$ is a subsheaf of $R^1(\wtpi\circ\wt e)_*\cG=\wt{\wh\rho}{}^{-1}\FcF$. This family of subsheaves, together with the natural isomorphisms induced by $\wt{\wh\rho}{}^{-1}\cG_{\leq\wh\psi}\simeq\cG_{\leq\wh\psi\circ\wt{\wh\rho}}$ for $\wh\psi\in\wh\Exp$, defines a Stokes filtration of $\FcF$.
\item
The Stokes-filtered sheaf obtained by the procedure of the theorem is called the localized $(0,\infty)$-Laplace transform of the Stokes-filtered local system $(\cL,\cL_\bbullet)$, and denoted by $\cF_\top^{(0,\infty)}(\cL,\cL_\bbullet)$, or simply by $(\wh\cL,\wh\cL_\bbullet)$.
\end{itemize}
\end{thdef}

Note that the inclusion $R^1(\wtpi\circ\wt e)_*\cG_{\leq\wh\psi}\subset R^1(\wtpi\circ\wt e)_*\cG$ is far from obvious from the sheaf-theoretic point of view.

\begin{theoreme}\label{th:RHF0infty}
If $(\cL,\cL_\bbullet)$ is the Stokes-filtered local system attached to $\ccM$, then $(\wh\cL,\wh\cL_\bbullet)$ is the Stokes-filtered local system attached to $\ccF^{(0,\infty)}\ccM$.
\end{theoreme}

\begin{proof}[\proofname\ of Theorems \ref{th:F0inftytop} and \ref{th:RHF0infty}]
Given a Stokes-filtered local system $(\cL,\cL_\bbullet)$ we can choose a finite dimensional $\CC\lpb t\rpb$-vector space with connection $\ccM$ such that $(\cL,\cL_\bbullet)$ corresponds to $\ccM$ by the Riemann-Hilbert correspondence recalled in \S\ref{subsec:Stokesfilteredls}. Both results are proved simultaneously, in order to apply Corollary \ref{cor:pushforward}, which reduces the problem to showing that there is, for each $\wh\psi\in\wh\Exp$, a natural isomorphism
\begin{equation}\label{eq:GpsiE}
\cG_{\leq\wh\psi}\isom\cH^0\DR^{\modD_Z}e^+(\pi^+M\otimes E^{\wh\psi(\eta)-t/\wh\rho(\eta)}).
\end{equation}
Naturality means here the following. Let us denote by $\wtj$ the open inclusion \hbox{$Z\moins D_Z\hto\wt Z$}. The pull-back $(\wpi\circ\wt e)^{-1}\cF$ of the local system $\cF$ on $\wt\PP^1_t$ satisfies $(\wpi\circ\wt e)^{-1}\cF=\wtj_*\wtj^{-1}(\wpi\circ\wt e)^{-1}\cF$. On $Z\moins D_Z$, $\DR e^+(\pi^+M\otimes E^{\wh\psi(\eta)-t/\wh\rho(\eta)})$ is naturally identified with $\wtj^{-1}(\wpi\circ\wt e)^{-1}\cF$, and both sheaves considered in the isomorphism above are naturally regarded as subsheaves of $\wtj_*\wtj^{-1}(\wpi\circ\wt e)^{-1}\cF$. The isomorphism above is nothing but the equality as such.

The question is therefore local on $\wt Z$, and thus on $\wt{\PP^1_t\times\Afu_\eta}$, and we can assume that $\pi^+M$ is decomposed with respect to its Hukuhara-Turrittin decomposition in the neighbourhood of $\wt z$. Then we are reduced to the case where $\Exp$ is reduced to one element, and then the assertion is easy to check.
\end{proof}

\begin{conclusion}\label{conclusion}
Computing $\ccF^{(0,\infty)}(\cL,\cL_\bbullet)$ reduces therefore to the following steps:
\begin{enumerate}
\item
Given the subset $\Exp$ and its ``Laplace transform'' $\wh\Exp$ obtained through the stationary phase formula, to compute a sequence of blowing-ups $e:Z\to\PP^1_t\times\Afu_\eta$ such that each $e^+(\pi^+M\otimes E^{\wh\psi(\eta)-t/\wh\rho(\eta)})$ is good.
\item
To compute the family of subsheaves $\cG_{\leq\wh\psi}$ of $(\wpi\circ\wt e)^{-1}\cF$.
\item
To compute the push-forwards $R^1(\wtpi\circ\wt e)_*\cG_{\leq\wh\psi}$ as subsheaves of $\wt{\wh\rho}^{-1}\FcF$.
\item
To make explicit, in the formula thus obtained, the $\mu_{\wh p}$-action induced by the isomorphisms $\wzeta^{-1}\cG_{\leq\wh\psi}\simeq\cG_{\leq\wh\psi_{\wzeta}}$.
\end{enumerate}
\end{conclusion}

\begin{remarque}[Moderate growth versus rapid decay]\label{rem:modversusrapid}
One can also determine the Stokes filtered local system $(\wh\cL,\wh\cL_\bbullet)$ by computing the subsheaves $\wh\cL_{<\wh\psi}$. The use of the moderate growth condition is mainly dictated by \eqref{eq:FcF}, in order to avoid mixed moderate growth/rapid decay in Theorem \ref{th:pushforward}. On the other hand, assume that $\ccM$ is purely irregular. Then\enlargethispage{1.5\baselineskip}%
\[
\DR^{\rmod(D_0\cup D_\infty)}(\pi^+M\otimes E^{-t/\tau})=\DR^{\rdc(D_0\cup D_\infty)}(\pi^+M\otimes E^{-t/\tau}).
\]
As a consequence, with such an assumption, we can replace everywhere in \S\ref{subsec:stokesFM} the moderate growth property with the rapid decay property, provided we change $\cG_{\leq\wh\psi}$ with $\cG_{<\wh\psi}$ and the order $\leqwtz$ with the strict order~$\lewtz$.
\end{remarque}

\Subsection{The Stokes structure on $\ccF^{(0,\infty)}\ccM$ under a simplifying assumption}\label{subsec:simplifassumpt}

The main example of this article satisfies the following assumptions.

\begin{assumption}\label{ass:noram}\mbox{}
\begin{enumerate}
\item\label{ass:noram1}
$\ccM$ is purely irregular,
\item\label{ass:noram2}
there exists a ramification $\rho:u\mto u^p=t$ of order $p$ and a $\CC\lpb u\rpb$-module $\ccN$ with connection such that $\ccM=\rho_+\ccN$ (equivalently, $M=\rho_+N$) and for which the Levelt-Turrittin decomposition of~$\ccN$ is not ramified:
\[
\CC\lpr u\rpr\otimes\ccN=\bigoplus_{\vi\in\Exp}(\wh\ccE^{-\vi}\otimes\wh\ccR_\vi),
\]
where $\Exp$ is a finite subset in $u^{-1}\CC[u^{-1}]$.
\end{enumerate}
\end{assumption}

Assumption \ref{ass:noram}\eqref{ass:noram2} means that no Stokes phenomenon for $\ccM$ is produced by the ramification. With this assumption, $\wh\rho^+\ccF^{(0,\infty)}(\ccM)$ is obtained through the diagram
\[
\xymatrix@=5mm{
&\Afu_u\times\GG_{\rmm,\eta}\ar[dl]_\pi\ar[dr]^{\wh \pi}&\\
\Afu_u&&\GG_{\rmm,\eta}
}
\]
by the formula
\[
\wh\rho^+\Fou M=\wh\pi_+(\pi^+N\otimes E^{-\rho(u)/\wh\rho(\eta)}).
\]
In other words, we can work directly with the $u$ coordinate, as if $M$ were non-ramified, and the only difference with the case where $M$ is non-ramified is the twist is by $E^{-\rho(u)/\wh\rho(\eta)}$, not by $E^{-u/\wh\rho(\eta)}$. For that reason, we will use the same notations for the maps in the $u$-variable, and we will now denote by $(\cL,\cL_\bbullet)$ the non-ramified Stokes-filtered local system attached to $\ccN$. We regard $N$ as a holonomic $\cD_{\PP^1_u}$-module which is localized at $u=\infty$. Considering now the diagram
\[
\xymatrix@=5mm{
&\PP^1_u\times\Afu_\eta\ar[dl]_\pi\ar[dr]^{\wh \pi}&\\
\PP^1_u&&\Afu_\eta
}
\]
we have
\[
\wh\rho^+\ccF^{(0,\infty)}(\ccM)=\wh\pi_+(\pi^+N\otimes E^{-\rho(u)/\wh\rho(\eta)}).
\]
Let $\wh\Exp\subset \eta^{-1}\CC[\eta^{-1}]$ be the set of exponential factors of $\CC\lpr\eta\rpr\otimes\wh\rho^+\ccF^{(0,\infty)}(M)$ as above.

\begin{definition}\label{def:suitable}
A proper modification $e:Z\to\PP^1_u\times\Afu_\eta$ is said to be \emph{suitable for $\ccM$} if
\begin{enumerate}
\item
it is a succession of point blowing-ups above $(0,\wh\infty)\in\PP^1_u\times\Afu_\eta$,
\item
the indeterminacy at $(0,\wh\infty)$ of each function $\omega_{\vi,\wh\psi}(u,\eta)$ ($\vi\in\Exp$, $\wh\psi\in\wh\Exp$), \cf\eqref{eq:kappa}, is \emph{mostly} resolved on $Z$, that is, the components of the divisors of zeros of $\omega_{\vi,\wh\psi}(u,\eta)$ which have multiplicity $>2$ do not meet the divisor of poles of $\omega_{\vi,\wh\psi}\circ e$ along $e^{-1}(0,\wh\infty)$, and those having multiplicity $\leq2$ meet it at most at smooth points.
\end{enumerate}
\end{definition}

It is known that such a suitable modification always exists. The results on $\wh\rho^+\ccF^{(0,\infty)}\ccM$ obtained in the previous subsection can be adapted in a straightforward way to the present setting, even for a suitable modification, provided that
\begin{enumerate}
\item
we denote by $(\cF,\cF_\bbullet)$ the Stokes-filtered sheaf attached to $N$ (in the $u$\nobreakdash-variable),
\item
we replace $t$ with $\rho(u)$ and $\wt t$ with $\wt u$ in the formulas.
\end{enumerate}

With Assumption \ref{ass:noram}\eqref{ass:noram2}, we do not have to take care of the ramification of $M$. With Assumption \ref{ass:noram}\eqref{ass:noram1}, we can work in the rapid decay setting. The explicit description of the family of subsheaves $\cG_{<\wh\psi}$ is then simpler, due to the following lemma.

\begin{lemme}
If $Z$ is suitable for $\ccM$, then it satisfies the conclusion of Theorem \ref{th:Majima} for $N$ in the rapid decay setting.
\end{lemme}

\begin{proof}
Clearly the goodness condition holds away from the intersection of the zero set of the functions $\omega_{\vi,\wh\psi}$ with the divisor $D_Z$. We can then argue as in \cite[Lem.\,A.1]{Bibi13} for the zeros of multiplicity $\leq2$.
\end{proof}

Note that we use the convention that the inequality $0\lewtz e^*\omega_{\vi,\wh\psi}$ is \emph{not} satisfied at a point $\wz$ whose image in $Z$ is an indeterminacy point of $e^*\omega_{\vi,\wh\psi}$. We also use \cite[Lem.\,A.1]{Bibi13} to prove \eqref{eq:GpsiE} at these points.

\subsection{The case of an elementary meromorphic connection}\label{sec:caseofelement}
We now restrict to the case where $\ccM=\El(\rho,-\vi,\ccR)$, with $0\neq\vi\in u^{-1}\CC[u^{-1}]$, $\rho:u\mto u^p=t$ and~$\ccR$ is a regular connection. We thus have $\ccM=\rho_+\ccN$ with $\ccN=\ccE^{-\vi}\otimes\ccR$. Assumption \ref{ass:noram} is thus satisfied. Since~$\ccR$ is a successive extension of rank-one meromorphic connections, the monodromy of $\ccF^{(0,\infty)}\ccM$ is given by Proposition \ref{prop:topmonodromyF}, where $T_\top$ is the monodromy corresponding to $\rho_+\ccR$, that is, such that $T_\top^p=T$, if $T$ is the monodromy of $\ccR$.

We set $\vi=\vi_q u^{-q}(1+\ro(u))$ with $\vi_q\in\CC\moins\{0\}$. According to \cite{Fang07,Bibi07a}, we have $\CC\lpr\tau\rpr\otimes\ccF^{(0,\infty)}\ccM\simeq\El(\wh\rho,-\wvi,(-1)^q\ccR)$, with
\begin{align}
\wh\rho(\eta)&=-\frac{\rho'}{\vi'}=\frac{p}{q\vi_q}\,\eta^{p+q}(1+\ro(\eta)),\notag\\
\wvi(\eta)&=\vi(\eta)+\frac{\rho(\eta)}{\wh\rho(\eta)}=\frac{p+q}{p}\,\vi_q\eta^{-q}(1+\ro(\eta)).\label{eq:wvieta}
\end{align}
We thus have
\[
\CC\lpr\eta\rpr\otimes\wh\rho^+\ccF^{(0,\infty)}\ccM\simeq\bigoplus_{\wzeta\in\mu_{p+q}}\CC\lpr\eta\rpr\otimes(\cE^{-\wvi_{\wzeta}(\eta)}\otimes(-1)^q\ccR),
\]
with $\wvi_{\wzeta}(\eta)\!=\!\wvi\bigl(\wzeta\eta(1+\ro(\eta)\bigr)$. As a consequence, $\Exp\!=\!\{\vi\}$ and \hbox{$\wh\Exp\!=\!\{\wvi_{\wzeta}(\eta)\mid\wzeta\!\in\!\mu_{p+q}\}$}. The blowing-up $e:Z\to\PP^1_u\times\Afu_\eta$ (see \S\ref{subsec:simplifassumpt}) is obtained by resolving the family of rational functions
\begin{equation}\label{eq:alphadef}
\omega_{\wzeta}(u,\eta):=\wvi_{\wzeta}(\eta)-\vi(u)-\rho(u)/\wh\rho(\eta),\quad\wzeta\in\mu_{p+q}.
\end{equation}

\subsubsection*{Blow-up of \texorpdfstring{$(0,\wh\infty)$}{0infty}}
Let us denote by $\ve$ the single blow-up of the point $(0,\wh\infty)$. In the chart centered at $(0,\wh\infty)$, the divisor $D$ reduces to $D_0\cup D_{\wh\infty}$, with $D_0=\{u=0\}$ and $D_{\wh\infty}=\{\eta=0\}$ (\cf\S\ref{subsec:reminderlocalLaplace}, by using the ramified coordinates $u,\eta$). We denote by the same name their strict transforms by the blow-up map $\ve$, and by $E$ the exceptional divisor $\ve^{-1}(0,\wh\infty)$.
\begin{figure}[htb]
\begin{center}
\begin{picture}(70,20)(0,10)

	\put(0,20){\line(1,0){15}}
	\put(8,28){\line(0,-1){15}}
	\put(4,20){\vector(-1,0){2}}
	\put(8,20){\vector(0,-1){5}}

	\put(20,20){$\stackrel{\ve}{\longleftarrow}$}
	
	\put(45,30){\line(-1,0){15}}
	\put(38,32){\line(2,-1){25}}
	\put(60,30){\line(0,-1){20}}
	
	\put(35,30){\vector(-1,0){2}} \put(32,32){\footnotesize $y $}
	\put(46,28){\vector(2,-1){3}} \put(48,28){\footnotesize $x$}
	\put(60,18){\vector(0,-1){3}} \put(61,15){\footnotesize $w $}
	\put(58,22){\vector(-2,1){3}} \put(55,24){\footnotesize $z $}

	\put(2,22){\footnotesize $\eta $}
	\put(9,16){\footnotesize $u $}
\put(42,18){\footnotesize $E$}
\qbezier(45,20)(49,22)(50,25)

\put(66,26){\footnotesize $D_{\wh\infty}$}
\qbezier(65.5,25.5)(63,23)(60.5,24)

\put(28.5,23.5){\footnotesize $D_0$}
\qbezier(32,25)(37,26)(38,29)
\end{picture}
\end{center}
\caption{The blow-up map $\ve: Z\to\P^1_u \times \Af^1_\eta $.}\label{figure:pblowup}
\end{figure}

The chart $(u,\eta)$ is covered by two charts with respective systems of coordinates $(x,y)$ and $(w,z)$, such that $\ve(x,y)=(xy,y)$ and $\ve(w,z)=(w,wz)$. The strict transform of $D_0$, defined by $x=0$, is located in the first chart, and that of $D_{\wh\infty}$, defined by $z=0$, is located in the second chart. Moreover, $(x{:}z)$ forms a system of projective coordinates on $E\simeq\PP^1$.

In the coordinates $(x,y)$ we find\enlargethispage{-\baselineskip}%
\begin{equation}\label{eq:alphadefxy}
\ve^*\omega_{\wzeta}(x,y)=-\frac{\vi_q}{p}\,x^{-q}y^{-q}\big[qx^{p+q}-(p+q)\wzeta^{-q}x^q+p+\ro(xy)\big].
\end{equation}
In these coordinates, the indeterminacy locus of $\omega_{\wzeta}(x,y)$ consists of the points with coordinate~$x$ on the divisor $E=\{y=0\}$ which are roots of the polynomial $f_{\wzeta}(X):=f(X/\wzeta)$ with
\begin{equation}\label{eq:fX}
f(X)=qX^{p+q}-(p+q)X^q+p.
\end{equation}

\begin{lemme}\label{lem:indeterminacykappa}
Under Assumption \ref{ass:pqcoprime}, the polynomial has a single double root at $X=1$ and $(p+q)-2$ simple roots which are all nonzero and whose absolute value is~$>1$ for $p-1$ of them and $<1$ for the remaining $q-1$.\qed
\end{lemme}

As a consequence, the blow-up $\ve:Z\to\PP^1_u\times\Afu_\eta$ of $(0,\wh\infty)$ is suitable for $\ccM$ in the sense of Definition \ref{def:suitable}.

\section{The case of an elementary connection -- cohomological description}

Following Conclusion \ref{conclusion} together with Remark \ref{rem:modversusrapid}, we are now going to describe the family of subsheaves~$\scrG_{< \wpsi}$ of $\cG\subset(\wt{\pi} \circ \wt\ve)^{-1} \mathscr{F}$ and their push-forwards $R^1(\wtpi \circ \wt\ve)_* \scrG_{<\wpsi}$, where $\ve:Z\to\P^1_u \times \Af^1_\eta $ is the blow-up of $(0, \wh\infty)$.

\subsection{The sheaves \texorpdfstring{$\cG$}{G} and \texorpdfstring{$\scrG_{< \wpsi}$}{Gpsi} on the fibres of \texorpdfstring{$\wtpi$}{pi}}\label{subsec:GGpsi}

\subsubsection{The real blow-up space \texorpdfstring{$\wt Z$}{wtZ}}
Let us first describe the map $\wtpi\circ\wt\ve:\wt Z\to\wt\PP^1_\eta$. It~will be enough for us to understand the description of $(\wtpi\circ\wt\ve)^{-1}(\SS^1_{\eta=0})$. This space lies over $D_{\wh\infty}\cup E$.

Over the coordinate chart $(x,y)$ of $Z$ (see Figure \ref{figure:pblowup}) is the chart of $\wt Z$ with polar coordinates $(r_x,\vt_x,r_y,\vt_y)$. In this chart, the pull-back $\wE$ of $E$ is defined by $r_y=0$ and the map $\wtpi\circ\wt\ve$ is given by $(r_x,\vt_x,r_y,\vt_y)\mto(r_y,\vt_y)$. Therefore, in this chart, $(\wtpi\circ\wt\ve)^{-1}(\SS^1_{\eta=0})$ is the product of the annulus $[0,\infty)\times\SS^1_{x=0}$ with the circle $\SS^1_{y=0}=\SS^1_{\eta=0}$. The restriction to $\wE$ of $\wpi\circ\wt\ve$ takes values in $\SS^1_{u=0}$ and is given by $(r_x,\vt_x,\vt_y)\mto\nobreak\vt=\nobreak\vt_x+\nobreak\vt_y$. Therefore, it will be more convenient for us to use the description $\wE=[0,\infty)\times\SS^1_{u=0}\times\SS^1_{\eta=0}$, with the coordinate $\vtn=\vt_y$ on $\SS^1_{\eta=0}=\SS^1_{y=0}$.

Over the coordinate chart $(w,z)$ of $Z$ is the chart of $\wt Z$ with polar coordinates $(r_w,\vt_w,r_z,\vt_z)$. In this chart, the pull-back $\wE\cup\wD_{\wh\infty}$ of $E\cup D_{\wh\infty}$ is defined by $r_wr_z=0$ and the map $\wtpi\circ\wt\ve$ is given by $(r_w,\vt_w,r_z,\vt_z)\mto(r_wr_z,\vt_w+\vt_z)$, while the map $\wpi\circ\wt\ve$ is given by $(r_w,\vt_w,r_z,\vt_z)\mto(r_u,\vt)=(r_w,\vt_w)$. We can therefore identify $\wD_{\wh\infty}$ with~$\wt\PP^1_u\times\SS^1_{\eta=0}$ so that the restriction of the map $\wpi\circ\wt\ve$ (\resp$\wtpi\circ\wt\ve$) is identified with the first (\resp the second) projection.

As a consequence, $\wD_{\wh\infty}\cup\wE$ it obtained by gluing
\begin{equation}\label{eq:DtildeEtilde}
\begin{split}
\wD_{\wh\infty}&=\{(r_u,\vt,\vtn)\mid r_u\in[0,\infty],\;\vt\in\SS^1_{u=0},\;\vtn\in\SS^1_{\eta=0}\}\quad\text{with}\\
\wE&=\{(r_x,\vt,\vtn)\mid r_x\in[0,\infty],\;\vt\in\SS^1_{u=0},\;\vtn\in\SS^1_{\eta=0}\}
\end{split}
\end{equation}
along $r_u=0$ (\resp$r_x=\infty$) by the identity on $\SS^1_{u=0}\times\SS^1_{\eta=0}$.
\begin{figure}[htb]
\begin{center}
\includegraphics[scale=.5]{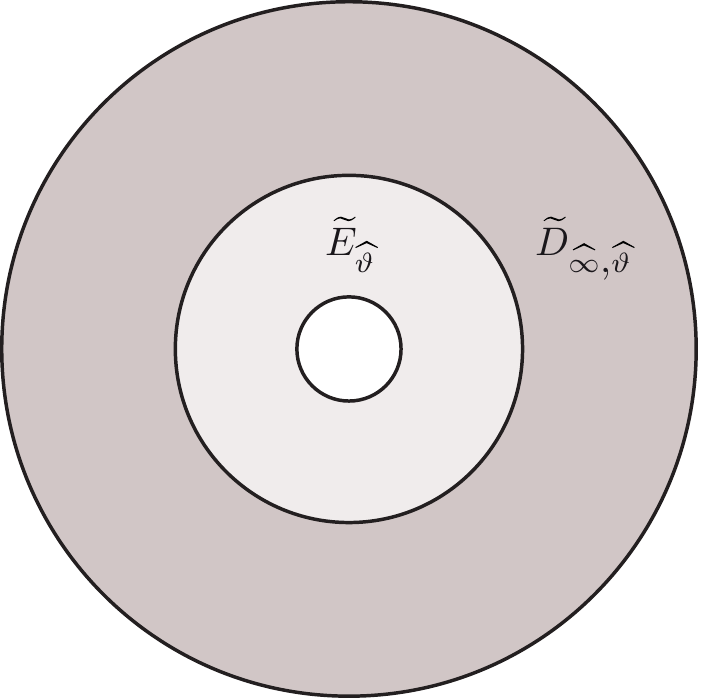}
\caption{The annulus $\wD_{\wh\infty,\vtn}\cup\wE_{\vtn}$.}\label{fig:ovA}
\end{center}
\end{figure}

\subsubsection{The sheaf \texorpdfstring{$\cG$}{G} on \texorpdfstring{$\wE\cup\wD_{\wh\infty}$}{ED}}
We denote by $\beta:B\hto\wE\cup\wD_{\wh\infty}$ the inclusion of the open subset consisting of the union of
\begin{itemize}
\item
$\wE\moins\{r_x=0\}$, $\wD_{\wh\infty}\moins\{r_u=\infty\}$,
\item
$\Big\{(0,\vt,\vtn)\mid q\vt\in\arg(\vi_q)-\rdInt\Big\}\subset\{r_x=0\}$,
\item
$\Big\{(\infty, \vt,\vtn) \mid p \vt-(p+q)\vtn\in -\arg(\vi_q)+\rdInt\Big\}\subset\{r_u=\infty\}$.
\end{itemize}

\begin{lemme}\label{lem:cG}
Let us still denote by $\wpi\circ\wt e$ the projection $\wE\cup\wD_{\wh\infty}\to\wt\PP^1_u$. Then $\cG_{|\wE\cup\wD_{\wh\infty}}=\beta_!(\wpi\circ\wt e)^{-1}\cF$.
\end{lemme}

\begin{proof}
By analyzing the map $\wt e:\wt Z\to\wt\PP^1_u\times\wt\Afu_\eta$ in the neighbourhood of $\wE\cup\wD_{\wh\infty}$.
\end{proof}

\subsubsection{The sheaves \texorpdfstring{$\scrG_{<\wpsi}$}{Gpsi} on \texorpdfstring{$\wE\cup\wD_{\wh\infty}$}{ED}}
Assume that $\wpsi=\wvi_{\wzeta}$. According to Definition \ref{def:Gpsihat} (in~the rapid decay setting) and recalling that $\Exp=\{\vi\}$, we obtain:

\begin{lemme}\label{lem:Balpha}
The restriction to $\wE\cup\wD_{\wh\infty}$ of the sheaf $\scrG_{<\wvi_{\wzeta}}$ is the subsheaf of $\cG$ given~by
\begin{starequation}\label{eq:NUStrich}
\shG_{< \wvi_{\wzeta}|\wE\cup\wD_{\wh\infty}}=(\beta_{\omega_{\wzeta}>0})_! \cG_{|\wE\cup\wD_{\wh\infty}}
\end{starequation}
where $\beta_{\omega_{\wzeta}>0}$ denotes the inclusion of the open subset
\begin{starstarequation} \label{eq:Balpha}
B_{\omega_{\wzeta}>0} :=\big\{\wz\in\wE\cup\wD_{\wh\infty}\mid 0 \lewtz\ve^* \omega_{\wzeta}(u, \eta)\big\} \hto\wE\cup\wD_{\wh\infty}.
\end{starstarequation}%
\end{lemme}

\begin{proof}
Let us make precise the reasoning at the points $\wz$ whose projection on $E\cup Z$ is an indeterminacy point for $\omega_{\wzeta}$. On $D_{\wh\infty}$, the factor $\ve^*\omega_{\wzeta}$ written in the coordinates $(w,z)$ reads as
$$
\ve^*\omega_{\wzeta}(w,z)=\frac{1}{w^q z^{p+q}} \cdot \lambda_{\wzeta}(w,z)
$$
for some holomorphic function $\lambda_{\wzeta}$ such that $\lambda_{\wzeta}(w,0)=\sfrac{-q \vi_q}{p}\neq 0$, a value which is independent of $\wzeta\in\mu_{p+q}$. Hence there is no indeterminacy point of $\omega_{\wzeta}$ on $D_{\wh\infty}$. We conclude from Lemma \ref{lem:indeterminacykappa} that these points lie on $E\moins(D_0\cup D_{\wh\infty})$ and that, in a suitable local coordinate $v$ on $E$ vanishing at such a point, $\ve^*\omega_{\wzeta}$ can be written as $v^a/y^q$ with $a=1$ or $a=2$. In any punctured neighbourhood of a point $(v=0,\vtn)$, the argument of $\ve^*\omega_{\wzeta}$ takes all possible values, therefore such a point cannot be inside $B_{\omega_{\wzeta}>0}$ (nor inside $B_{\omega_{\wzeta}<0}$). It thus lies on the boundary of $B_{\omega_{\wzeta}>0}$ and, according to a computation similar to \cite[Lem.\,A.1]{Bibi13}, Formula \eqref{eq:NUStrich} also holds at such a point.
\end{proof}

In other words, we will forget the indeterminacy points in the remaining part of this article.

\subsubsection{The sheaf \texorpdfstring{$\shG_{<\wvi_{\wzeta}}$}{Gpsi} on \texorpdfstring{$\wD_{\wh\infty,\vtn}$}{D}} \label{sec:sheafD}

We have
\begin{equation}\label{eq:BalphinD}
B_{\omega_{\wzeta}>0}\cap\wD_{\wh\infty}=
\Big\{(r_u, \vt,\vtn) \mid p \vt-(p+q)\vtn\in -\arg(\vi_q)+\rdInt\Big\}.
\end{equation}
This is the pull-back under the radial projection $[0,\infty]\ni r_u\mto\infty$ of the set \hbox{$B\cap\{r_u=\infty\}$} considered before Lemma \ref{lem:cG} and is independent of $\wzeta$. As a consequence, on $\wD_{\wh\infty}$, $\shG_{<\wvi_{\wzeta}|\{r_u=\infty\}}$ is equal to $\cG_{|\{r_u=\infty\}}$ and $\shG_{<\wvi_{\wzeta}|\{r_u<\infty\}}\subset\cG_{|\{r_u<\infty\}}$.

\subsubsection{The sheaf \texorpdfstring{$\shG_{<\wvi_{\wzeta}}$}{Gpsi} on \texorpdfstring{$\wE_{\vtn}$}{E}}
From \eqref{eq:alphadef} and \eqref{eq:alphadefxy}, the condition $0\lewtz\ve^* \omega_{\wzeta}$ for the point $\wz=(r_{x}, \vt, 0, \vtn)\in\wE$ lying over $(x,0)\in E\moins\wzeta f^{-1}(0)$ reads
\[
q \vt -\arg(f(\wzeta^{-1} x))\in\arg(\vi_q) +\rdInt\pmod{2\pi}.
\]
Therefore
\begin{equation}\label{eq:alphmz1}
B_{\omega_{\wzeta}>0}\cap\wE=\Big\{(r_x, \vt,\vtn) \mid q \vt -\arg(f(\wzeta^{-1} x))\in\arg(\vi_q) +\rdInt\bmod2\pi\Big\}.
\end{equation}

From \eqref{eq:Balpha} it clearly follows:

\refstepcounter{equation}\label{increasingB}
\noindent\eqref{increasingB}\enspace
For each fixed $\vtn$, the fibres $(B_{\omega_{\wzeta}>0}\cap\wE)_{\vtn}$ are increasing with respect to $\wzeta\in\nobreak\mu_{p+q}$ ordered according to the order of the family $(\wvi_{\wzeta})_{\wzeta}$ at $\vtn$.\par\smallskip

Let us observe that for $|x|\gg 1$, we have $q \vt_{x}-\arg(f(\wzeta^{-1} x)) \sim -p \vt_{x}$,
so that in the limit $|x|\to\infty $, we have
\begin{equation}\label{eq:Evt0atinfty}
B_{\omega_{\wzeta}>0}\cap\wE\cap\{r_x=\infty\}=\Big\{(\infty, \vt,\vtn) \mid p \vt -(p+q) \vtn\in -\arg(\vi_q)+\rdInt\Big\},
\end{equation}
which is independent of $\wzeta\in\mu_{p+q}$ and coincides with $B_{\omega_{\wzeta}>0}\cap\wD_{\wh\infty}\cap\{r_u=0\}$ (see \eqref{eq:BalphinD}).

Similarly, for $|x|\to 0$, we see that $-q \vt_x+\arg(f(\wzeta^{-1}x)) \sim -q \vt_x$ and hence
\begin{equation}\label{eq:Evt0at0}
B_{\omega_{\wzeta}>0}\cap\wE\cap\{r_x=0\}=\Big\{(0, \vt,\vtn) \mid q \vt\in\arg(\vi_q) +\rdInt\Big\},
\end{equation}
is independent on $\wzeta\in\mu_{p+q}$, and coincides with $B\cap\{r_x=0\}$.

As a consequence, $\shG_{<\wvi_{\wzeta}|\wE}$ is obviously contained in $\cG_{|\wE}$, and coincides with $\cG$ on $\{r_x=0\}$.

\subsection{The Stokes-filtered sheaf \texorpdfstring{$(\protect\wtrho^{-1}\protect\wh{\cL}, \protect\wh{\cL}_\bbullet)$}{rhoL}}
From Theorems \ref{th:F0inftytop} and \ref{th:RHF0infty} in the rapid decay case and after the ramification $\wtrho$, we obtain:

\begin{corollaire}\label{cor:topwhL}
The Stokes-filtered sheaf $(\wtrho^{-1}\wh\cL,\wh\cL_\bbullet)$ on $\SS^1_{\eta=0}$ is given by:
\begin{itemize}
\item
$\wtrho^{-1}\wh\cL=R^1(\wtpi \circ \wt\ve)_* \scrG$,
\item
for each $\wzeta\in\mu_{p+q}$, $\wh\cL_{<\wvi_{\wzeta}}=R^1(\wtpi \circ \wt\ve)_* \scrG_{<\wvi_{\wzeta}}$.
\end{itemize}
In particular, the latter is contained in the former.\qed
\end{corollaire}

We sill simplify the topological situation in order to ease the explicit computation of these objects. We continue to restrict the sheaves to $\wE\cup\wD_{\wh\infty}$. We notice that
$B_{\omega_{\wzeta}>0}\cap\wD_{\wh\infty}$ defined by \eqref{eq:BalphinD} is independent of~$\wzeta$. We thus set $B^\circ:=(B\cap\wE)\cup(B_{\omega_{\wzeta}>0}\cap\wD_{\wh\infty})$ (for any $\wzeta$). We have, for any $\wzeta$, the inclusions $B_{\omega_{\wzeta}>0}\subset B^\circ\subset B$. The first inclusion is an equality on $\wD_{\wh\infty}$, while the second one is so on $\{r_u=\infty\}$. We denote by $\beta^\circ$ the inclusion $B^\circ\hto \wE\cap\wD_{\wh\infty}$ and we set $\cG^\circ:=(\beta^\circ)_!(\wpi\circ\wt e)^{-1}\cF$. We thus have a sequence of inclusions $\scrG_{<\wvi_{\wzeta}}\hto\cG^\circ\hto\cG$.

\begin{lemme}\mbox{}
\begin{enumerate}
\item
The inclusion $\cG^\circ\hto\cG$ induces an isomorphism $R^1(\wtpi \circ \wt\ve)_* \scrG^\circ\isom R^1(\wtpi \circ \wt\ve)_* \scrG$.
\item
The natural restriction morphisms
\[
R^1(\wtpi \circ \wt\ve)_* \scrG^\circ\to R^1(\wtpi \circ \wt\ve)_* \scrG^\circ_{|\wE}\quad\text{and}\quad R^1(\wtpi \circ \wt\ve)_* \scrG_{<\wvi_{\wzeta}}\to R^1(\wtpi \circ \wt\ve)_* \scrG_{<\wvi_{\wzeta}|\wE}
\]
are isomorphisms.
\qed
\end{enumerate}
\end{lemme}

We can therefore replace, in Corollary \ref{cor:topwhL}, the triple $(\wE\cup\wD_{\wh\infty},(B_{\omega_{\wzeta}>0})_{\wzeta\in\mu_{p+q}},B)$ with $(\wE,(B_{\omega_{\wzeta}>0}\cap\wE)_{\wzeta\in\mu_{p+q}},B^\circ\cap\wE)$ and the pair $((\scrG_{<\wvi_{\wzeta}})_{\wzeta\in\mu_{p+q}},\cG)$ with the pair $((\scrG_{<\wvi_{\wzeta}|\wE})_{\wzeta\in\mu_{p+q}},\cG^\circ_{|\wE})$.

\subsection{Stokes data for \texorpdfstring{$\protect\wrho\protect\ccF^{(0,\infty)}\ccM$}{rhoFM} -- abstract version}\label{subsec:Stokesabstract}
The computation of the Stokes data can be done with the simplified model above. Namely, we denote by~$\wA$ the annulus $[0,\infty]\times\SS^1_{u=0}$ with coordinates $(r,\vt)$, where $r$ now abbreviates $r_x$. The open annulus $A$ is defined as $(0,\infty)\times\SS^1_{u=0}$ and the boundary components are denoted by $\bin\ov A$ ($r=0$) and $\bout\ov A$ ($r=\infty$). We denote by $\wt\cF$ the local system on $\wA\times\SS^1_{\eta=0}\simeq\wE$ pull-back of $\cF$ by $\wpi\circ\wt e$ and we consider the subsets $B_{\omega_{\wzeta}>0}$ and $B$, which now denote the intersection of $B_{\omega_{\wzeta}>0}$ (\resp$B^\circ$) with $\wA\times\SS^1_{\eta=0}$. We will denote by~$\wt\cF{}^{\vtn}$ the restriction of $\wt\cF$ to the fibre $\wA\times\{\vtn\}$ (so $\wt\cF{}^{\vtn}$ is canonically equal to $\cF$), and similarly for $B_{\omega_{\wzeta}>0}$ and $B$, as well as for $\beta_{\omega_{\wzeta}>0}$ and $\beta$.

Let us choose an \angle $\vtn_o\in\SS^1_{\eta=0}$ such that the set $\{\vtn_\ell:=\vtn_o+\sfrac{\ell\pi}{q} \mid \ell\in\Z\}$ does not contain a Stokes direction for any difference of two exponential factors $\wvi_{\wzeta}$ and $\wvi_{\wzeta'}$ for $\wzeta,\wzeta'\in\mu_{p+q}$. Let $I_\ell:=[\vtn_\ell, \vtn_{\ell+1}] $.
We then define $\bLMH_\ell :=H^1(\wA, (\beta^{\vtn_\ell})_!\wt\cF{}^{\vtn_\ell})$ for $\ell=0,\dots, 2q-1$ and
$S_\ell^{\ell+1}: \bLMH_\ell\to\bLMH_{\ell+1}$
by the diagram of isomorphisms induced by the restriction isomorphisms
\begin{equation}\label{eq:Labstr}
\begin{array}{c}
\xymatrix@C=.5cm{
& H^1(\wA \times I_\ell, (\beta^{I_\ell})_!\wt\cF) \ar[dl]_{\cong} \ar[dr]^{\cong} \\
H^1(\wA, (\beta^{\vtn_\ell})_!\wt\cF{}^{\vtn_\ell}) \ar[rr]^{S_\ell^{\ell+1}} &&
H^1(\wA, (\beta^{\vtn_{\ell+1}})_!\wt\cF{}^{\vtn_{\ell+1}})
}
\end{array}
\end{equation}
We arrange the elements of $\mu_{p+q}$, identified with $\wh\Exp$ through $\wzeta\mto\wvi_{\wzeta}$, according to the ordering \eqref{eq:levt} at $\vtn_o $:
\[
\mu_{p+q}=\{\wzeta_0 <_{\vtn_o} \cdots <_{\vtn_o} \wzeta_{p+q-1}\}.
\]
This ordering repeats itself at $\wzeta_\ell$ for each even $\ell=2 \mu$ and is completely opposite for each odd $\ell=2 \mu +1$. From Corollary \ref{cor:topwhL} in this simplified setting, we conclude that the family
\begin{equation}\label{eq:Filell}
H^1\bigl(\wA, (\beta^{\vtn_\ell}_{\omega_{\wzeta_k>0}})_!\wt\cF{}^{\vtn_\ell}\bigr)_{k=0,\dots,p+q-1}
\end{equation}
forms a increasing (\resp decreasing) filtration of $H^1(\wA, (\beta^{\vtn_\ell})_!\wt\cF{}^{\vtn_\ell})$ for $\ell$ even (\resp$\ell$~odd) in $\{0,\dots,2q-1\}$. We denote by $F \bLMH_\ell$ such an increasing/decreasing filtration. We can conclude:

\begin{theoreme}\label{theo:Stokesdataabstract}
The Stokes structure of the de-ramified local Fourier transform $\wrho\protect\ccF^{(0,\infty)}\ccM$ of the elementary meromorphic connection $\ccM=\El(\rho, -\vi, \ccR)$ can be represented by the following linear Stokes data:
\begin{itemize}
\item
the vector spaces $\bLMH_\ell=H^1\bigl(\wA, (\beta^{\vtn_\ell})_!\wt\cF{}^{\vtn_\ell}\bigr)$ --see \eqref{eq:Labstr}-- for $\ell=0,\dots, 2q-1$,
\item
the isomorphisms $S_\ell^{\ell+1}$ from \eqref{eq:Labstr},
\item
the filtrations $F\bLMH_\ell$ given by \eqref{eq:Filell}.\qed
\end{itemize}
\end{theoreme}

\section{Topology of the support of the sheaves \texorpdfstring{$\shG_{<\protect\wvi_{\wzeta}}$}{Gwvi} on the annulus}\label{sec:support}

In this section, we will describe the shape of the support $B^{\vtn}_{\omega_{\wzeta}>0}$ inside the annulus~$\wA$ in order to obtain Corollary \ref{cor:Geta2} below.

\subsection{Morse theory on the closed annulus}
In this section, it will be more convenient to use the argument $\arg x=\vt_x$, related to $\vt$ as follows:
\begin{equation}\label{eq:vtxvt}
\vt=\vt_x+\vtn.
\end{equation}
We consider the function (\cf\eqref{eq:fX})
\begin{align*}
\wA \smallsetminus \{x \mid f(x)=0\} &\To{G} \SS^1, \\[-3pt]
x &\Mto{\hphantom{G}} \theta=\arg(f(x)/x^q)=-q \vt_x+\arg f(x)
\end{align*}
with $f(x)$ given by \eqref{eq:fX}. For $\wzeta\in\mu_{p+q}$ and $\vtn\in\SS^1_{\eta=0}$, we set
\[
G_{\wzeta}^{\vtn}(x):=G(\wzeta^{-1}x)-q(\vtn+\arg\wzeta)+\arg(\vi_q)=-q \vt+\arg(\vi_q)+\arg f(\wzeta^{-1}x).
\]
We have
\begin{equation}\label{eq:Bvt0}\textstyle
B^{\vtn}_{\omega_{\wzeta}>0}= (G_{\wzeta}^{\vtn})^{-1}\bigl(\rdInt\bigr).
\end{equation}

Let us denote by $b:\wA{}'\to\wA$ the real oriented blow-up of the roots of the polynomial $f(x)$ (where $G$ is not defined) and let us set $A'=A\moins f^{-1}(0)$. We will use classical Morse-theoretic arguments for the extension of $G$ to the manifold with boundary~$\wA{}'$ (\cf \cite[\S3]{H-L73}).

\begin{lemme} \label{lem:claim1}
The function $G$ extends to a differentiable function on $\wA{}'$. Its restriction to $\partial\wA{}'$ has no critical point. The set of critical points in $A'$, which are all of Morse type of index one, is equal to $\mu_{p+q}-\{1\}$ embedded in $\{x\mid|x|=1\}$. To each critical point $\wzeta\in\mu_{p+q}-\{1\}$ corresponds the critical value $\arg(\wzeta^p-1)=\frac\pi2+\frac{\kappa\pi}{p+q}\bmod2\pi$ for some $\kappa=1,\dots,p+q-1$, which belongs to $(\frac\pi2,\frac{3\pi}{2})$. The critical value $\frac\pi2+\frac{\kappa\pi}{p+q}\bmod2\pi$ comes from exactly one critical point, which is $\exp(\frac{2qk\pi}{p+q})$ for $k\in\{1,\dots,p+q-1\}$ such that $k\equiv a\kappa\bmod(p+q)$, with $a$ as in \eqref{eq:ordev}.
\end{lemme}

\begin{proof}
The assertion on the critical points and their values is obtained by noticing that $G$ is obtained as the composition of the function $x\mto f(x)/x^q$ from $A'$ to~$\CC^*$ with the projection to $\SS^1$.

Let us consider the local behaviour of $G$ around a root $x_o\in\wA$ of $f(x)$. Let~$\SS^1_{x_o}$ denote the boundary circle over $x_o$ and assume first that $x_o$ is a simple root. The function $G$ then reads
$$
G(x)=\arg(\sfrac{f(x)}{x^q})=\arg((x-x_o) \cdot \sfrac{g(x)}{x^q})
$$
in a neighborhood $\SS^1_{x_o}\subset \wA{}' $ for some polynomial $g(x)$ which does not vanish at~$x_o$. This shows that $G$ extends differentiably to $\SS^1_{x_o}$ and its restriction $\SS^1_{x_o}\to\SS^1$ has maximal rank everywhere. Locally around the double zero $x=1$, the situation amounts to
$$
G(x)=\arg\bigl((x-1)^2\cdot \sfrac{g(x)}{x^q}\bigr)
$$
with $g(x)$ non-vanishing at $x=1$. The same arguments as above apply regarding the extension of $G$ to $\SS^1_1$. Lastly, the behaviour of $G$ on $\partial\wA$ has been given in \S\ref{subsec:GGpsi}.
\end{proof}

As a consequence, we can regard $G$ as a Morse function on the pair $(\wA{}',\partial\wA{}')$, and also on the pair $(\wA\moins f^{-1}(0),\partial\wA)$ (\cf\cite[Th.\,3.1.6]{H-L73}).

\begin{proposition}\label{prop:G1}
For each $\theta\in\SS^1$, each connected component of the subset $B^\theta:=
G^{-1}\bigl(\theta+\nobreak(\tfrac{\pi}{2},\tfrac{3\pi}{2})\bigr)\subset \wA\moins f^{-1}(0)$ is homeomorphic to the union of an open disc with holes (\ie interior closed discs deleted) and nonempty disjoint open intervals (at~least one) in its boundary. It is embedded in $\wA$ in such a way that the nonempty open intervals in the boundary are contained in $\partial \wA$.
\end{proposition}

\begin{proof}
We will start with the case $\theta=\pi$.

\begin{lemme}\label{lem:G1}
The subset $
G^{-1}\bigl((-\tfrac{\pi}{2},\tfrac{\pi}{2})\bigr)\subset \wA\moins f^{-1}(0)$ is homeomorphic to a disjoint union of $(p+q)$ semi-closed discs, each of them being a closed disc with a closed interval in its boundary deleted, embedded in $\wA$ in such a~way that the remaining nonempty open interval in the boundary is contained in $\partial \wA$.
\end{lemme}

Consider the annulus $\wA=[0,\infty] \times \SS^1$ with coordinate $(r,\vt_x)$. We will use the notation
\begin{align}\label{eq:evDeltain}
\begin{cases}
\evDin_m :=\{r=0\} \times\bigl(\frac{2m \pi}{q}- \frac{\pi}{2q}, \frac{2m \pi}{q}+\frac{\pi}{2q}\bigr),\\[8pt]
\oddDin_m :=\{r=0\} \times\bigl(\frac{(2m+1)\pi}{q}- \frac{\pi}{2q}, \frac{(2m+1)\pi}{q}+\frac{\pi}{2q}\bigr),
\end{cases}
&m=0,\dots,q-1\\[5pt]
\begin{cases}
\evDout_n :=\{r=\infty\} \times\bigl(\frac{2n\pi}{p}- \frac{\pi}{2p}, \frac{2n\pi}{p}+\frac{\pi}{2p}\bigr),\\[8pt]
\oddDout_n :=\{r=\infty\} \times\bigl(\frac{(2n+1)\pi}{p}- \frac{\pi}{2p}, \frac{(2n+1)\pi}{p}+\frac{\pi}{2p}\bigr),
\end{cases}
&n=0,\dots,p-1.\label{eq:evDeltaout}
\end{align}

\begin{figure}[htb]
\begin{minipage}{.5\textwidth}
\begin{center}\setlength{\unitlength}{.57mm}
\begin{picture}(80,70)
\put(0,0){\includegraphics[scale=0.23]{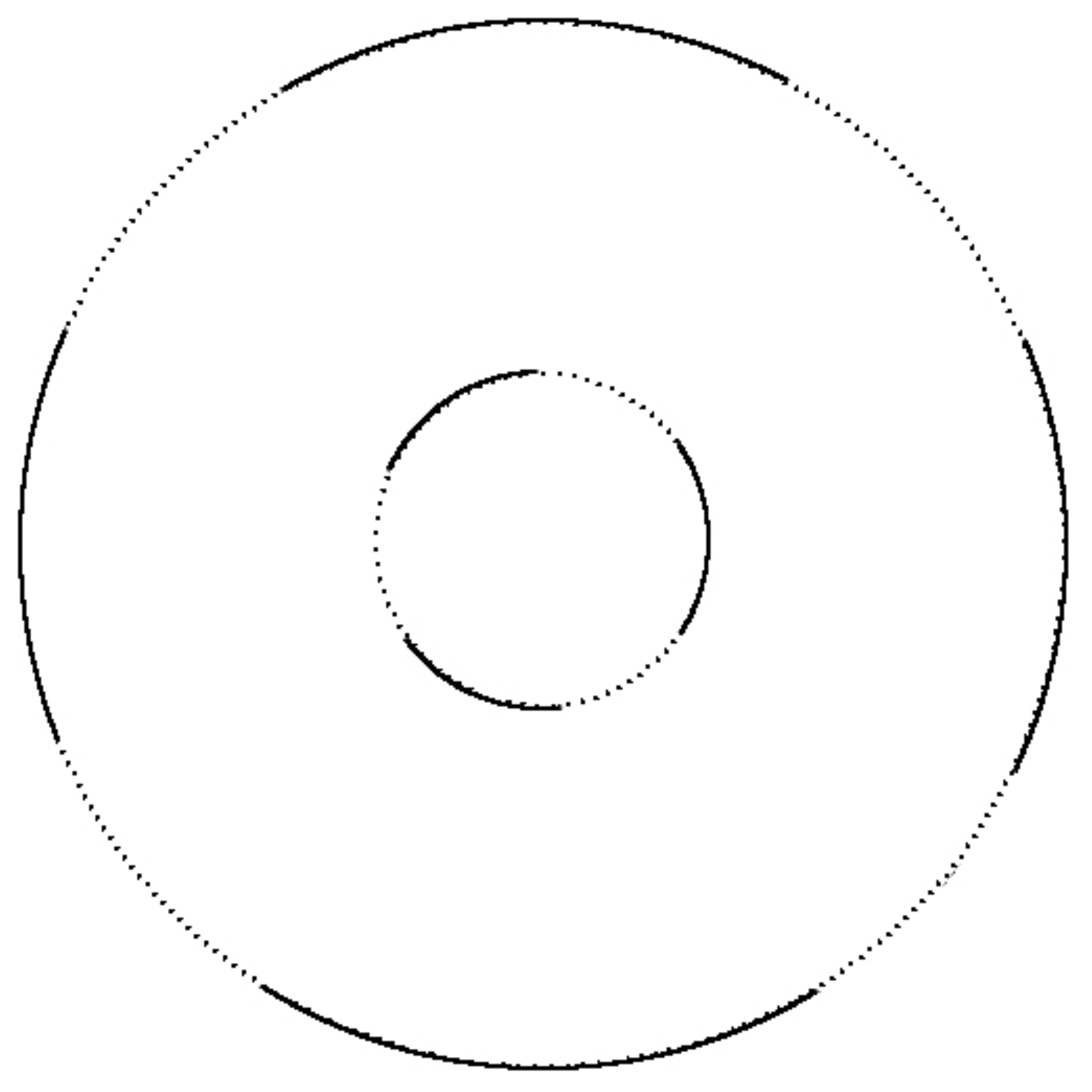}}
\put(73,30){\footnotesize $\evDout_0$}
\put(65,60){\footnotesize $\oddDout_0$}
\put(-4,70){\footnotesize $\evDout_1$}
\qbezier(10,72)(15,72)(23,70)
\put(-10,58){\footnotesize $\oddDout_1$}
\qbezier(3,62)(5,64)(11,62)

\put(49,35){\footnotesize $\evDin_0$}
\put(40,48){\footnotesize $\oddDin_0$}
\put(19,48){\footnotesize $\evDin_1$}
\end{picture}
\caption{The intervals $\evDin_m$, $\oddDin_m$, $\evDout_n$, $\oddDout_n$ in the chosen numbering.}\label{figure:fibreDelta}
\end{center}
\end{minipage}
\hfill
\begin{minipage}{.5\textwidth}
\begin{center}
\includegraphics[scale=0.5]{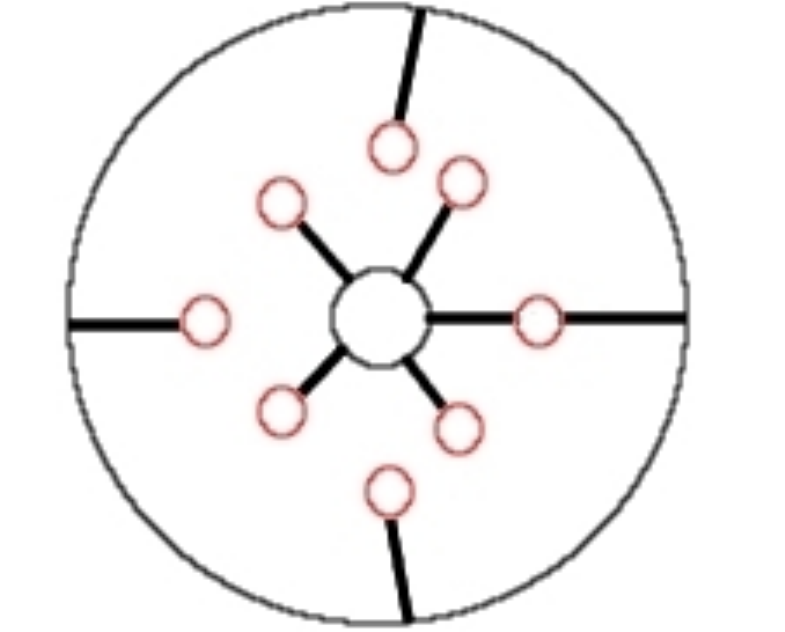}
\caption{The topological space $G^{-1}(0)$. Each of the red circles stands for a (deleted) root of $f(x)$. The number of curves starting at $\SS^1_{r=0}$ is $q$, the number of curves ending in $\SS^1_{r=\infty}$ is $p$.}\label{figurestarthtpy}
\end{center}
\end{minipage}
\end{figure}

We then have\enlargethispage{-\baselineskip}%
\[
\bin G^{-1}\bigl((-\tfrac{\pi}{2},\tfrac{\pi}{2})\bigr)=\bigcup_{m=0}^{q-1}\evDin_m,\qquad
\bout G^{-1}\bigl((-\tfrac{\pi}{2},\tfrac{\pi}{2})\bigr)=\bigcup_{n=0}^{p-1}\evDout_n.
\]

\begin{proof}[Proof of Lemma \ref{lem:G1}]
By Morse theory (\cf\loccit), $G^{-1}\bigl((-\frac{\pi}{2},\frac{\pi}{2})\bigr)$ is diffeomorphic to the product $G^{-1}(0)\times(-\frac{\pi}{2},\frac{\pi}{2})$. It is therefore enough to check that $G^{-1}(0)$ is the disjoint union of $(p+q)$ semi-closed intervals with closed end points on $\partial A$. Since the closure $\ov{G^{-1}(0)}$ cuts the boundary $\partial\wA{}'$ transversally, we conclude that
\begin{itemize}
\item
in the neighbourhood of $\evDin_m$, $G^{-1}(0)$ is a smooth curve abutting to \hbox{$(0,\frac{2m\pi}{q})$} ($m=0,\dots,q-1$),
\item
in the neighbourhood of $\evDout_n$, $G^{-1}(0)$ is a smooth curve abutting to \hbox{$(\infty,\frac{2n\pi}{p})$} ($n=0,\dots,p-1$),
\item
in the neighbourhood of a simple root of $f$, $G^{-1}(0)$ is a smooth curve with one end in this puncture,
\item
in the neighbourhood of the double root of $f$, $G^{-1}(0)$ is a smooth curve with two ends in this puncture.
\end{itemize}

The lemma is then a consequence of the following properties.

\pagebreak[2]
\begin{claims*}[See Figure \ref{figurestarthtpy}]\mbox{}
\begin{enumerate}
\item
The connected component of $G^{-1}(0)$ abutting to \hbox{$(0,\frac{2m\pi}{q})$} ($m=1,\dots,q-1$) is a semi-closed interval with its open end at some simple root of $f$ with absolute value~$<1$.
\item
The connected component of $G^{-1}(0)$ abutting to \hbox{$(\infty,\frac{2n\pi}{p})$} ($n=1,\dots,p-1$) is a semi-closed interval with its open end at some simple root of $f$ with absolute value~$>1$.
\item
The connected components of $G^{-1}(0)$ abutting to $(0,0)$ and $(\infty,0)$ are semi-closed intervals with open end at $x=1$.
\end{enumerate}
\end{claims*}

\subsubsection*{Proof of the claims}
Writing $x=re^{i\vt_x}$ with $r\in(0,\infty)$, one has
\begin{equation}\label{eq:sinsin}
G(x)\in\{0, \pi\}\Longleftrightarrow \frac{f(x)}{x^q}\in\R \smallsetminus \{0\}
\Longleftrightarrow q r^{p+q} \sin(p \vt_x) - p \sin(q \vt_x)=0.
\end{equation}
We consider a fixed \angle $\vt_x$. If $\sin(q \vt_x)=0=\sin(p \vt_x)$ -- which is the case for $\vt_x=0$ or $\vt_x=\pi $ only, since we assume $p,q$ to be co-prime -- then all $r>0$ are a solutions to the last equation. Since $p$ or $q$ is odd, $f(x)/x^q=qx^p-(p+q)+px^{-q}$ is negative for $x<0$, so $(0,\infty)\times\{\vt_x=\pi\}$ is contained in $G^{-1}(\pi)$. On the other hand, $f(x)/x^q$ is $>0$ on $[(0,1)\cup(1,\infty)]\times\{\vt_x=0\}$, hence the third claim.

In case $\sin(p\vt_x)\neq 0$, the equation reads
\begin{equation}\label{eq:sinsign}
r^{p+q}=\frac{p}{q} \cdot \frac{\sin(q\vt_x)}{\sin(p\vt_x)},
\end{equation}
and has exactly one positive solution $r>0$ whenever $\sin(p\vt_x)$ and $\sin(q \vt_x)$ have the same sign, and no solution otherwise.

Now, for fixed $r $, we see that for small $r \ll 1$, Equation \eqref{eq:sinsin} has exactly $q$ solutions~$\vt_x$ contributing to $G^{-1}(0)$, given in the limit $r \searrow 0$ by the \angles \hbox{$\{\frac{2m \pi}{q} \mid m=0,\dots q-1\}$}. For large $r \gg 1$, it has exactly $p$ solutions~$\vt_x$ contributing to $G^{-1}(0)$, given in the limit $r \nearrow\infty $ by \hbox{$\{\frac{2n \pi}{p} \mid n=0,\dots, p-1\}$}.

Since $0$ and $\pi $ are the only common zeroes of $\sin(q \vt_x)$ or $\sin(p\vt_x)$, \ie at each other zero of one of these functions, the corresponding function changes its sign whereas the other one doesn't. This means that locally at these \angles there is an open sector (with the \angle as a boundary of it) where equation \eqref{eq:sinsin} does not admit any solution $r>0$. As a consequence, between each of the \angles
$$
\Big\{\frac{2m\pi}{q} \mid m=0,\dots, q-1\Big\} \cup \Big\{\frac{2n\pi}{p} \mid n=0,\dots, p-1\Big\}
$$
there is an open sector such that equation \eqref{eq:sinsin} has no solution $r>0$ for any \angle $\vt_x$ inside this sector.

Lastly, for $|x|=1$, we get $G(x)=0$ if and only if $p \sin(q \vt_x)=q \sin(p \vt_x)$ and $q \cos(p\vt_x)+p \cos(q \vt_x) -(p+q) > 0
$, which cannot occur. Therefore
\begin{equation}\label{eq:S1in}
\{x=e^{i \vt_x} \mid \vt_x\neq 0\}\cap G^{-1}(0)=\emptyset,
\end{equation}
keeping in mind that $x=1$ is a root of $f(x)$.

Therefore, a connected component of $G^{-1}(0)$ meeting $\{r=0\}$ (\resp $\{r=\nobreak\infty\}$) at a nonzero argument cannot end at another similar point, either because it cannot cross a forbidden sector, so cannot end at $\{r=0\}$ (\resp $\{r=\infty\}$), or because it cannot cross the circle $r=1$, so cannot end at $\{r=\infty\}$ (\resp $\{r=0\}$). Such a component must then have an end at a simple root of $f$. This gives the first two claims, according to the last part of Lemma \ref{lem:indeterminacykappa}.
\end{proof}

\subsubsection*{End of the proof of Proposition \ref{prop:G1}}
The proof will be done by induction on the number of critical points contained in $B^\theta:=G^{-1}\bigl(\theta+\rdInt\bigr)$, and Lemma \ref{lem:G1} provides the initial step of the induction. We now consider the subset $B^\theta$ for $\theta$ varying from $\pi$ to $2\pi$.
\begin{figure}[htb]
\centerline{An example with $p=4$ and $q=5$.}\par\smallskip
\begin{minipage}[t]{.46\textwidth}
\begin{center}\setlength{\unitlength}{.9mm}
\includegraphics[scale=0.4]{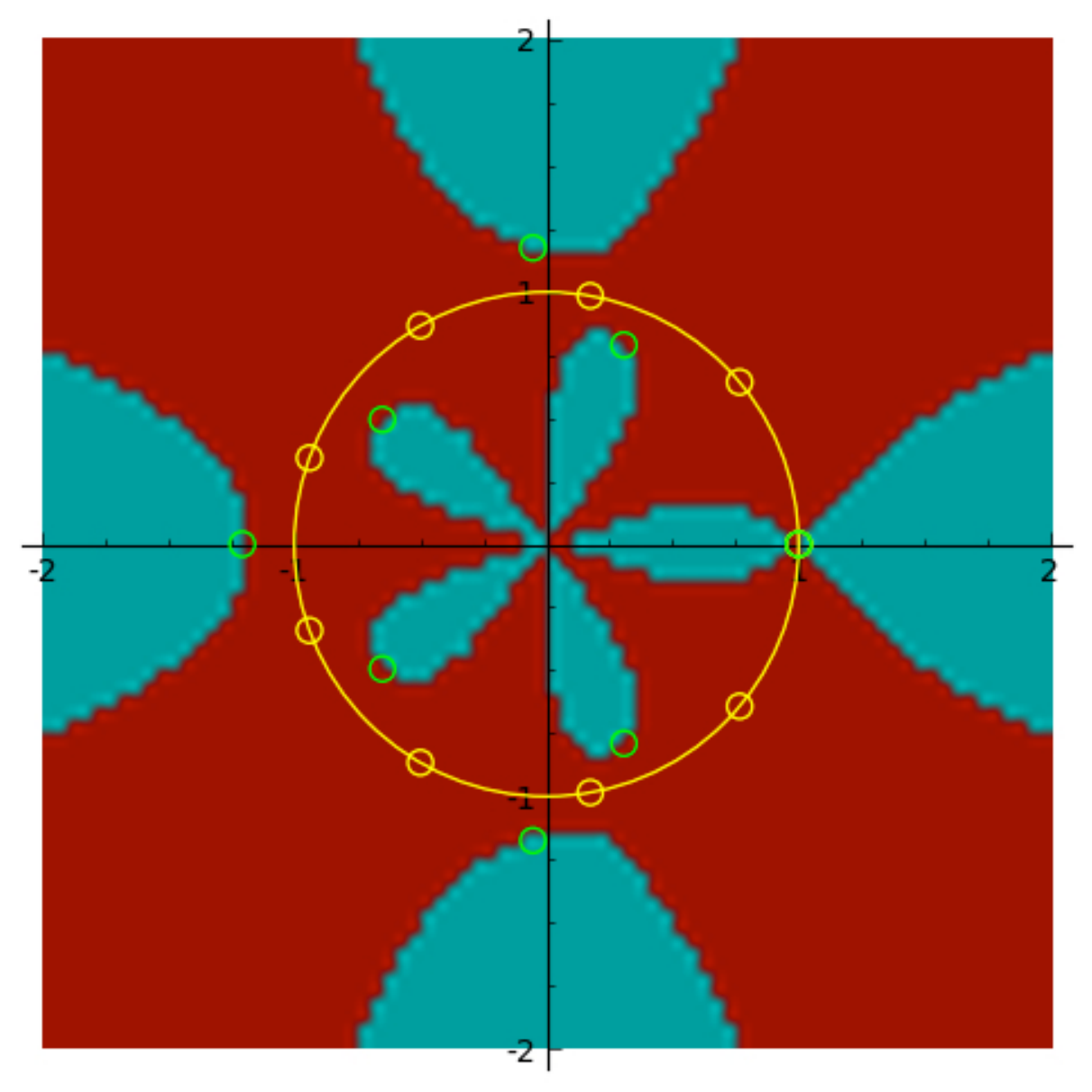}
\caption{The subset $B^\pi$ for the starting value $\pi$ of $\theta$. The green circles are the zeroes of $f(x)$. The small yellow circles are the roots of unity $\mu_{p+q}$, \ie the critical points of the Morse function.}\label{figurerd1}
\end{center}
\end{minipage}
\hfill
\begin{minipage}[t]{.46\textwidth}
\begin{center}\setlength{\unitlength}{.9mm}
\includegraphics[scale=0.4]{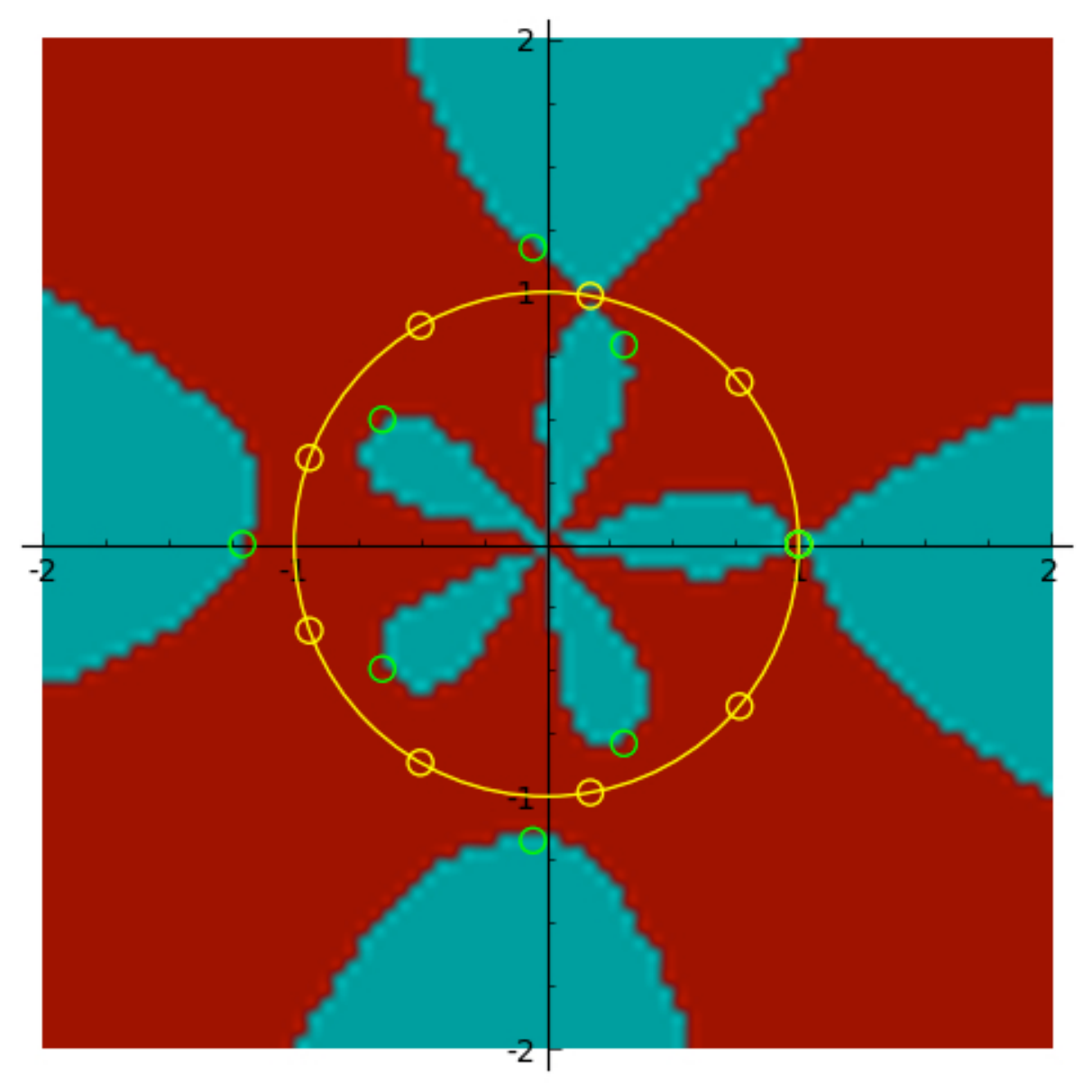}
\caption{The subset $B^\theta$ when the first critical point enters the picture (\hbox{$\theta_1=\pi+\frac{\pi}{p+q}$}). We see the gluing of the corresponding discs in the upper half of the picture.}\label{figurerd2}
\end{center}
\end{minipage}
\end{figure}
The case from $\pi$ to $0$ is very similar. When $\theta$ varies, no critical point can go out from $G^{-1}(\theta+\frac{\pi}{2})$, but a critical point may enter in $G^{-1}(\theta+\frac{3\pi}{2})$. This occurs successively for
\[
\theta=\theta_\kappa=\pi+\frac{\kappa\pi}{p+q},\quad\kappa=1,\dots,p+q-1,
\]
with corresponding critical point given by Lemma \ref{lem:claim1}. Attaching a $1$-cell at such a~$\theta$ consists in connecting a connected component of $B^{\theta-\ve}$ intersecting $\{r<1\}$ with one intersecting $\{r>1\}$.

The behaviour of $B^\theta$ when $\theta$ varies between $\theta_\kappa-\ve$ and $\theta_\kappa+\ve$ ($\ve>\nobreak0$ small) is pictured in Figure \ref{fig:B2thetacrit}, where the dotted lines are the part of the boundary which do not belong to $B^\theta$.
\begin{figure}[htb]
\begin{center}
\includegraphics[scale=.2]{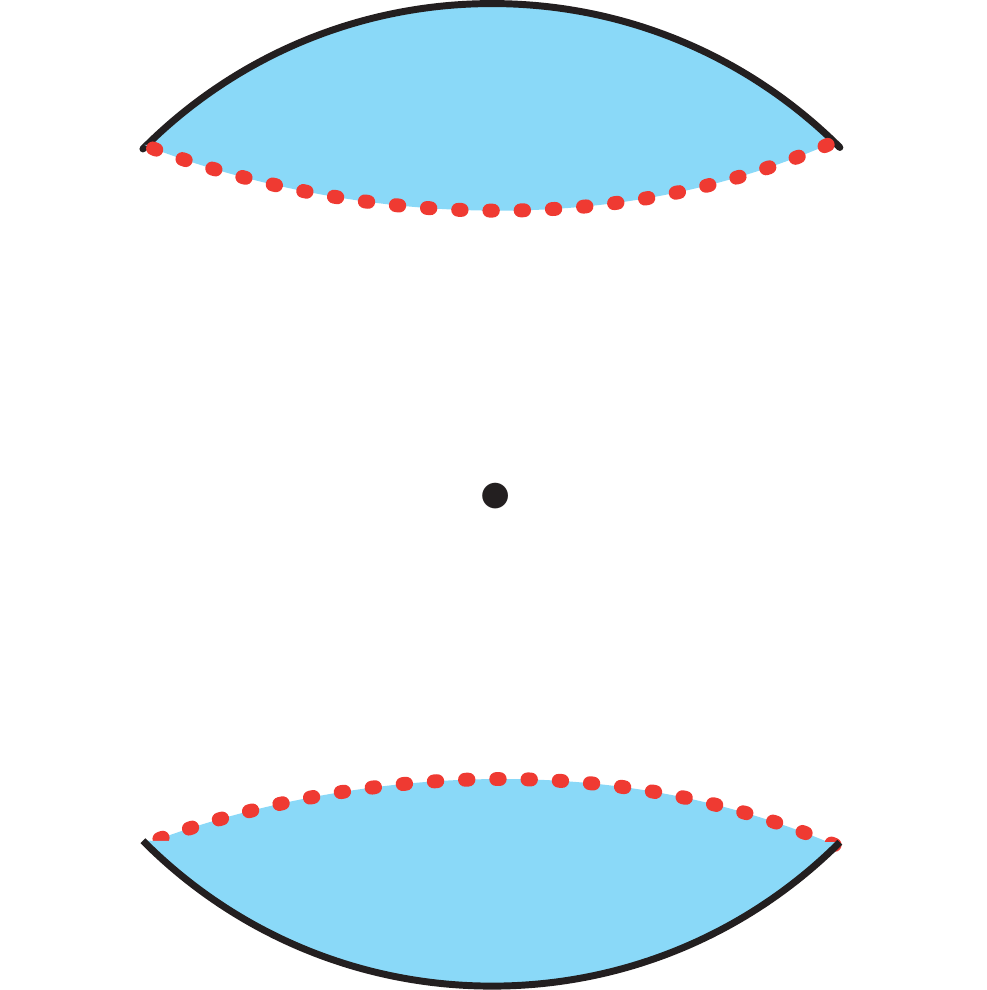}\hspace*{2cm}\includegraphics[scale=.2]{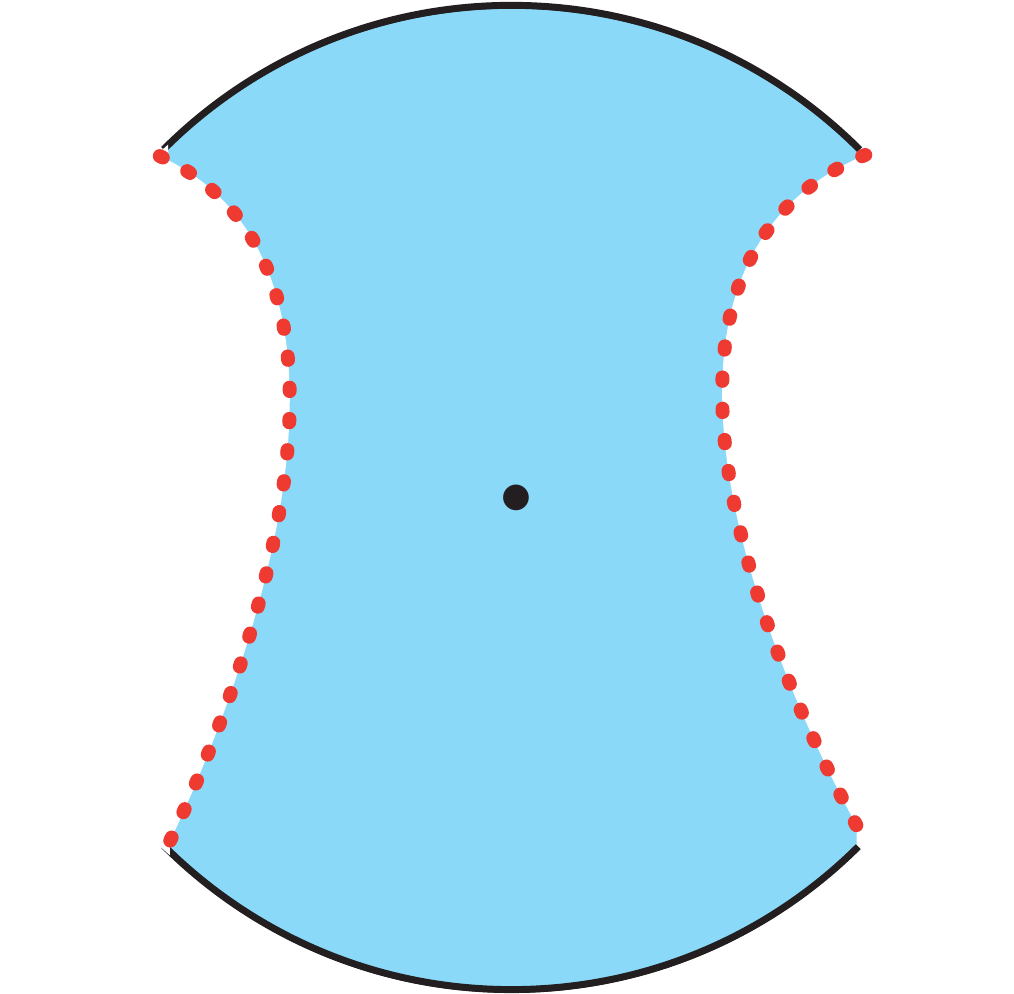}
\caption{Local view at the critical point $\wzeta=\exp(\frac{2k\pi i}{p+q})$ of the set $B^\theta$ at $\theta= \theta_\kappa-\ve$ and $\theta=\theta_\kappa+\ve$, with dotted lines deleted.}\label{fig:B2thetacrit}
\end{center}
\end{figure}
Indeed, for $\theta<\theta_\kappa$, the boundary of $B^\theta$ is contained in $\partial\wA$ by induction, hence far from the critical point and, by Morse theoretic arguments, the properties of~$B^\theta$ remain constant away from a neighbourhood of the corresponding critical point. The assertion of the proposition is thus also satisfied for $\theta\in[\theta_\kappa,\theta_\kappa+\ve)$.
\end{proof}

\begin{corollaire}\label{cor:Geta}
Let $\cK$ be a local system on $\wA$. Then, for each $\theta\!\in\!\SS^1$, $H^j(\wA, \beta^\theta_!\cK)\!=\!\nobreak0$ for $j\neq1$ and
\[
H^1(\wA, \beta^\theta_!\cK)\simeq\bigoplus_{\wzeta\in B^\theta\cap\mu^{p+q}}\cK_{\wzeta}.
\]
\end{corollaire}

\begin{proof}
For the first assertion, we note that the topological boundary of each connected component of $B^\theta$ intersects the component but is not contained in it, according to Proposition \ref{prop:G1}. Therefore, $\beta^\theta_!\cK$ has no global section, nor has its dual sheaf, hence the conclusion by using Poincaré-Verdier duality.

The second assertion is proved by induction as in Proposition \ref{prop:G1} on the number of critical points contained in $B^\theta$. The case $\theta=\pi$ follows from Lemma \ref{lem:G1} and \cite[Lem.\,7.13]{Bibi10}. When $\theta$ is in the neighbourhood of a critical value $\theta_\kappa$, we decompose $B^\theta$ in two closed subsets $B_1^\theta\cup B_2^\theta$, so that $B_2^{\theta_\kappa-\ve}$ \resp $B_2^{\theta_\kappa+\ve}$ are as in Figure \ref{fig:B2thetacrit} and $B_1^{\theta_\kappa-\ve}$ is diffeomorphic to $B_1^{\theta_\kappa+\ve}$ by a Morse-theoretic argument. We then have an exact sequence, according to the first part and with obvious notation,
\[
0\to H^1(\wA, \beta^\theta_!\cK)\to H^1(\wA, \beta^\theta_{1,!}\cK)\oplus H^1(\wA, \beta^\theta_{2,!}\cK)\to H^1(\wA, \beta^\theta_{12,!}\cK)\to0
\]
and $H^1(\wA, \beta^\theta_{1,!}\cK)$ remains constant, as well as $H^1(\wA, \beta^\theta_{12,!}\cK)$, when $\theta$ varies around~$\theta_\kappa$. The only changes come from $H^1(\wA, \beta^\theta_{2,!}\cK)$, which is zero for $\theta\in(\theta_\kappa-\ve,\theta_\kappa]$ and has dimension $\rk\cK$ for $\theta\in(\theta_\kappa,\theta_\kappa+\ve)$.
\end{proof}

Going back to the sets $B_{\omega_{\wzeta}>0}^{\vtn}$, we conclude:

\begin{corollaire}\label{cor:Geta2}
Let $\cK$ be a local system on the annulus $\wA$. Then, for each $\vtn\in\SS^1_{\eta=0}$ and $\wzeta\in\mu_{p+q}$, we have
\[
H^1\bigl(\wA,(\beta^{\vtn}_{\omega_{\wzeta}>0})_!\cK\bigr)\simeq\bigoplus_{\substack{\wzeta'\in\mu_{p+q}\moins\{1\}\\\wzeta'\in B_{\omega_{\wzeta}>0}^{\vtn}}}\cK_{\wzeta'},
\]
where we regard $\mu_{p+q}$ as embedded in $\wA=[0,\infty]\times\SS^1$ as $\{(1,\vt)\mid \vt^{p+q}=1\}$.\qed
\end{corollaire}

\subsection{The family \texorpdfstring{$(B_{\omega_{\wzeta}>0}^{\vtn_\ell})$}{Bvtn} for \texorpdfstring{$\wzeta\in\mu_{p+q}$}{wzeta} and \texorpdfstring{$\ell=0,\dots,2q-1$}{ell}}\label{subsec:Bwzeta}

We now make precise the choice of an \angle $\vtn_o\in\SS^1_{\eta=0}$, and so the set \hbox{$\{\vtn_\ell\!:=\!\vtn_o+\frac{\ell\pi}{q} \mid \ell\!=\!0,\dots,2q-1\}$}. Due to the previous results, the choice of a base point $\vtn'_o\in\SS^1_{\eta=0}$ such that $q \vtn'_o-\nobreak\arg(\vi_q)=\pi$ seems appropriate. However, since $\vtn'_o$ is a Stokes direction for some pair in $(\wvi_{\wzeta})_{\wzeta\in\mu_{p+q}}$, we modify it and set $\vtn_o :=\vtn'_o+\sfrac{\ve}{q}$ for $\ve>0$ small enough in such a way that $\vtn_o $ is not a Stokes direction. We thus have $q \vtn_o-\arg(\vi_q)=\pi+\ve$.

We now fix $\ell\in\{0,\dots,2q-1\}$. The order on the family $(\wvi_{\wzeta})_{\wzeta}$ at $\vtn_\ell$ is defined by $\wvi_{\wzeta}\le_{\vtn_\ell}\wvi_{\wzeta'}\iff \reel(\wvi_{\wzeta}-\wvi_{\wzeta'})<0$ in some neighbourhood of $\vtn_\ell$, which amounts to $\reel\bigl(e^{-\ve i}(\wzeta^p-\wzeta^{\prime p})\bigr)<0$ for $\ell$ odd, and the reverse order for $\ell$ even, according to Formula \eqref{eq:wvieta} and to the relation $q\vtn_\ell-\arg(\vi_q)=(\ell+1)\pi+\ve$. Therefore it corresponds to the even (\resp odd) order of Definition \ref{def:standardorder} on $\mu_{p+q}$ when $\ell$ is even (\resp odd). Recall that, with respect to this order, the family indexed by $\wzeta$:
\[
B_{\omega_{\wzeta}>0}^{\vtn_\ell}=\Big\{x\in\wA\mid G(\wzeta^{-1}x)\in q\arg\wzeta+q\vtn_\ell-\arg(\vi_q)+ \rdInt\Big\}
\]
is increasing. One can also check directly with the formula given in Lemma \ref{lem:claim1} that the order on $\mu_{p+q}$ at $\vtn_\ell$ coincides with the order by the number of critical points contained in $B_{\omega_{\wzeta}>0}^{\vtn_\ell}$. We thus have
\begin{equation}\label{eq:Bfiltre}
B_{\omega_{\wzeta}>0}^{\vtn_\ell}=\begin{cases}
\evB_{\omega_{\wzeta}>0}:=\wzeta\cdot G^{-1}\Bigl(q\arg\wzeta +\ve +\Bigl({-}\dfrac\pi2,\dfrac\pi2\Bigr)\Bigr)&\text{if $\ell$ is even},\\[5pt]
\oddB_{\omega_{\wzeta}>0}:=\wzeta\cdot G^{-1}\Bigl(q\arg\wzeta +\ve +\Bigl(\dfrac\pi2,\dfrac{3\pi}2\Bigr)\Bigr)&\text{if $\ell$ is odd}.
\end{cases}
\end{equation}

We will set
\begin{equation}\label{eq:evIeps}
\begin{aligned}
\evIin{m}&:=\evDin_m-\sfrac\ve q,&\evIout{n}&:=\evDout_n+\sfrac\ve p,\\[5pt]
\oddIin{m}&:=\oddDin_m-\sfrac\ve q,&\oddIout{n}&:=\oddDout_n+\sfrac\ve p.
\end{aligned}
\end{equation}
Then we have
\begin{equation}\label{eq:dB}
\bin B^{\vtn_\ell}_{\omega_{\wzeta}>0}=\left\{\begin{array}{lcl}
\bigcup_{m=0}^{q-1}\evIin{m}&\quad\text{($\ell$ even)}\quad&\bigcup_{n=0}^{p-1}\evIout{n}\\[7pt]
\bigcup_{m=0}^{q-1}\oddIin{m}&\quad\text{($\ell$ odd)}\quad&\bigcup_{n=0}^{p-1}\oddIout{n}
\end{array}\right\}
=\bout B^{\vtn_\ell}_{\omega_{\wzeta}>0}.
\end{equation}

In the following, $\ve>0$ will be chosen strictly smaller of any constant quantity like $\pi/(p+q)$, etc.

\section{Topological model for the standard linear Stokes data}\label{sec:standard}

In this section we introduce a ``linear model'' and we develop a computation of the Stokes data in this model. That this model suffices for the computation of the desired Stokes data will be shown in Section \ref{finalsection}.

\subsection{Standard filtration on a local system on the annulus}\label{sec:standtop}

We are given two coprime integers $p,q\in\N $. Recall that we defined two opposite orderings on the set of $(p+q)^{\text{th}}$ roots of unity in Definition \ref{def:standardorder}. We define the following two subsets related to those constructed in \S\ref{subsec:GGpsi} (notation as in \eqref{eq:evDeltain} and \eqref{eq:evDeltaout}):
\begin{equation}\label{eq:evoddbeta}
\begin{aligned}
\Aev&:=A \cup\Bigl(\bigsqcup_{m=0}^{q-1} \evDin_m\sqcup\bigsqcup_{n=0}^{p-1} \evDout_n\Bigr) \Hto{\evbeta} \wA,\\
\Aodd&:=A \cup\Bigl(\bigsqcup_{m=0}^{q-1} \oddDin_m \sqcup\bigsqcup_{n=0}^{p-1} \oddDout_n\Bigr) \Hto{\oddbeta} \wA.
\end{aligned}
\end{equation}
Both spaces are obtained from the closed annulus $\wA$ by deleting $p$ closed intervals in the outer boundary and $q$ closed intervals in the inner boundary. On the other hand, we embed $\mu_{p+q}$ into the circle with $r=1$.

\subsubsection{The even case}

The combinatorics which will play a role in the following is governed by the following map (\cf\eqref{eq:evinout} for the notation):
\begin{equation}
\label{eq:cev}
\begin{split}
\{0,\dots, p+q-1\}&\to\{0,\dots,q\}\times\{0,\dots,p\}\\
k&\mto\bigl(\evin(k),\evout(k)\bigr).
\end{split}
\end{equation}
We will also consider the composition $\cev$ with the projection to $\ZZ/q\ZZ\times\ZZ/p\ZZ$, and we will denote by $m$ (\resp $n$) the representative of $\evin(k)\bmod q$ in $\{0,\dots,q-1\}$ (\resp of $\evout(k)\bmod p$ in $\{0,\dots,p-1\}$), so we also regard $\cev:k\mto(m,n)$ as a map $\{0,\dots, p+q-1\}\to\{0,\dots,q-1\}\times\{0,\dots,p-1\}$.

\begin{remarques}\mbox{}\label{rem:cev}
\begin{enumerate}
\item\label{rem:cev1}
The map \eqref{eq:cev} clearly injective. Moreover, if $p$ and $q$ are $\neq1$, $\cev$ is also injective. Indeed, let $k_1,k_2$ be such that $(m_1,n_1)=(m_2,n_2)$.
\begin{itemize}
\item
If $m_1=m_2\neq0$ and $n_1=n_2\neq0$, then $\evin(k_1)=\evin(k_2)$, $\evout(k_1)=\evout(k_2)$, and $\evout(k_1)=k_1-\evin(k_1)$, $\evout(k_2)=k_2-\evin(k_2)$, so $k_1=k_2$.
\item
If $m_1=m_2=0$ and $n_1=n_2\neq0$, then $\evout(k_1)=\evout(k_2)=:n\in\{1,\dots,p-1\}$ and $\evin(k_1)=k_1-n$, $\evin(k_2)=k_2-n$. It is thus enough to consider the case $\evin(k_1)=0$ and $\evin(k_2)=q$, which is then equivalent to $k_1=n$, $k_2=n+q$, $\frac{n}{p+q}<\frac{1}{2q}$ and $\frac{n+q}{p+q}\geq1-\frac{1}{2q}$. This implies $p+q>pq$, which can occur only if $q=1$ (since $p\neq1$ with our starting assumption).
\item
The case $m_1=m_2\neq0$ and $n_1=n_2=0$ is treated similarly.
\item
Assume $m_1=m_2=0$ and $n_1=n_2=0$.
\begin{itemize}
\item
If $\evin(k_1)=\evin(k_2)=0$, the only case to consider is $\evout(k_1)=0$ and $\evout(k_2)=p$, which implies $p+q>2pq$ and cannot occur.
\item
If $\evin(k_1)=\evin(k_2)=q$, we have $k_1=\evout(k_1)+q$ and $k_2=\evout(k_2)+q$ with $n_i=0$ or $p$, but $n_i=p$ is not possible since $k_i<p+q$, so $\evout(k_1)=\evout(k_2)=0$, and therefore $k_1=k_2$.
\item
If $\evin(k_1)=0$ and $\evin(k_2)=q$, then $k_1=\evout(k_1)$, $k_2=\evout(k_2)+q$, so the only problematic cases are $k_1=0,p$, $\evout(k_2)=0$ and $k_2=q$. But $\evout(k_2)=0$ and $k_2=q$ imply $qp/(p+q)<1/2$, which cannot occur.
\end{itemize}
\item
If $q=1$ and $p$ is even (\resp odd) the pairs $k_1=p/2,k_2=(p+2)/2$ (\resp $k_1=(p-1)/2,k_2=(p+1)/2$) give rise to the same pair $(m,n)=(0,p/2)$ (\resp $(m,n)=(0,(p-1)/2)$). The case $p=1$ is similar.
\end{itemize}

\item\label{rem:cev2}
The map $\cev$ arises in the following way: $m$ is the unique element in $\{0,\dots,q-1\}$ such that the rotation of $\evDin_0$ by $\xi^k$ has non-empty intersection with $\evIin{m}$ as defined by \eqref{eq:evIeps}:
\[
\Bigl(\frac{2k\pi}{p+q}+\evDin_0\Bigr)\cap\evIin{m}\neq\emptyset\quad\forall\ve>0\text{ small enough}.
\]
Equivalently, this condition decomposes as
\[
\Bigl(\frac{2k\pi}{p+q}+\evDin_0\Bigr)\cap\evDin_m
\begin{cases}
\neq\emptyset,\text{ or}\\
=\emptyset\text{ and }\frac{2k\pi}{p+q}+\frac{\pi}{2q}=\frac{2m\pi}q-\frac{\pi}{2q}\bmod2\pi,
\end{cases}
\]
and the latter case can occur only if $p+q$ is even. Similarly, $n$ is the unique element of $\{0,\dots,p-1\}$ such that
\[
\Bigl(\frac{2k\pi}{p+q}+\evDout_0\Bigr)\cap\evIout{n}\neq\emptyset\quad\forall\ve>0\text{ small enough}.
\]
\item\label{rem:cev3}
It will be suggestive to write $k\mto(m, n)$ as
\[
\evconn{m}{k}{n}.
\]
We will construct a family of curves $\evgamma_k$ in Proposition \ref{prop:standtopmodel} below such that $\evgamma_k$ connects $\evDin_m $ and $\evDout_n $ passing through $\xi^k$. If $q\geq2$ we can characterize the index $\evminin(k)$ in \eqref{eq:minin} as the index $k'$ such that $\evgamma_{k'}$ starts from $\evDin_m$ but not $\evgamma_{k'-1}$ (where $k'$ is understood $\bmod (p+q)$) -- and similarly for $\evmaxout(k)$.
\item
\emph{In the remaining part of this section, we will assume that $p>q$}. Otherwise, the same arguments hold by interchanging the roles of the inner and outer circle of the annulus $\wA=[0,\infty]\times\SS^1$ with coordinates $(r,\vt)$ in the following constructions.
\end{enumerate}
\end{remarques}

The following lemma is straightforward and is illustrated by Figure \ref{fig:leray}.

\begin{lemme}\label{lem:3xis}
Assuming that $\evconn{m}{k}{n}$
one of the following holds:
$$
\begin{array}{rlcl}
\text{i)} & \evconn{m-1}{(k-1)}{n} & \text{and} & \evconn{m}{(k+1)}{n+1}\\[.1cm]
\text{ii)} & \evconn{m}{(k-1)}{n-1} & \text{and} & \evconn{m}{(k+1)}{n+1}\\[.1cm]
\text{iii)} & \evconn{m}{(k-1)}{n-1} & \text{and} & \evconn{m+1}{(k+1)}{n}.\\
\end{array}
$$
Moreover,
$$
\begin{array}{rcl}
\text{i) holds} & \Longleftrightarrow & \frac{k}{p+q}\in\big[\frac mq-\frac1{2q},\frac mq-\frac1{2q}+\frac1{p+q}\bigr), \\[.2cm]
\text{ii) holds} & \Longleftrightarrow & \frac{k}{p+q}\in\big[\frac mq-\frac1{2q}+\frac1{p+q},\frac mq+\frac1{2q}-\frac1{p+q}\bigr),\\[.2cm]
\text{iii) holds} & \Longleftrightarrow & \frac{k}{p+q}\in\big[\frac mq+\frac1{2q}-\frac1{p+q},\frac mq+\frac1{2q}\bigr).\\[.2cm]
\end{array}
$$
In particular, $\#(\cev^\mathrm{out})^{-1}(n)\leq2$ for each $n\in\{0,\dots,p-1\}$.\qed
\end{lemme}

\begin{figure}[htb]
\begin{center}
\begin{picture}(150,60)
\put(10,0){\includegraphics[scale=0.4]{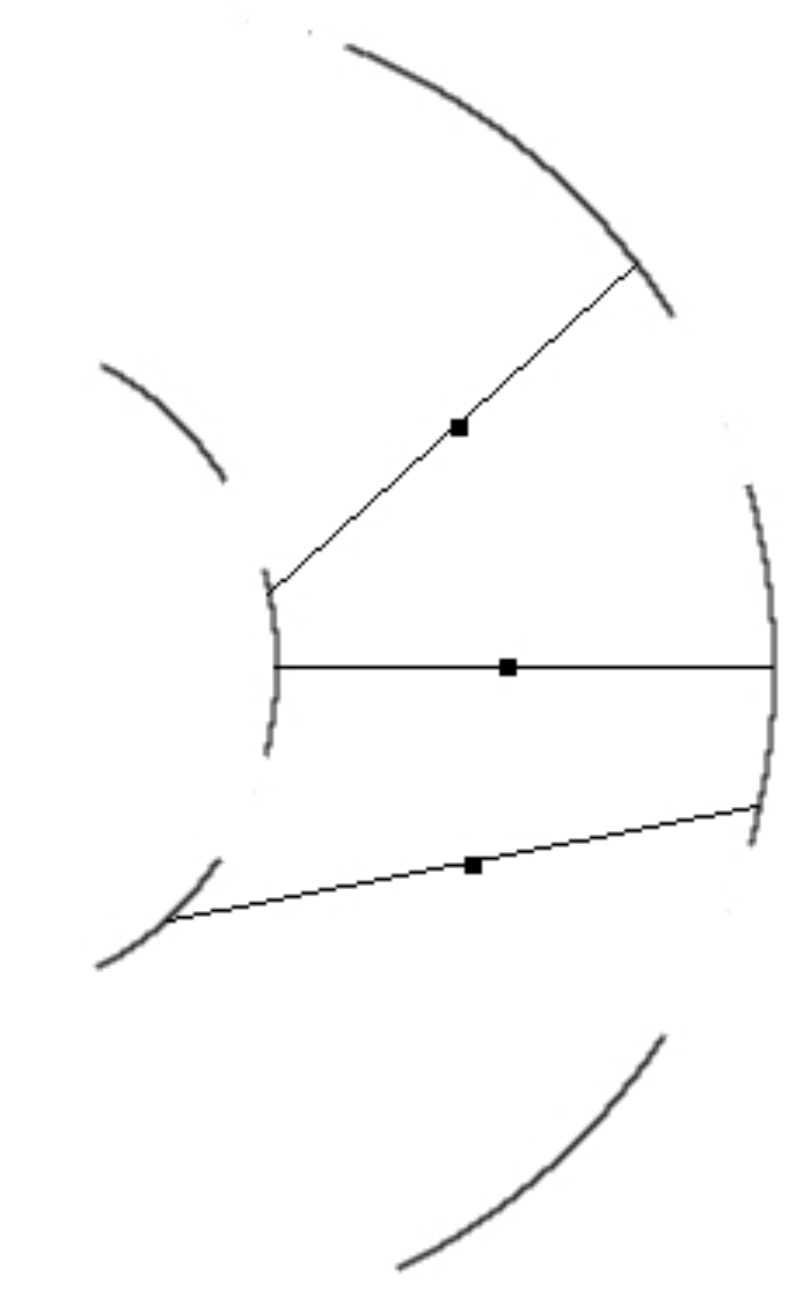}}
\put(50,0){\includegraphics[scale=0.4]{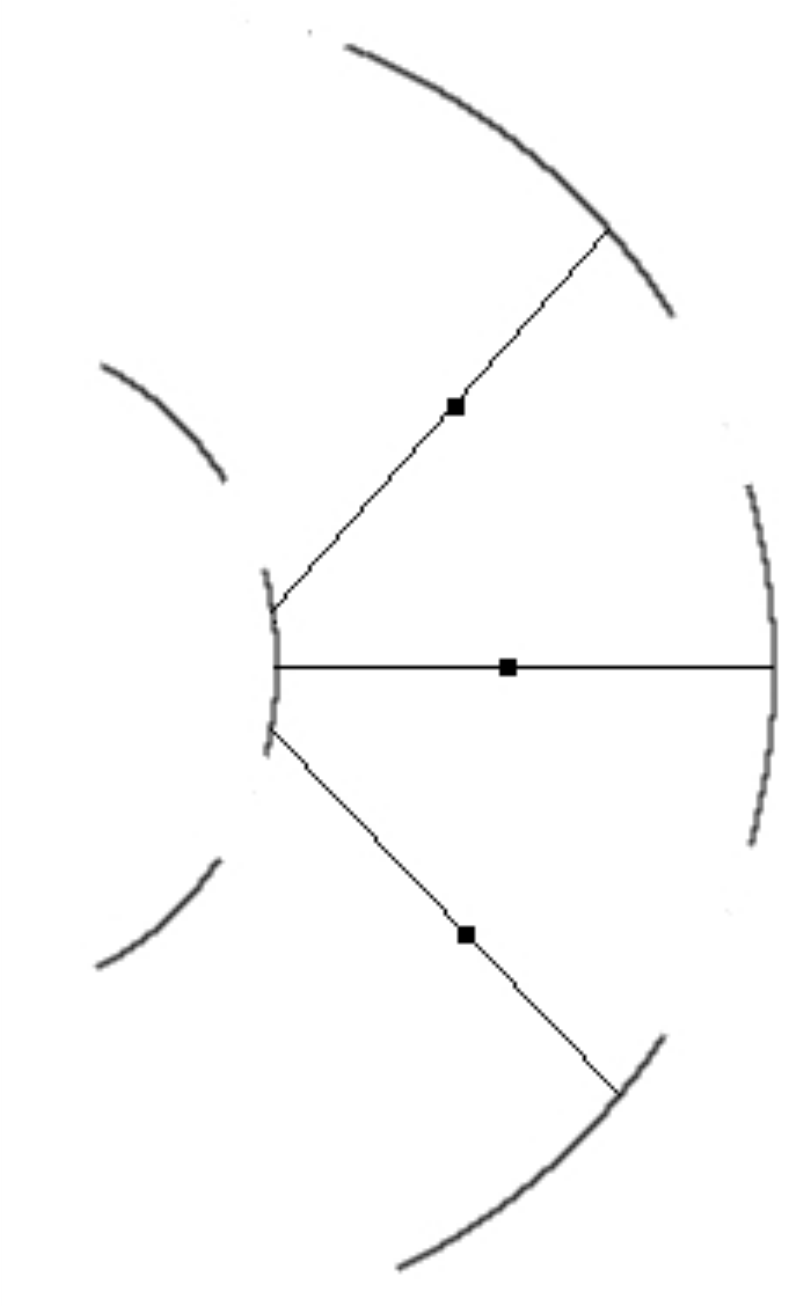}}
\put(90,0){\includegraphics[scale=0.4]{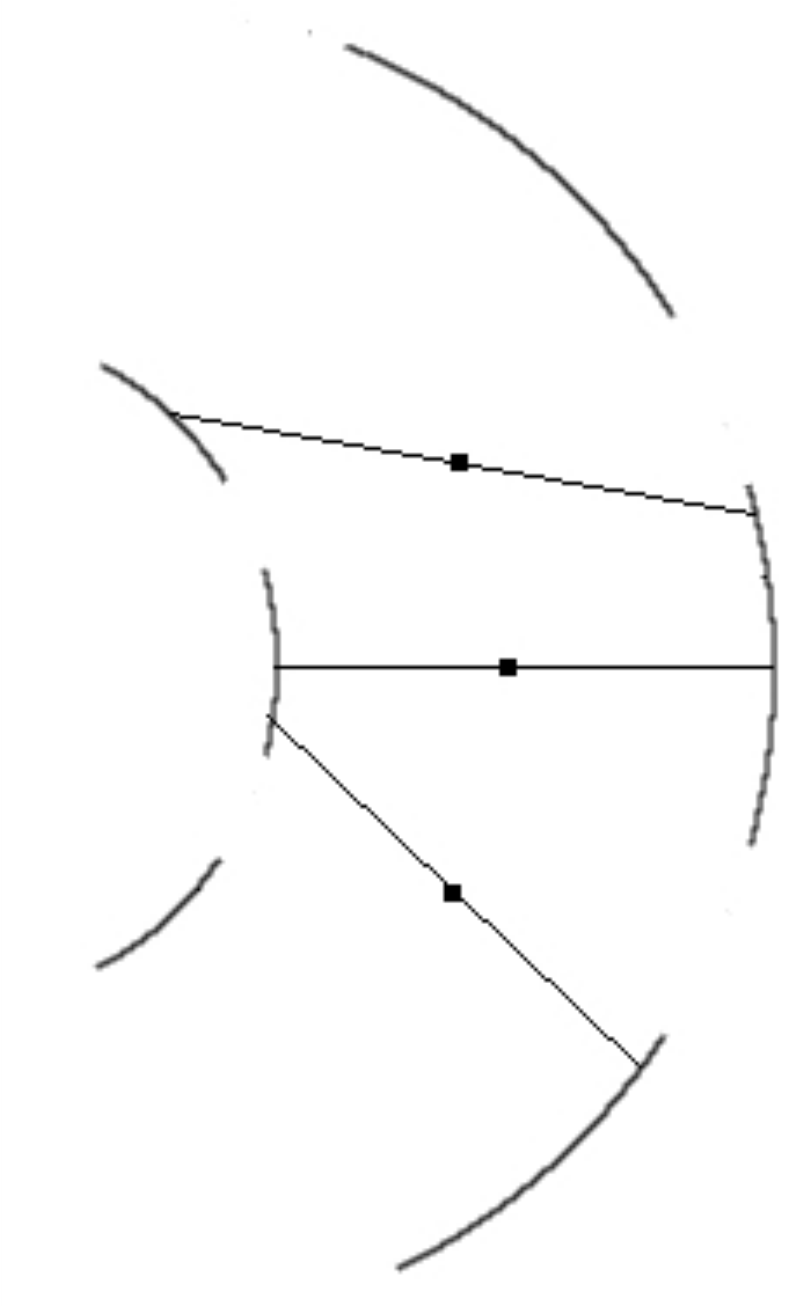}}

\put(15,25){\footnotesize $\evDin_m$}
\put(9,33){\footnotesize $\evDin_{m+1}$}
\put(6,16){\footnotesize $\evDin_{m-1}$}
\put(30,28){\footnotesize $\xi^k$}
\put(28,38){\footnotesize $\xi^{k+1}$}
\put(28,15){\footnotesize $\xi^{k-1}$}
\put(43,25){\footnotesize $\evDout_{n}$}
\put(31,49){\footnotesize $\evDout_{n+1}$}
\put(35,4){\footnotesize $\evDout_{n-1}$}
\end{picture}
\end{center}
\caption{The three possible situations for $\cev(k-1),\cev(k),\cev(k+1)$.}\label{fig:leray}
\end{figure}

In the following we use the convention that $(p+q-1)+1=0$.

\begin{proposition}\label{prop:standtopmodel}
There exists a closed covering $(\evfS_k)_{k=0,\dots,p+q-1}$ of $\wA$ satisfying the following properties:
\begin{enumerate}
\item
Each $\evfS_k$ is the homeomorphic image of the unit square with edges denoted successively by $a,b,c,d$.
\item
The image of $a$ (\resp $c$) is a closed interval in the inner (\resp outer) boundary of $\wA$.
\item\label{prop:standtopmodel3}
For each $k$, there is exactly one of the intervals $\oddDin$ and $\oddDout$ which intersects $\partial(\evfS_k)$, either $\oddDin_m$ or $\oddDout_n$ with $(m,n):=\cev(k)$, and it is entirely contained in $\partial(\evfS_k)$.
\item
The image of $b$ (\resp $d$) is a smooth curve $\evgamma_k$ (\resp $\evgamma_{k+1}$) which satisfies
\begin{enumerate}
\item
$\evgamma_k\cap\partial\wA\subset\partial(\evgamma_k)$,
\item
$\evgamma_k\cap\mu_{p+q}=\{\xi^k\}$,
\item\label{prop:standtopmodel4c}
$\evgamma_k$ connects $\evDin_m$ with $\evDout_{n}$, where $(m,n):=\cev(k)$.
\end{enumerate}
\item
We have $\evfS_k\cap\evfS_{k'}=\begin{cases}\emptyset&\text{if }k'\neq k+1,k-1,\\ \evgamma_k&\text{if }k'=k-1,\\ \evgamma_{k+1}&\text{if }k'=k+1.\end{cases}$
\end{enumerate}
\end{proposition}

\begin{proof}
We will represent $\wA$ as obtained from $[0,\infty]\times[0,2\pi]$ by identifying $[0,\infty]\times\{0\}$ with $[0,\infty]\times\{2\pi\}$ through the identity map. We start by considering the closed covering $(\evfS'_k)$ cut out by the line segments $\evgamma'_k=[0,\infty]\times\{\frac{2k\pi}{p+q}\}$.
\begin{figure}[htb]
\begin{center}\setlength{\unitlength}{.8mm}
\begin{picture}(100,35)
\put(0,2){\includegraphics[width=8cm,height=2.5cm]{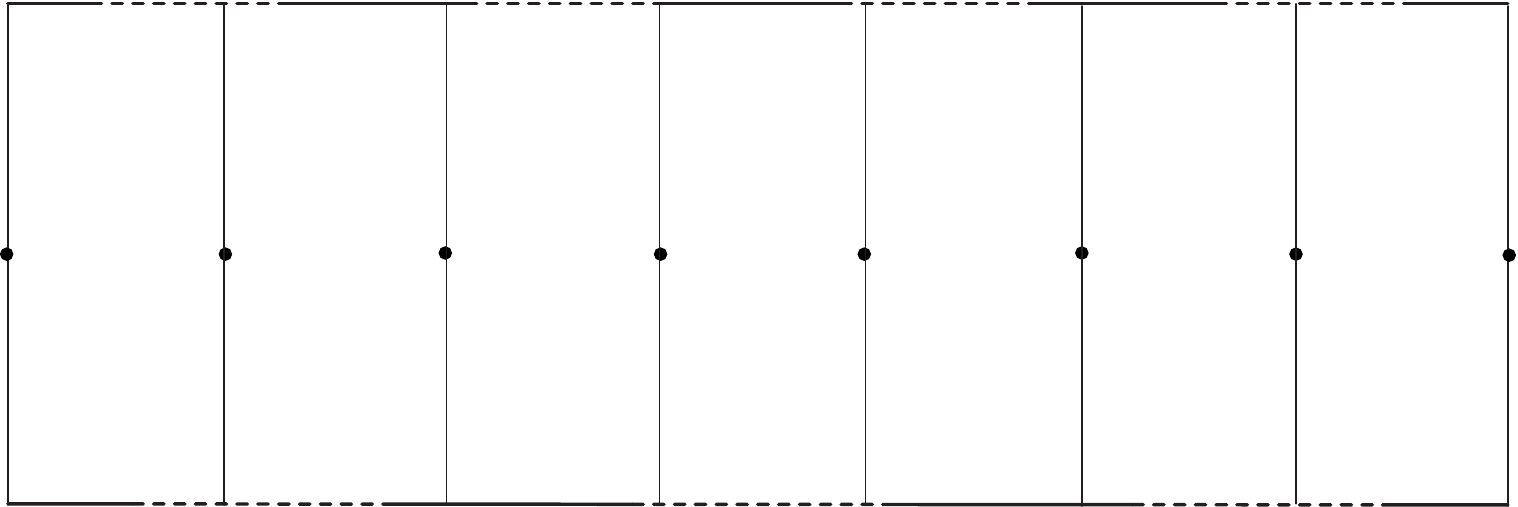}}
\put(2,17){$\xi^0$}
\put(95,17){$\xi^0$}
\put(59,17){$\xi^k$}
\put(34,17){$\xi^{k+1}$}
\put(46,25){$\evfS'_k$}
\put(58,7){$\evgamma'_k$}
\put(31,7){$\evgamma'_{k+1}$}
\put(62,-2){$\evDin_m$}
\put(45,35){$\evDout_n$}
\put(44,-2){$\oddDin_m$}
\put(29,35){$\oddDout_n$}
\end{picture}
\caption{The closed covering $(\evfS'_k)$ (example with $p=4$, $q=3$).}
\end{center}
\end{figure}
However, Properties~\eqref{prop:standtopmodel3} and \eqref{prop:standtopmodel4c} are not satisfied in general by some line segments $\evgamma'_k$ for $k\neq0$.
\begin{figure}[htb]
\begin{center}\setlength{\unitlength}{.8mm}
\begin{picture}(100,35)
\put(0,2){\includegraphics[width=8cm,height=2.5cm]{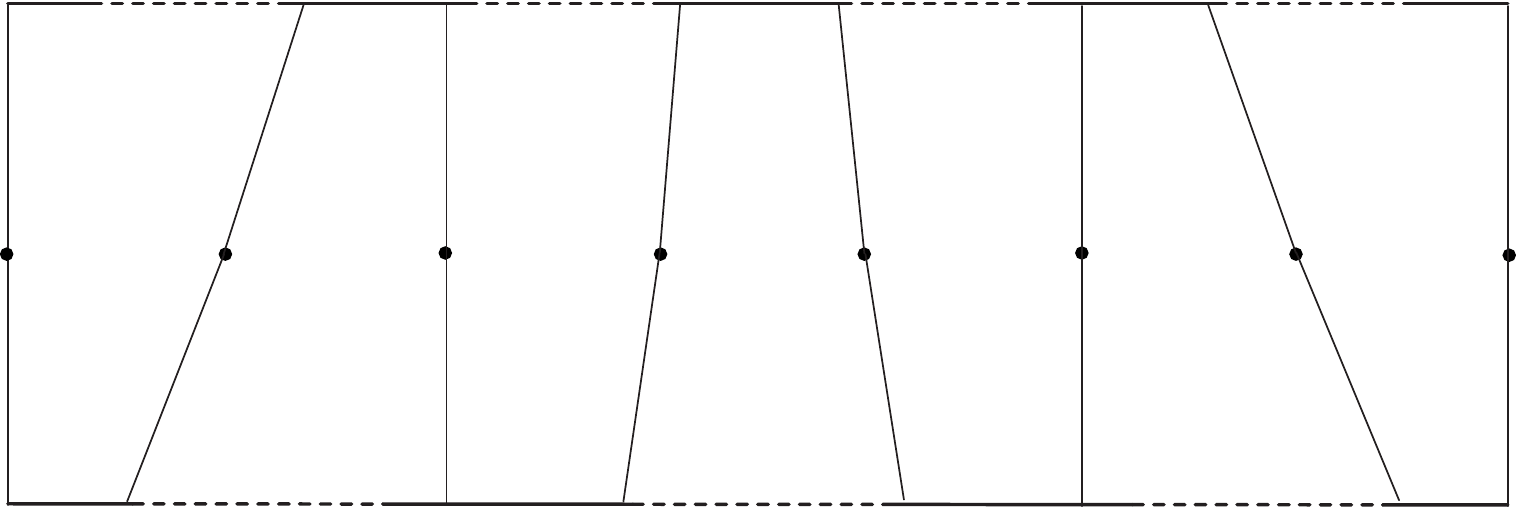}}
\put(2,17){$\xi^0$}
\put(95,17){$\xi^0$}
\put(59,17){$\xi^k$}
\put(34,17){$\xi^{k+1}$}
\put(46,25){$\evfS_k$}
\put(59,7){$\evgamma_k$}
\put(43,7){$\evgamma_{k+1}$}
\put(62,-2){$\evDin_m$}
\put(45,35){$\evDout_n$}
\put(44,-2){$\oddDin_m$}
\put(29,35){$\oddDout_n$}
\end{picture}
\caption{The closed covering $(\evfS_k)$ (example with $p=4$, $q=3$).}
\end{center}
\end{figure}
We then move the ends of the line segment $\evgamma'_k$, keeping $\xi^k$ fixed, so that the new curve~$\evgamma_k$ fulfills \eqref{prop:standtopmodel4c}. Then, if $q>1$, \eqref{prop:standtopmodel3} is fulfilled, according to Lemma \ref{lem:3xis}.

The case $q=1$, which needs a special treatment since $\cev$ is not injective in this case, is checked on Figure \ref{fig:casq=1}.
\begin{figure}[htb]
\begin{center}
\includegraphics[width=8cm,height=2cm]{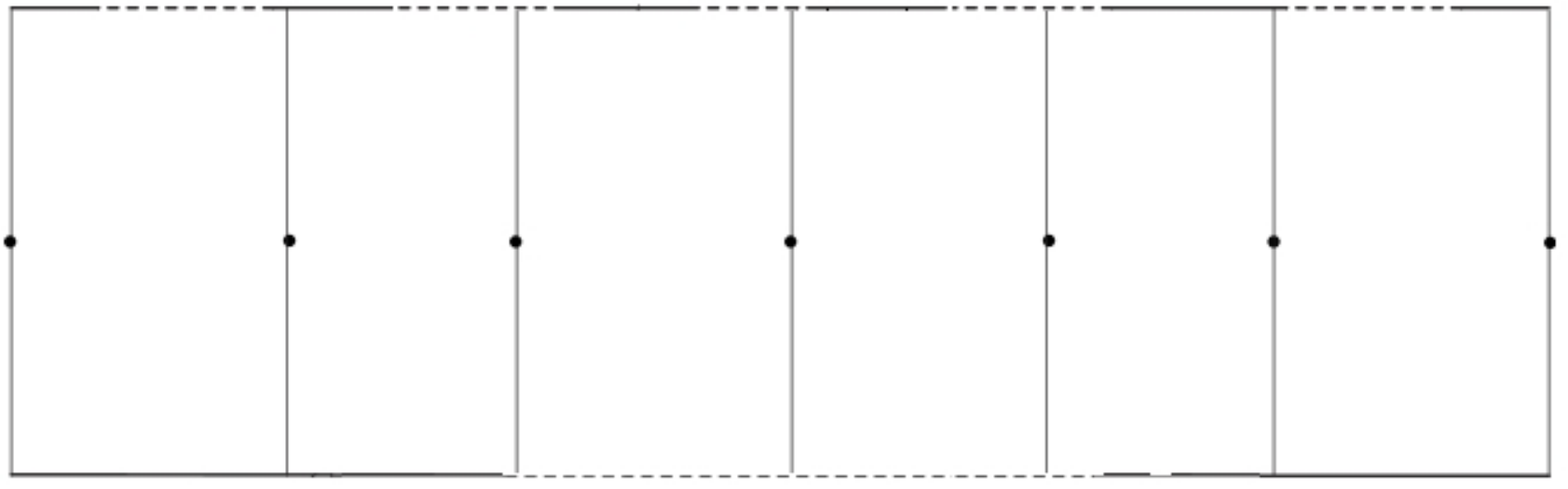}
\par\vspace*{3mm}
\includegraphics[width=8cm,height=2cm]{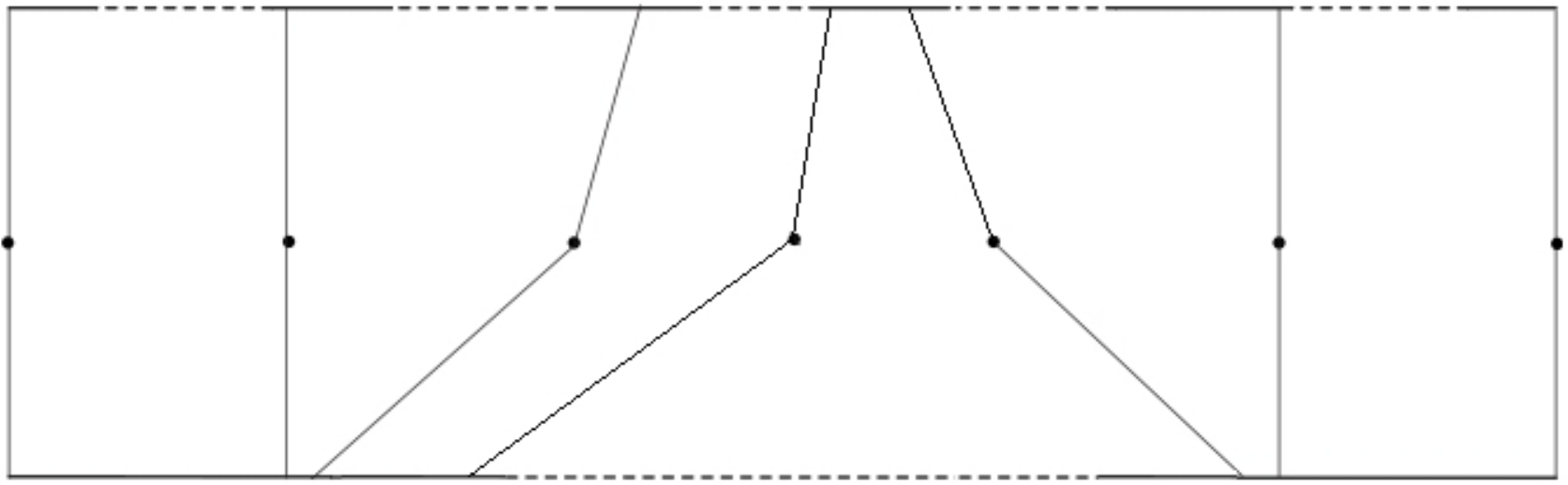}
\put(-40,6){\footnotesize $\xi^{\frac{p+q}{2}}$}
\caption{The closed coverings $(\evfS'_k)$, $(\evfS_k)$ ($p=5$, $q=1$).}\label{fig:casq=1}
\end{center}
\end{figure}
In the case $q=1 $ and $p \equiv 1 \text{ mod } 2 $, we use the convention that the curve $\evgamma_{\frac{p+1}{2}}$ is chosen as in Figure \ref{fig:casq=1}.
\end{proof}

\begin{corollaire}\label{coro:even}
Let $\cK$ be a local system on $\wA$. Then the covering $(\evfS_k)_{k=0,\dots,p+q-1}$ is a Leray covering for $\evbeta_!\cK$ ($\evbeta$ as in \eqref{eq:evoddbeta}) and the identification $H^0(\evgamma_k,\cK) \!\cong\nobreak\!\nobreak\cK_{\xi^k}$ induces an isomorphism
\begin{starequation}\label{eq:Psiev}
\Psiev: H^1(\wA, \evbeta_!\cK) \isom\bigoplus_{\wzeta\in\mu_{p+q}}\hspace*{-2mm}\cK_{\wzeta}.
\end{starequation}%
\end{corollaire}

\begin{proof}
The Leray property follows from Proposition \ref{prop:standtopmodel}\eqref{prop:standtopmodel3} together with \cite[Lem.\,7.13]{Bibi10}.
\end{proof}

Let us fix the even ordering $\mu_{p+q}=\{\wzeta_0 \lev \wzeta_1\lev\cdots \lev \wzeta_{p+q-1}\}$ on $\mu_{p+q}$ (\cf Definition \ref{def:standardorder}).

\begin{definition}\label{def:standfilev}
The even standard Stokes (increasing) filtration on $\bigoplus_{\wzeta\in\mu_{p+q}}\hspace*{-2mm}\cK_{\wzeta}$ is defined by the even ordering $F_k (\bigoplus_{\wzeta\in\mu_{p+q}}\hspace*{-2mm}\cK_{\wzeta}) :=\bigoplus_{j\leq k}\cK_{\wzeta_j}$.
\end{definition}

\subsubsection{The odd case} \label{sec:stdodd}

Similarly to \eqref{eq:cev}, consider the map (\cf\eqref{eq:oddinout} for the notation):
\begin{equation}\label{eq:codd}
\begin{split}
\{0,\dots, p+q-1\}&\to\{0,\dots,q-1\}\times\{-1,\dots,p-1\}\\
k&\mto\bigl(\oddin(k),\oddout(k)\bigr)
\end{split}
\end{equation}
and we also consider the composition $\codd$ with the projection to $\ZZ/q\ZZ\times\ZZ/p\ZZ$, and we still denote by $m$ (\resp $n$) the representative of $\oddin(k)\bmod q$ in $\{0,\dots,q-1\}$ (\resp of $\oddout(k)\bmod p$ in $\{0,\dots,p-1\}$), so we regard $\codd:k\mto(m,n)$ as a map $\{0,\dots, p+q-1\}\to\{0,\dots,q-1\}\times\{0,\dots,p-1\}$. We thus have $m=\oddin(k)$.

\begin{remarque}
As in Remark \ref{rem:cev}, $(m,n)$ is characterized by the properties
\[
\begin{cases}
\dpl\Bigl(\frac{2k\pi}{p+q}-\frac{\pi}{2q},\frac{2k\pi}{p+q}+\frac{\pi}{2q}\Bigr)\cap\oddIin{m}\neq\emptyset,\\[8pt]
\dpl\Bigl(\frac{2k\pi}{p+q}-\frac{\pi}{2q},\frac{2k\pi}{p+q}+\frac{\pi}{2q}\Bigr)\cap\oddIout{n}\neq\emptyset,
\end{cases}
\quad\forall\ve>0\text{ small enough}.
\]
We thus have
\begin{equation}\label{eq:connGodd}
\oddconn{m}{k}{k-m-1}\quad \text{if }\frac{k}{p+q}\in\Big[\frac mq,\frac mq+\frac1q\Bigr)\bmod1.
\end{equation}
\end{remarque}

Let us denote by $\ximinodd=\ximaxev\in\mu_{p+q}$ the minimal element with respect to the odd ordering, which is the maximal one with respect to the even ordering. There are two cases to be distinguished.
\begin{itemize}
\item
If $p+q$ is even, we have $\ximinodd=\xi^{a(p+q)/2}=-1$ and we set $\kminodd:=(p+q)/2$.
\item
If $p+q$ is odd, we have $\ximinodd=\xi^{-a (p+q-1)/2}$ (\cf \eqref{eq:ordev} for $a$). We set correspondingly $\kminodd:=-a (p+q-1)/2$.
\end{itemize}

\begin{remarque}
Although we could proceed as in the even case to define the Leray covering $(\oddfS_k)_{k=0,\dots,p+q-1}$, it will be convenient for a future use to define it by rotation from the even covering already constructed.
\end{remarque}

For $k\in\{0,\dots,p+q-1\}$, let us set $\{0,\dots,p+q-1\}\ni k':=k+\kminodd\bmod(p+q)$. We now define
\begin{equation}\label{eq:oddalpha}
\oddfS_{k'} :=\ximinodd \cdot \evfS_k
\end{equation}
to be the rotation of $\evfS_k$ by $\ximinodd$.

\begin{proposition}
If we assume that the paths $\evgamma_k$ satisfy the following supplementary properties when $p+q$ is odd:
\begin{starequation}\label{eq:oddevalpha}
\left\{\begin{aligned}
\evgamma_k\cap\evDin_m&\in \evDin_m-\frac{\pi}{q(p+q)},\\
\evgamma_k\cap\evDout_n&\in \evDout_n+\frac{\pi}{q(p+q)},
\end{aligned}\right.
\end{starequation}%
then the odd analogue of Proposition \ref{prop:standtopmodel} holds for $(\oddfS_k)_{k=0,\dots,p+q-1}$ (where we replace $\evDin$ and $\evDout$ in \ref{prop:standtopmodel}\eqref{prop:standtopmodel3} with the corresponding rotated intervals).
\end{proposition}

We note that since the length of $\evDin_m$ (\resp $\evDout_n$) is $\pi/q$, \eqref{eq:oddevalpha} can be realized by moving the end points of the paths $\evgamma_k$.

\begin{proof}
We only need to check \eqref{prop:standtopmodel4c}. Assume first that $p+q$ is even (so~that~$p$ and $q$ are odd, $\arg\ximinodd=\pi$ and $\kminodd=(p+q)/2$). We will show
\begin{align}\label{align:inDoddD}
\ximinodd\cdot\evDin_{\evin(k)}&=\oddDin_{\oddin(k')},\\ \notag
\ximinodd\cdot\evDout_{\evout(k)}&=\oddDout_{\oddout(k')}.
\end{align}
This amounts to showing
\begin{align*}
\frac2q\Big[\frac{q(k'-(p+q)/2)}{p+q}+\frac12\Big]+1&\equiv\frac2q\Big[\frac{qk'}{p+q}\Big]+\frac1q\mod2,\\
\frac2p\Big\lceil\frac{p(k'-(p+q)/2)}{p+q}-\frac12\Big\rceil+1&\equiv\frac2p\Big\lceil\frac{pk'}{p+q}-1\Big\rceil+\frac1p\mod2,
\end{align*}
which is easily checked.

Assume now that $p+q$ is odd. Let us write $aq\!=\!c(p+q)-1$ for some $c\!\in\!\ZZ$ (then $c-aq$ is odd), so~that
\[
\arg\ximinodd=-a\pi+\frac{a\pi}{p+q}=\frac{(c-aq)\pi}{p+q}-\frac{\pi}{q(p+q)}
\]
and
\[
\ximinodd\cdot\evDin_m=\oddDin_{m'}-\frac{\pi}{q(p+q)},\quad\text{with }\{0,\dots,q-1\}\ni m':=m+\tfrac{c-aq-1}2\bmod q,
\]
and therefore, since $\oddDin_{m'}$ has length $\pi/q$, $(\ximinodd\cdot\evDin_m)\cap\oddDin_{m'}\neq\emptyset$. Similarly,
\[
\ximinodd\cdot\evDout_n=\oddDout_{n'}+\frac{\pi}{p(p+q)},\ \text{with }\{0,\dots,p-1\}\ni n':=n+\tfrac{a-c-ap-1}2\bmod p.
\]
Let us check that $m'=\oddin(k)$ and $n'=\oddout(k)$, therefore proving the odd analogue of \ref{prop:standtopmodel}\eqref{prop:standtopmodel4c}. We have
\begin{align*}
\frac{qk'}{p+q}&=\frac{qk}{p+q}-\frac{aq(p+q-1)/2}{p+q}\in m-\frac{aq(p+q-1)/2}{p+q}+\Big[-\frac12,\frac12\Bigr)\bmod q\\
&=m'+\frac12-\frac1{2(p+q)}+\Big[-\frac12,\frac12\Bigr)\bmod q\\
&=\Big[m'-\frac1{2(p+q)},m'+1-\frac1{2(p+q)}\Bigr)\bmod q,
\end{align*}
which implies $qk'/(p+q)\in[m',m'+1)$, hence $\oddin(k')=m'$. Similarly,
\begin{align*}
\frac{pk'}{p+q}-1=\frac{pk}{p+q}-1&-\frac{ap(p+q-1)/2}{p+q}\\
&\in n-1-\frac{ap(p+q-1)/2}{p+q}+\Bigl(-\frac12,\frac12\Big]\bmod p\\
&=n'-\frac12+\frac1{2(p+q)}+\Bigl(-\frac12,\frac12\Big]\bmod p\\
&=\Bigl(n'-1+\frac1{2(p+q)},n'+\frac1{2(p+q)}\Big]\bmod p,
\end{align*}
hence $\oddout(k')=n'$.
\end{proof}

The obvious odd analogues of Lemma \ref{lem:3xis} and Corollary \ref{coro:even} hold, and we have an isomorphism
\begin{equation}\label{eq:Psiodd}
\Psiodd: H^1(\wA, \oddbeta_!\cK) \isom\bigoplus_{\wzeta\in\mu_{p+q}}\hspace*{-2mm}\cK_{\wzeta}.
\end{equation}

\begin{definition}\label{def:standfilodd}
The odd standard Stokes (decreasing) filtration on $\bigoplus_{\wzeta\in\mu_{p+q}}\hspace*{-2mm}\cK_{\wzeta}$ is defined by the odd ordering, i.e. if we number the elements of $\mu_{p+q}$ by the {\em even} ordering on $\mu_{p+q}=\{\zeta_0 \lev \zeta_1 \lev\cdots\}$ we have $F^k (\bigoplus_{\wzeta\in\mu_{p+q}}\hspace*{-2mm}\cK_{\wzeta}) :=\bigoplus_{j\geq k}\cK_{\wzeta_j}$.
\end{definition}

\subsection{The topological model}
We will now mimic the description of \S\ref{subsec:Stokesabstract}, from which we keep the notation, in the present simplified setting. However we choose the coordinates $(r,\vt_x,\vtn)$ on the product $\wA\times\SS^1_{\eta=0}$. This will make formulas simpler. We thus introduce the isomorphism\enlargethispage{\baselineskip}%
\begin{equation}\label{eq:vrho}
\begin{split}
[0,\infty]\times\SS^1_{u=0}\times\SS^1_{\eta=0}&\underset\sim{\To{\varrho}}[0,\infty]\times\SS^1_{x=0}\times\SS^1_{\eta=0}\\
(r,\vt,\vtn)&\Mto{\hphantom{\varrho}}(r,\vt_x,\vtn)=(r,\vt-\vtn,\vtn).
\end{split}
\end{equation}
The subset $\beta: B \hto \wA \times \SS^1$ contains~$A$ and has its inner (\resp outer) boundary described by \eqref{eq:Evt0at0} (\resp\eqref{eq:Evt0atinfty}). If we fix~$\vtn_\ell$ ($\ell=0,\dots2q-1$) as in \S\ref{subsec:Bwzeta} but with $\ve=0$ for simplicity, we obtain
\[
\bin B^{\vtn_\ell}=\left\{\begin{array}{lcl}
\bigcup_{m=0}^{q-1}\evDin_m&\quad\text{if $\ell$ is even}\quad&\bigcup_{n=0}^{p-1}\evDout_n\\[7pt]
\bigcup_{m=0}^{q-1}\oddDin_m&\quad\text{if $\ell$ is odd}\quad&\bigcup_{n=0}^{p-1}\oddDout_n
\end{array}\right\}
=\bout B^{\vtn_\ell}.
\]
With this notation, and given a local system $\cF$ on $[0,\infty]\times\SS^1_{u=0}$, we set $\wt\cK:=\vrho_*\wt\cF$ and $\cK^\ell:=\wt\cK_{|\wA\times\{\vtn_\ell\}}$. Then \eqref{eq:Labstr} reads
\begin{equation}\label{eq:Labstrsimple}
\begin{array}{c}
\xymatrix@C=.2cm@R=.5cm{
& H^1(\wA \times I_\ell, (\beta^{I_\ell})_!\wt\cK) \ar[dl]_{\sim} \ar[dr]^{\sim} \\
H^1(\wA, (\beta^\ell)_!\cK^\ell) \ar[rr]^{S_\ell^{\ell+1}}\ar[d]_{\Psi^\ell} &&
H^1(\wA, (\beta^{\ell+1})_!\cK^{\ell+1})\ar[d]^{\Psi^{\ell+1}}\\
\dpl\bigoplus_{\wzeta\in\mu_{p+q}}\hspace*{-2mm}\cK^\ell_{\wzeta}\ar[rr]^{\sigma^{\ell+1}_\ell}
& &\dpl\bigoplus_{\wzeta\in\mu_{p+q}}\hspace*{-2mm}\cK^{\ell+1}_{\wzeta}
}
\end{array}
\end{equation}
where $\Psi^\ell$ (\resp $\beta^\ell$) is $\Psiev$ or $\Psiodd$ (\resp $\evbeta$ or $\oddbeta$) according to the parity of $\ell$.

Let $(\bV, \bT)$ be the monodromy pair of the local system $\cF$ on the annulus. Then~$\wt\cK$ can be described as the triple $(\bV,\bT,\id)$, where $\bT$ is the monodromy with respect to $\SS^1_{x=0}$ and $\id$ that with respect to $\SS^1_{\eta=0}$.

\begin{convention}\label{conv:identificationV}
We fix an identification $\Gamma\bigl((\SS^1_{x=0}\moins\{\vt_x=-\ve\})\times\SS^1_{\eta=0},\wt\cK\bigr)$ with~$\bV$, and $\bT$ is recovered as usual by going in the positive direction from $\vt_x=-2\ve$ to $\vt_x=0$. This fixes, for each $\ell$, an identification $\bigoplus_{\wzeta\in\mu_{p+q}}\cK^\ell_{\wzeta}=\bigoplus_{k=0}^{p+q-1}\bV\otimes{\bf1}_k$, with the notation as in \eqref{eq:bVpq}.
\end{convention}

\begin{definition}\label{def:topmodel}
The data $\bigl((\bigoplus_{\wzeta\in\mu_{p+q}}\hspace*{-2mm}\cK^\ell_{\wzeta})_\ell, (\sigma^{\ell+1}_\ell)_\ell,F)$ -- the last entry being defined by \ref{def:standfilev} and \ref{def:standfilodd} -- will be called the topological model for the standard Stokes structure associated to $(\bV, \bT)$.
\end{definition}

For each $\ell=0,\dots,2q-1$, let us fix a diffeomorphism $\wA\times I_\ell\isom\wA\times I_\ell$ by lifting the vector field $\partial_{\vtn}$ to $\wA\times\SS^1_{\eta=0}$ in such a way that the lift is equal to $\partial_{\vtn}$ away from a small neighbourhood of $\partial\wA$ and such that the diffeomorphism induces $B^{I_\ell}\isom B^{\vtn_{\ell+1}}\times I_\ell$. It also induces a diffeomorphism $\wA\times\{\vtn_\ell\}\to\wA\times\{\vtn_{\ell+1}\}$ and the push-forward of $\cK^\ell$ is isomorphic to $\cK^{\ell+1}$. Moreover, \eqref{eq:Evt0at0} and \eqref{eq:Evt0atinfty} show that
\begin{itemize}
\item
for $\ell$ even, ${}^\ell\!\Din_m$ is sent to ${}^{\ell+1}\!\Din_{m-1}$ and ${}^\ell\!\Dout_n$ to ${}^{\ell+1}\!\Dout_n$,
\item
for $\ell$ odd, ${}^\ell\!\Din_m$ is sent to ${}^{\ell+1}\!\Din_m$ and ${}^\ell\!\Dout_n$ to ${}^{\ell+1}\!\Dout_{n+1}$.
\end{itemize}
The Leray covering $(\evfS_k)_k$ is sent to a Leray covering $(\evtfS_k)_k$ and the curves $(\evgamma_k)_k$ are sent to curves $(\evtgamma_k)_k$. The Leray covering $(\evtfS_k)_k$ induces the isomorphism
\[
\wt\Psi^\ell:H^1(\wA, (\beta^{\ell+1})_!\cK^{\ell+1})\isom\bigoplus_{\wzeta\in\mu_{p+q}}\hspace*{-2mm}\cK^{\ell+1}_{\wzeta},
\]
which is transported by the previous isomorphism from $\Psi^\ell$, and the lower part of \eqref{eq:Labstrsimple} reads
\begin{equation}\label{eq:Labstrsimpleell+1}
\begin{array}{c}
\xymatrix@C=.2cm@R=.5cm{
H^1(\wA, (\beta^{\ell+1})_!\cK^{\ell+1})\ar@{=}[rr]\ar[d]_{\wt\Psi^\ell} &&
H^1(\wA, (\beta^{\ell+1})_!\cK^{\ell+1})\ar[d]^{\Psi^{\ell+1}}\\
\dpl\bigoplus_{\wzeta\in\mu_{p+q}}\hspace*{-2mm}\cK^{\ell+1}_{\wzeta}\ar[rr]^{\sigma^{\ell+1}_\ell}
& &\dpl\bigoplus_{\wzeta\in\mu_{p+q}}\hspace*{-2mm}\cK^{\ell+1}_{\wzeta}
}
\end{array}
\end{equation}

\subsubsection{Explicit description of \texorpdfstring{$\sigma_\ell^{\ell+1}$}{sigma} for \texorpdfstring{$\ell$}{ell} even}\label{subsec:cechcompu}
Due to Proposition \ref{prop:standtopmodel}, we know that (regarding $m$ \resp $k-m$ modulo $q$ \resp $p$ as indicated above)
\begin{equation}\label{eq:evalphconn}
\evgamma_k:
\begin{cases}
\evconn{m}{k}{k-m} & \text{if } \frac{k}{p+q}\in\Big[\frac mq,\frac mq+\frac1{2q}\Bigr)\bmod1, \\[.2cm]
\evconn{m+1}{k}{k-m-1} & \text{if }\frac{k}{p+q}\in\Big[\frac mq+\frac1{2q},\frac mq+\frac1q\Bigr)\bmod1,
\end{cases}
\end{equation}
and consequently
\begin{equation}\label{eq:connGev}
\evtgamma_k:
\begin{cases}
\oddconn{m-1}{k}{k-m} & \text{if }\frac{k}{p+q}\in\Big[\frac mq,\frac mq+\frac1{2q}\Bigr)\bmod1,\\[.2cm]
\oddconn{m}{k}{k-m-1} & \text{if }\frac{k}{p+q}\in\Big[\frac mq+\frac1{2q},\frac mq+\frac1q\Bigr)\bmod1.
\end{cases}
\end{equation}
From \eqref{eq:connGodd} we deduce that $\oddgamma_k$ and $\evtgamma_k$ show the same behaviour whenever $\frac{k}{p+q}\in\Big[\frac mq+\nobreak\frac1{2q},\frac mq+\nobreak\frac1q\Bigr)\allowbreak\bmod1$.

Fixing an $m\in\{0,\dots, q-1\}$ and considering the interval $\oddDin_m$, we see that the curves $\oddgamma_k$ with
$\ellmin(m):=\lceil \frac{m(p+q)}{q} \rceil \le k \le \lceil \frac{(m+1)(p+q)}{q}-1 \rceil=:\ellmax(m)$ are exactly the ones starting at $\oddDin_m$. Additionally, we define $\ellmid(m):=\lceil\frac{(2m+1)(p+q)}{2q} \rceil $. We have $\ellmid(m)-\ellmax(m-1)\geq2$ and $\ellmax(m)-\ellmid(m)\geq1$. Moreover, for fixed~$m$, the curves $\oddgamma_k$ and $\evtgamma_k$
\begin{itemize}
\item
connect the same intervals if $\ellmid(m) \le k \le \ellmax(m)$ and
\item
connect different intervals if $\ellmin(m) \le k < \ellmid(m)$.
\end{itemize}

\begin{figure}[htb]
\begin{center}
\begin{picture}(90,40)(0,0)
\put(0,3){\includegraphics[scale=0.4]{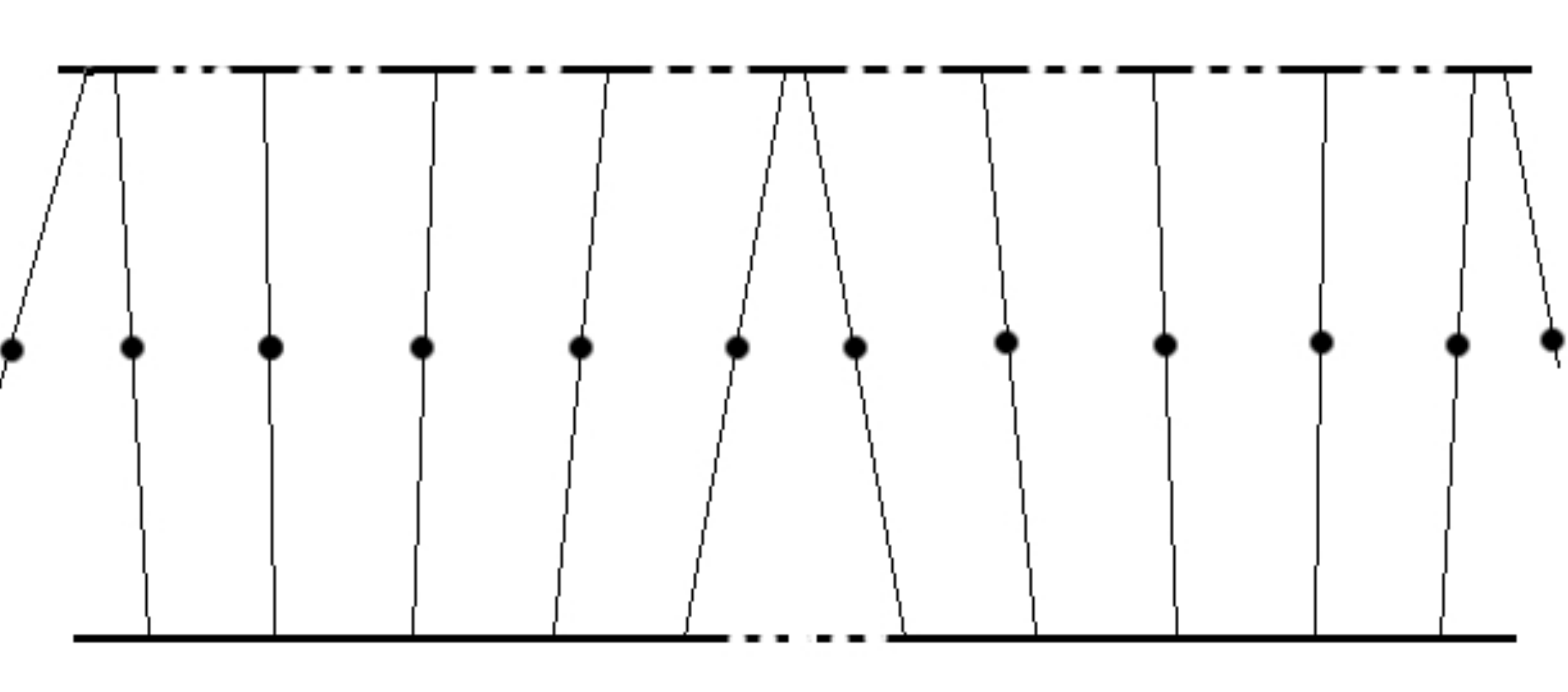}}

\put(0,13){\footnotesize $\ellmax(m)$}
\qbezier(7,21)(3,18)(2,16)
\put(42,15){\footnotesize $\ellmin(m)$}
\qbezier(41,21)(42,18)(45,16)
\put(48,30){\footnotesize $\ellmax(m-1)$}
\qbezier(52,29)(52,25)(48,23)

\put(28,30){\footnotesize $\ellmid(m)$}
\qbezier(30,29)(28,25)(25,23)

\put(16,20){\footnotesize $k$}
\put(9,39){\footnotesize $\oddDout_{k-m-1}$}

\put(20,0){\footnotesize $\oddDin_m$}
\put(60,0){\footnotesize $\oddDin_{m-1}$}
\end{picture}
\end{center}
\caption{The curves $\oddgamma_k$ starting at $\oddDin_m$ and $\oddDin_{m-1}$ for fixed~$m$. Exactly for $\ellmid(m) \le k \le \ellmax(m)$, the curves $\evtgamma_k$ have the same behaviour as $\oddgamma_k$ (we denote $k$ instead of $\xi^k$).}\label{figuremminus1}
\end{figure}

Let us consider the situation between $\ellmax(m-1)$ and $\ellmid(m)$ for a given $m\in\{0,\dots,q-1\}$, as sketched on Figure \ref{figure:lmaxmid1}.
\begin{figure}[htb]
\begin{center}
\begin{picture}(80,40)(0,0)
\put(0,3){\includegraphics[scale=0.38]{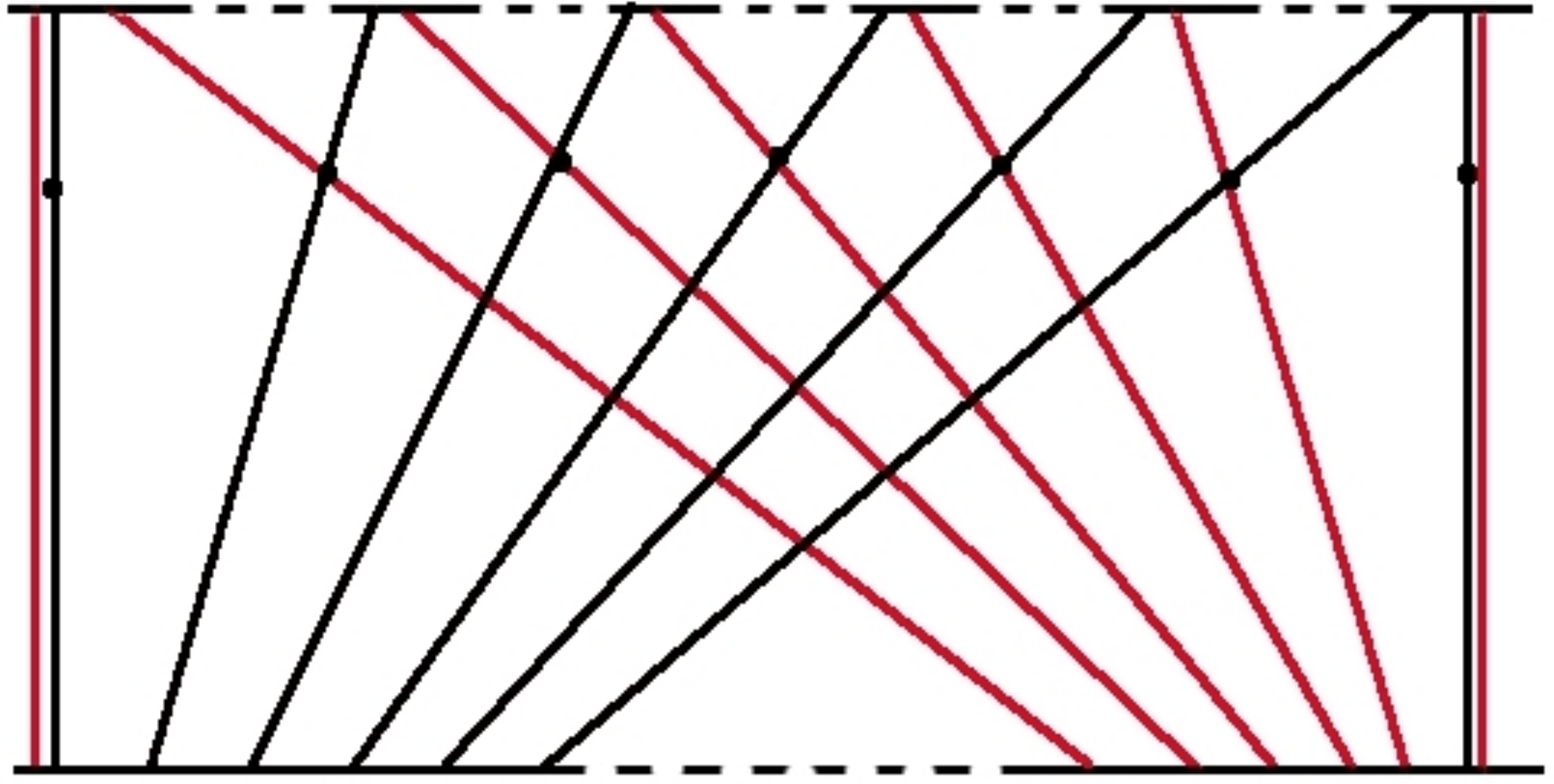}}

\put(-10,30){\footnotesize $\ellmid(m)$}
\put(70,30.3){\footnotesize $\ellmax(m-1)$}
\put(17,30.5){\footnotesize $k$}

\put(3,20){\footnotesize $\oddgamma_k$}
\qbezier(13.5,26)(11,26)(7,23)
\put(6,28){\footnotesize $\evtgamma_k$}
\qbezier(8,30.5)(9,34)(10,35)

\put(10,0){\footnotesize $\oddDin_m$}
\put(60,0){\footnotesize $\oddDin_{m-1}$}
\end{picture}
\end{center}
\caption{The sector $\fS$ and the curves $\oddgamma_k$ and $\evtgamma_k$ in the range $[\ellmax(m-1), \ellmid(m)] $.}\label{figure:lmaxmid1}
\end{figure}

We set
\[
s:=\ellmid(m)-\ellmax(m-1)\geq2.
\]
This picture allows us to define the matrix ${}^\ev\bS(m)\in\mathrm{Mat}_{(s+1) \times (s+1)}(\End(\bV))$. Writing $1=\id_{\bV}$, we set\enlargethispage{\baselineskip}%
\begin{equation}\label{eq:Smatrix}
\arraycolsep3.5pt
{}^{\ev}\bS(m):=
\begin{pmatrix}
1&1&1&1&\dots &1&0 \\
0&-1&-1&-1&\dots &-1&0 \\
0&1&0&0&\dots &0&0 \\
0&0&1&0&\dots &0&0 \\
\vdots&\vdots&\vdots& \ddots & \ddots & \vdots&\vdots \\
0&0&0&\dots&1 &0&0\\
0&0&0&0&\dots &1&1
\end{pmatrix}.
\end{equation}
In case $s$ takes its minimal value $s=2$, the block with diagonal zeros disappears and we obtain
\[\arraycolsep3.5pt
{}^{\ev}\bS(m)=\begin{pmatrix}
1&1&0 \\
0&-1&0 \\
0&1&1
\end{pmatrix}.
\]

\begin{remarque}\label{rem:heur}
The heuristic rule is that, on Figure \ref{figure:lmaxmid1}, the red segment going through $\ellmax(m-1)-1+j$ ($j=1,\dots,s+1$) oriented from in to out ($j$th column of ${}^{\ev}\bS(m)$) is written as a sum of black segments oriented from in to out with the same origin and end horizontal segments. The vertical black segments are also regarded to be red.
\end{remarque}

\begin{lemme}\label{lem:Lmatrix10}
The isomorphism $\sigma_\ell^{\ell+1}:\bigoplus_{k=0}^{p+q-1}\bV\otimes{\bf1}_k\to\bigoplus_{k=0}^{p+q-1}\bV\otimes{\bf1}_k$ does not depend on $\ell$ if $\ell$ is even and is the linear map decomposing into diagonal blocks as follows:
\begin{enumerate}
\item
$\id_{[\ellmid(m),\ellmax(m))}$ (maybe empty), $m=0,\dots,q-1$,
\item
${}^\ev\bS(m)_{[\ellmax(m-1),\ellmid(m)]}$ (of size at least $3$), $m=1,\dots,q-1$,
\item
$\diag(1,\bT, \bT,\dots, \bT) \cdot {}^\ev\bS(0)_{[\ellmax(q-1),\ellmid(0))} \cdot \diag(1, \bT^{-1}, \bT^{-1},\dots, \bT^{-1})$.
\end{enumerate}
\end{lemme}

\begin{proof}
We can first replace $\evtfS_k$ with $\oddfS_k$ when $\ellmid(m)\leq k<\ellmax(m)$ without changing $\wt\Psi_\ell$, and thus identify $\evtgamma_k$ with $\oddgamma_k$ when $\ellmid(m)\leq k\leq\ellmax(m)$. It follows that the restriction of the linear isomorphism $\sigma_\ell^{\ell+1}$
to the subspace
$$
\bigoplus_{k=\ellmid(m)+1}^{\ellmax(m)-1} \cK^{\ell+1}_{\xi^k}
$$
equals the identity. Note that this block may be empty if $p<3q$.

Next, let us understand the situation when $\ellmax(m-1) < k < \ellmid(m)$ (\cf Figure~\ref{figure:lmaxmid1}). In such a case,
\begin{equation}\label{eq:alphaevS}
\oddgamma_k: \oddconn{m}{k}{k-m-1} \text{\quad and \quad} \evtgamma_k: \oddconn{m-1}{k}{k-m}.
\end{equation}
Consider the closed sector
\[
\fS(m):=\bigcup_{k=\ellmax(m-1)-1}^{\ellmid(m)}\oddfS_k=\bigcup_{k=\ellmax(m-1)-1}^{\ellmid(m)}\evtfS_k
\]
of the closed annulus between the curves $\evtgamma_{\ellmax(m-1)}=\oddgamma_{\ellmax(m-1)}\geq2$ and $\evtgamma_{\ellmid(m)}=\oddgamma_{\ellmid(m)}$, due to our previous identification. Our goal is to replace both Leray coverings on $\fS(m)$ with a single one, in order to compute the matrix of $\sigma_\ell^{\ell+1}$ in the corresponding basis. Let consider the closed coverings $\fA:=(A_j)_{j=-1,\dots, s}$ and $\fE:=(E_j)_{j=-1,\dots,s}$, with $A_j=\oddfS_{\ellmax(m-1)+j}$ and $E_j=\evtfS_{\ellmax(m-1)+j}$ as in Figure \ref{figure:lerayAE}.
\begin{figure}[htb]
\begin{center}
\begin{picture}(125,110)(5,0)
\put(5,83){\includegraphics[scale=0.3]{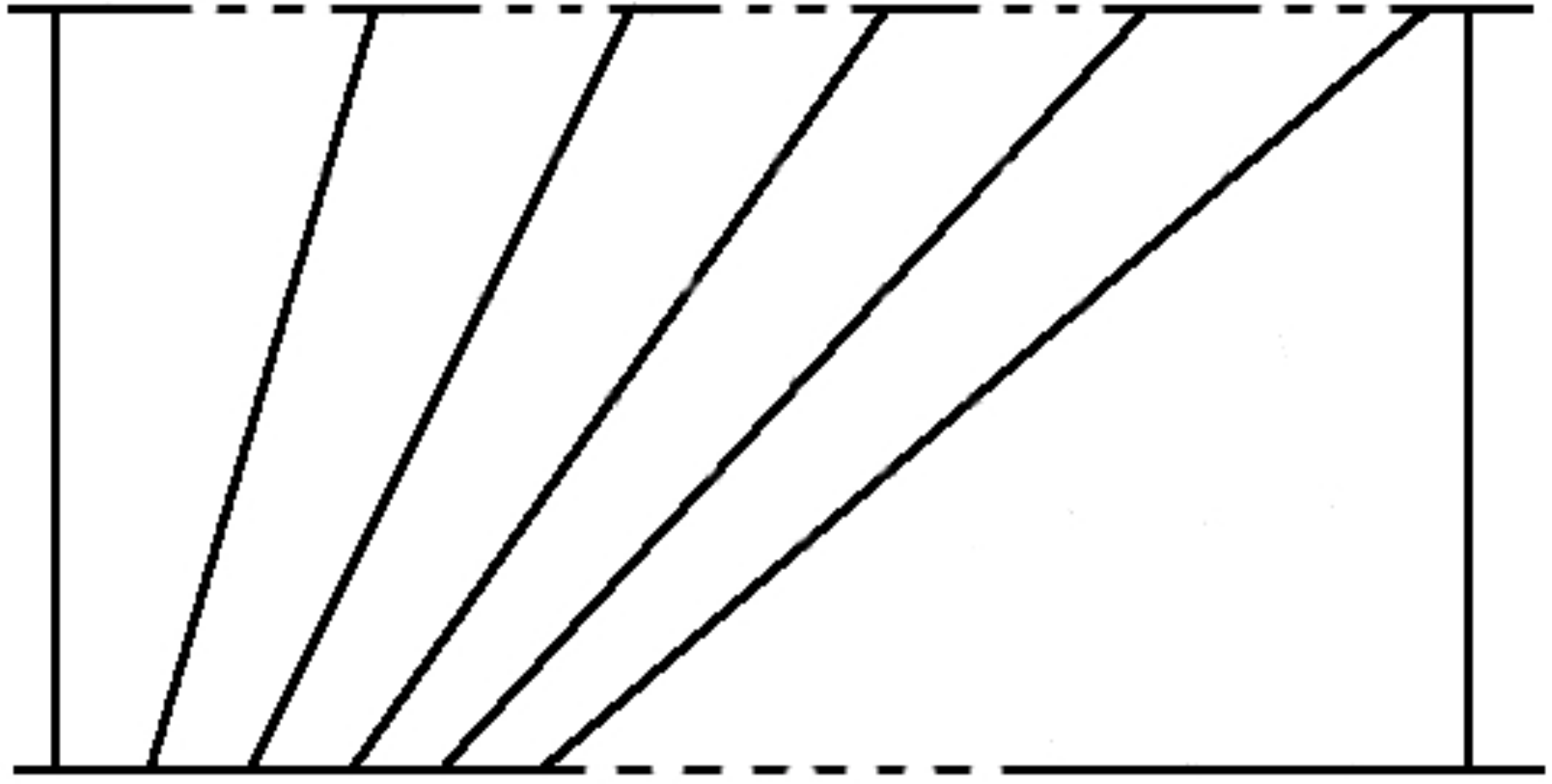}}
\put(70,83){\includegraphics[scale=0.3]{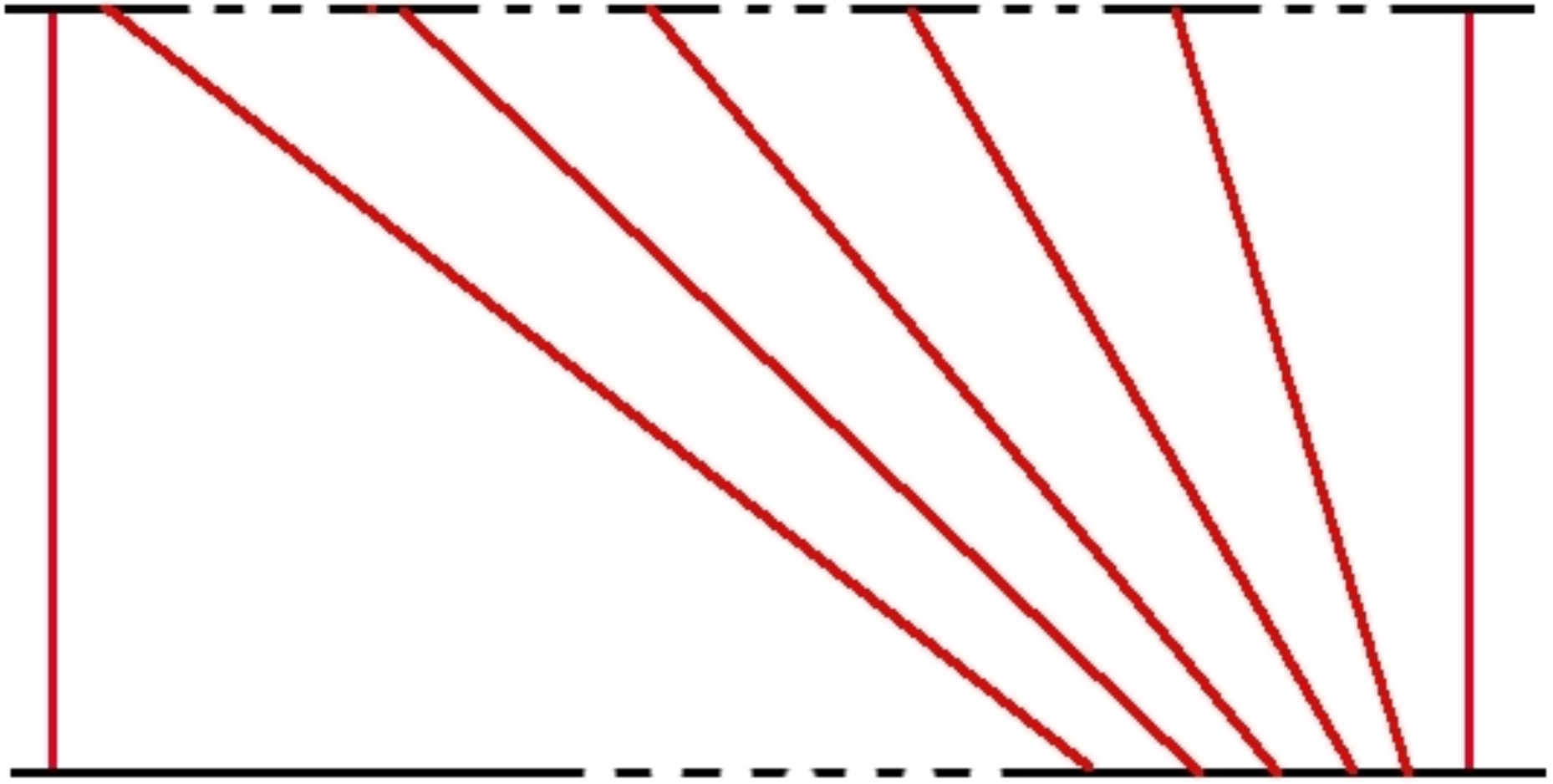}}

\put(5,43){\includegraphics[scale=0.3]{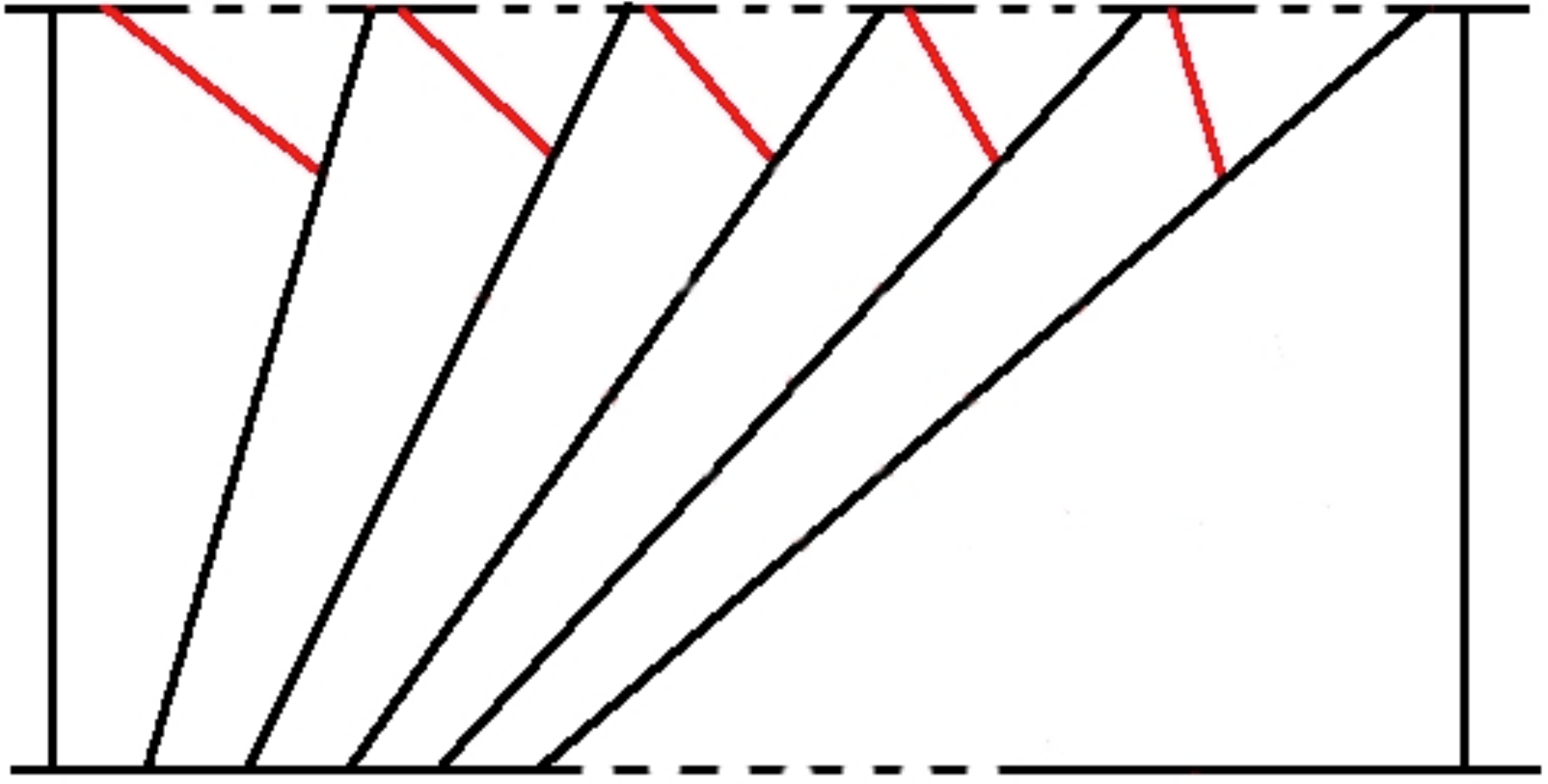}}
\put(70,43){\includegraphics[scale=0.3]{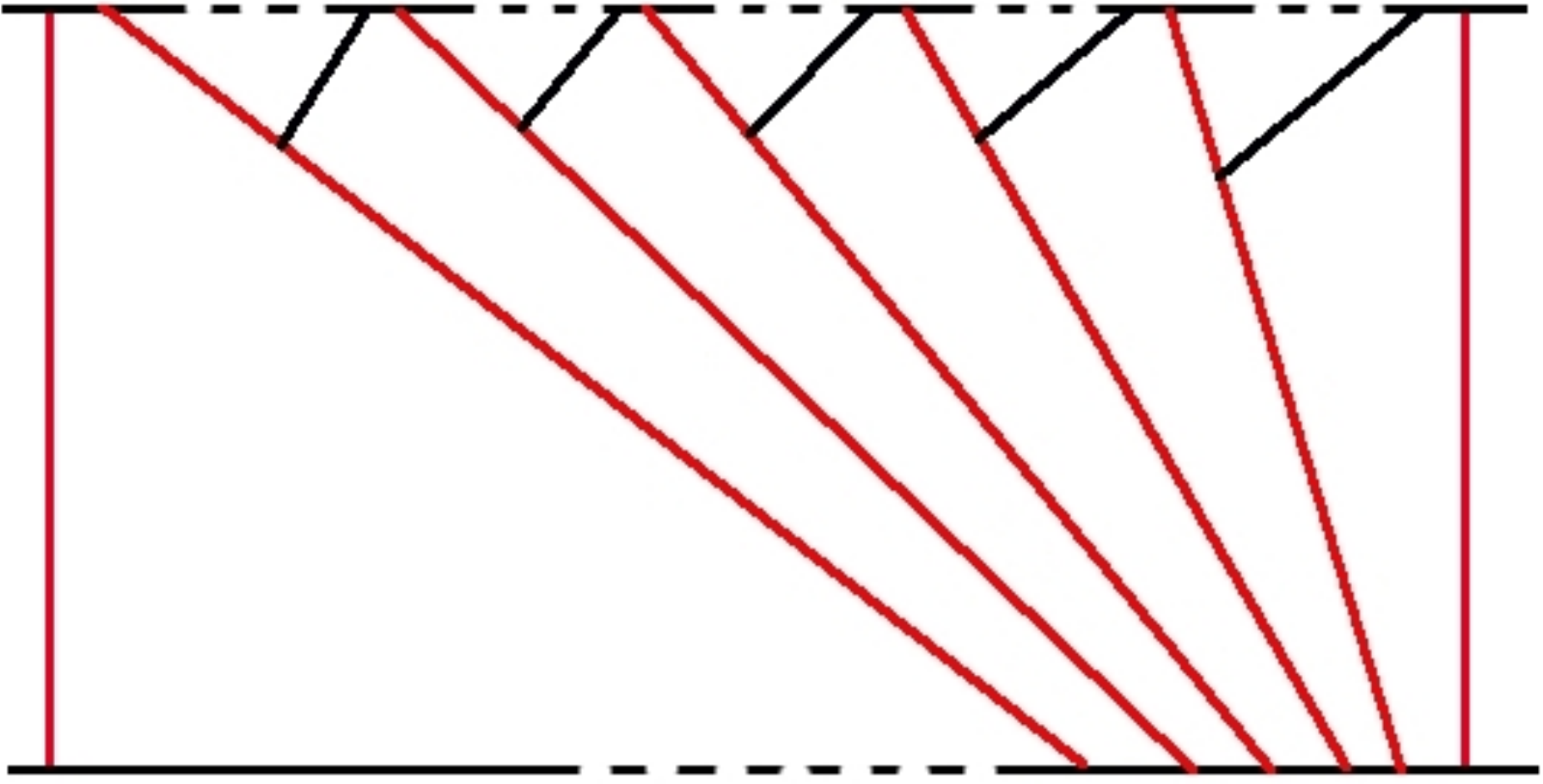}}

\put(35,3){\includegraphics[scale=0.3]{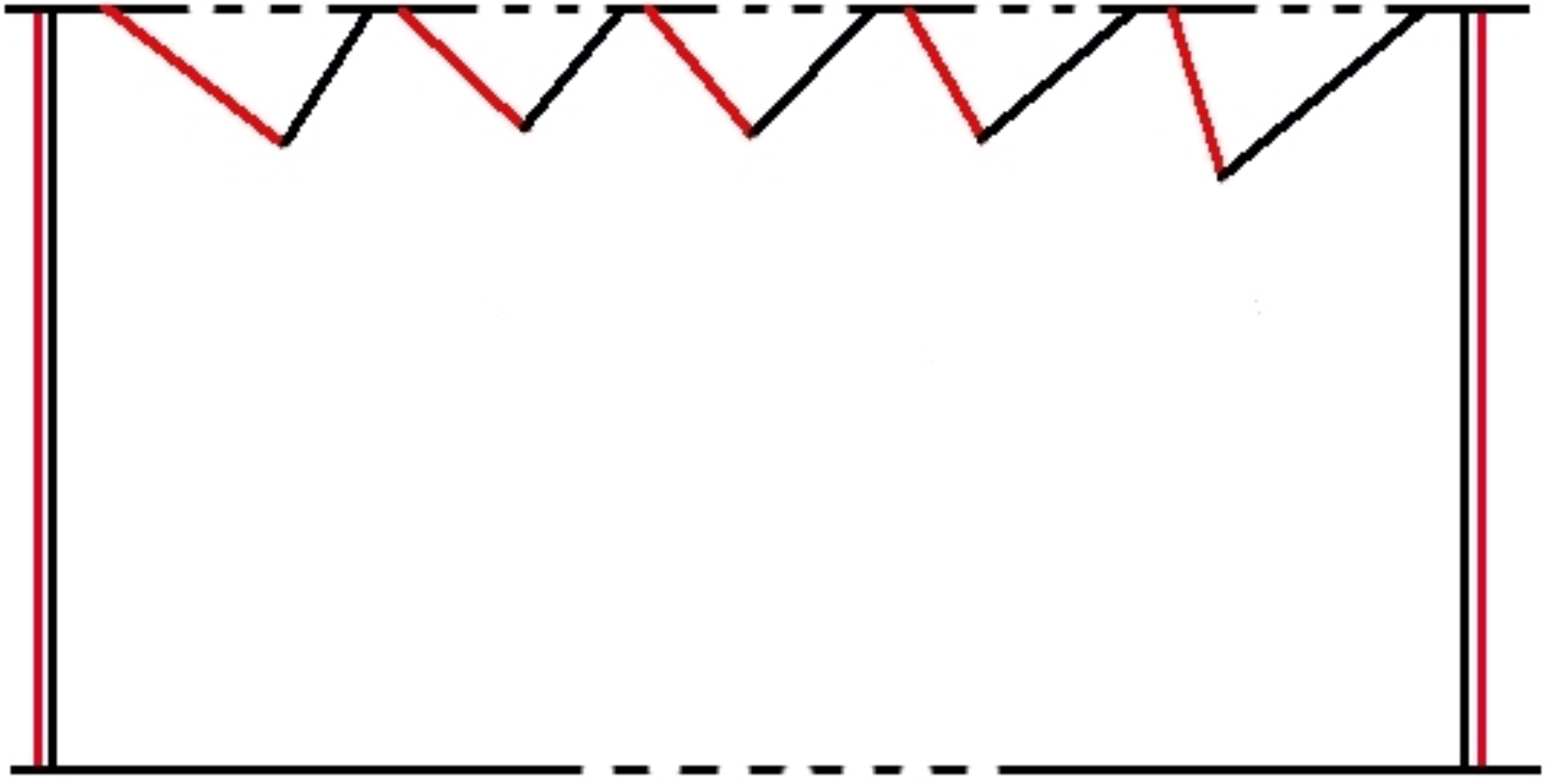}}

\put(2,96){\footnotesize $A_s$}
\put(8,102){\footnotesize $A_{s-1}$}
\put(28,106){\footnotesize $A_j$}
\put(52,92){\footnotesize $A_0$}
\put(60,92){\footnotesize $A_{-1}$}

\put(67,100){\footnotesize $E_s$}
\put(77,98){\footnotesize $E_{s-1}$}
\put(93,105){\footnotesize $E_j$}
\put(118,104){\footnotesize $E_0$}
\put(125,100){\footnotesize $E_{-1}$}

\put(2,56){\footnotesize $B_s$}
\put(25,62){\footnotesize $B^+_j$}
\put(31,68.5){\footnotesize $B^-_j$}
\put(52,52){\footnotesize $B_0$}
\put(60,52){\footnotesize $B_{-1}$}

\put(67,60){\footnotesize $D_s$}
\put(77,58){\footnotesize $D_{s-1}$}
\put(93,64){\footnotesize $D^+_j$}
\put(87,68.5){\footnotesize $D^-_j$}
\put(118,62){\footnotesize $D^+_0$}
\put(114.3,68.5){\footnotesize $D^-_0$}
\put(125,60){\footnotesize $D_{-1}$}

\put(32,16){\footnotesize $C_s$}
\put(48,23){\footnotesize $C_{s-1}$}
\qbezier(49,25.5)(47,26)(45,30)
\put(57,23){\footnotesize $C_{s-2}$}
\qbezier(58,25.5)(56,26)(54,30)
\put(82,23){\footnotesize $C_1$}
\qbezier(83,25.5)(81,26)(80,30)
\put(62,12){\footnotesize $C_0$}
\put(90,12){\footnotesize $C_{-1}$}

\put(10,80){\footnotesize $\oddDin_m$}
\put(45,80){\footnotesize $\oddDin_{m-1}$}
\put(75,80){\footnotesize $\oddDin_m$}
\put(110,80){\footnotesize $\oddDin_{m-1}$}

\put(28,82){\vector(0,-1){8}}
\put(30,76){\footnotesize refinement}

\put(36,34){\vector(-1,1){8}}
\put(35,36){\footnotesize refinement}

\put(93,82){\vector(0,-1){8}}
\put(95,76){\footnotesize refinement}

\put(85,34){\vector(1,1){8}}
\put(90,36){\footnotesize refinement}

\end{picture}
\end{center}
\caption{The coverings $\fA$ and $\fE$ of the sector $\fS$, their refinements $\fB,\fD$, themselves being refinements of a common covering $\fC $.}\label{figure:lerayAE}
\end{figure}
We first pass from both coverings to individual refinements. Let us denote~by
\begin{itemize}
\item
$\fB:=\{B_{-1}, B_0, B_s\} \cup \{B^{\pm}_j \mid j=1,\dots, s-1\}$ the refinement of $\fA$ and
\item
$\fD:=\{D_{-1}, D_{s-1}, D_s\} \cup \{D^{\pm}_j \mid j=0,\dots, s-2\}$ the refinement of $\fE$
\end{itemize}
as shown on Figure \ref{figure:lerayAE}, so that $B_{-1}=D_{-1}=\oddfS_{\ellmax(m-1)-1}$ and $B_s=D_s=\oddfS_{\ellmid(m)}$. From these refinements we again pass back to a common coarser covering $\fC:=\{C_j\mid j=-1,\dots, s\}$ -- in the sense of a closed covering which refines to both~$\fB$ and $\fD$. The construction of these can be read off Figure \ref{figure:lerayAE}. The proof of Lemma \ref{lem:Lmatrix10} consists in a computation of the corresponding base change.
\end{proof}

\begin{remarque}\label{rem:elloddev} \mbox{}
\begin{enumerate}
\item The above computations have been carried out in the case $q<p $. In the case $q>p $, we can proceed in the same way (which amounts to interchanging the roles of the inner and outer boundary components). The case $p=q=1 $ will be studied in \S\ref{subsec:pq1} below.

\item The isomorphism $\sigma^{\ell+1}_\ell $ of Lemma \ref{lem:Lmatrix10} coincides with $\sigma^\odd_\ev $
defined by \eqref{eq:defL01}. Let us verify this statement. Recall that $\evtgamma_k \cap \partial^{\text{in}} \wA \in \oddDin_{\evin(k)-1}$. By definition of $\ellmin(m), \ellmid(m) $ and $\ellmax(m) $, we deduce that
$$
k \in [\ellmid(m), \ellmax(m)] \Longleftrightarrow \evin(k)-1=\oddin(k).
$$
Then $\sigma^{\ell+1}_\ell(v \otimes 1_k)=v \otimes 1_k = \sigma_\ev^\odd (v \otimes 1_k) $. In the other case, $\sigma^{\ell+1}_\ell(v \otimes 1_k) $ is determined by the corresponding column of the block ${}^{\ev}\bS(m) $ (different from the first or last column). The rows with the non-vanishing entries $1,-1,1 $ correspond to the places
$$
\oddminin(k),\quad \oddmaxout(\oddminin(k)) \quad\text{and}\quad \oddmaxout(k)
$$
respectively.
\end{enumerate}
\end{remarque}

\subsubsection{Explicit description for \texorpdfstring{$\sigma^{\ell+1}_\ell$}{sigma} for \texorpdfstring{$\ell$}{ell} odd}

The result in this case is similar. In analogy to \eqref{eq:evalphconn} and \eqref{eq:connGev}, we now have
\begin{equation}\label{eq:oddgammak}
\oddgamma_k:
\begin{cases}
\oddconn{m-1}{k}{k-m} & \text{if } \frac{k}{p+q} \in \Big[ \frac{m}{q}-\frac{1}{2q}, \frac{m}{q} \Bigr) \bmod 1
\\[0.3cm]
\oddconn{m}{k}{k-m-1} & \text{if } \frac{k}{p+q} \in \Big[ \frac{m}{q}, \frac{m}{q}+\frac{1}{2q} \Bigr) \bmod 1
\end{cases}
\end{equation}
and consequently
\begin{equation}\label{eq:oddgammatildek}
{}^{\text{odd}} \widetilde{\gamma}_k:
\begin{cases}
\evconn{m-1}{k}{k-m+1} & \text{if }
\frac{k}{p+q} \in \Big[ \frac{m}{q}-\frac{1}{2q}, \frac{m}{q} \Bigr) \bmod 1
\\[0.3cm]
\oddconn{m}{k}{k-m} & \text{if }
\frac{k}{p+q} \in \Big[ \frac{m}{q}, \frac{m}{q}+\frac{1}{2q} \Bigr) \bmod 1.
\end{cases}
\end{equation}
Hence $\evgamma_k$ and ${}^{\text{odd}} \widetilde{\gamma}_k$ show the same behaviour whenever $\tfrac{k}{p+q} \in\big[ \tfrac{m}{q}, \tfrac{m}{q}+\tfrac{1}{2q}\bigr) \bmod 1$.

Picking up the notation from \S\ref{subsec:cechcompu}, we see that for fixed $m \in \{0, \ldots, q-1 \}$, the curves $\evgamma_k$ start at $\evDin_m $ exactly for
\begin{equation}\label{eq:evkm}
\kmid(m-1) = \Big\lceil \frac{(2m-1)(p+q)}{2q} \Big\rceil \le k \le \Big\lceil \frac{(2m+1)(p+q)}{2q}-1 \Big\rceil = \kmid(m)-1,
\end{equation}
and that among these the curves $\evgamma_k$ and ${}^{\text{odd}}\widetilde{\gamma}_k$
\begin{itemize}
\item connect the same intervals if $\kmin(m) \le k \le \kmid(m)-1 $ and
\item connect different intervals if $\kmid(m-1) \le k < \kmin(m) $.
\end{itemize}

\begin{figure}[htb]
\begin{center}
\begin{picture}(90,40)(0,0)
\put(0,3){\includegraphics[scale=0.4]{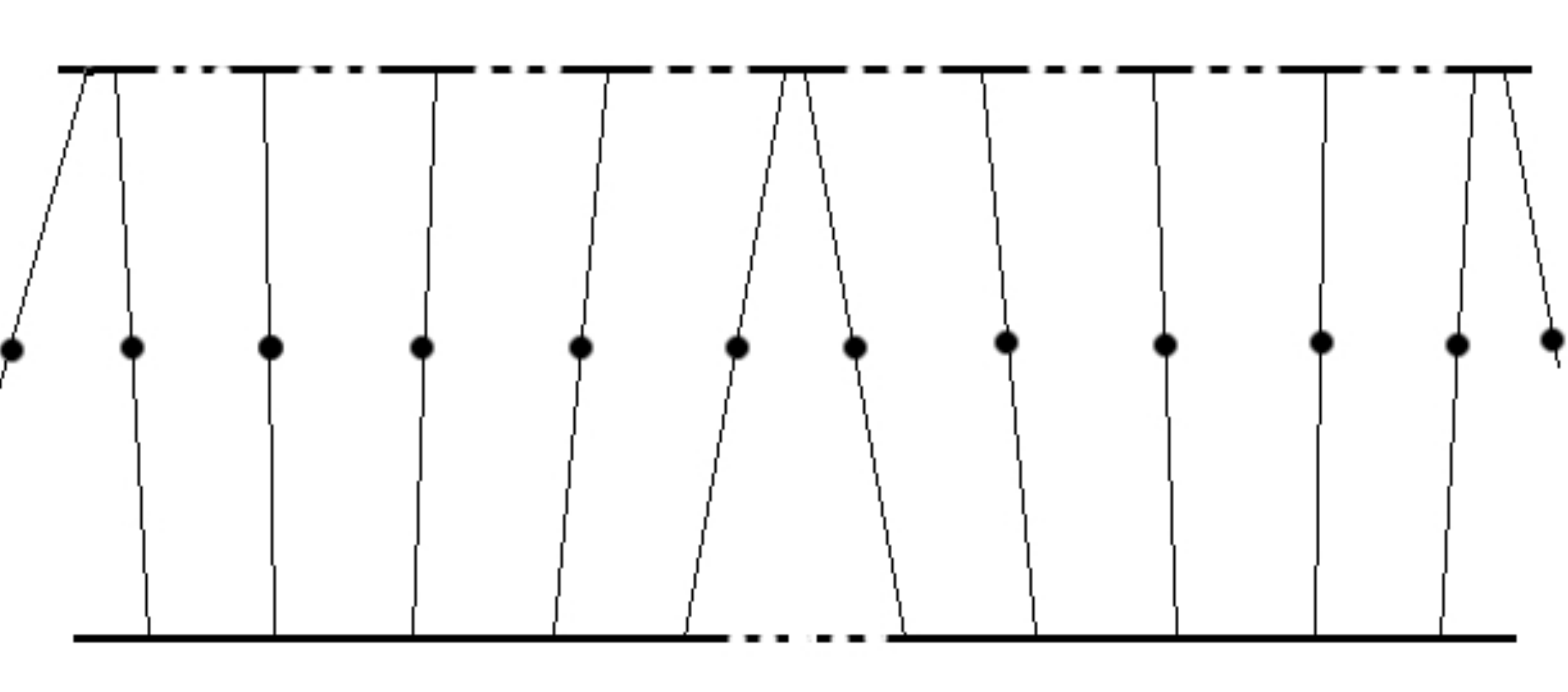}}

\put(0,13){\footnotesize $\kmid(m)-1 $}
\qbezier(7,21)(3,18)(2,16)
\put(42,15){\footnotesize $\kmid(m-1) $}
\qbezier(41,21)(42,18)(45,16)
\put(48,30){\footnotesize $\kmid(m-1)-1 $}
\qbezier(52,29)(52,25)(48,23)

\put(28,30){\footnotesize $\kmin(m) $}
\qbezier(30,29)(28,25)(25,23)

\put(16,20){\footnotesize $k$}
\put(9,39){\footnotesize $\evDout_{k-m}$}

\put(20,0){\footnotesize $\evDin_m $}
\put(60,0){\footnotesize $\evDin_{m-1}$}
\end{picture}
\end{center}
\caption{The curves $\evgamma_k$ starting at $\evDin_m $ and $\evDin_{m-1}$ for fixed $m $. Exactly for $\kmin(m) \le k \le \kmid(m)-1 $, the curves ${}^{\text{odd}} \widetilde \gamma_k$ have the same behaviour as $\evgamma_k$.}\label{figuremminus1odd}
\end{figure}

Consequently, the situation between $\kmid(m-1)-1$ and $\kmin(m)$ for a given $m $ gives the exact same picture as before, \cf Figure \ref{figuremminus1} and Figure \ref{figuremminus1odd}. We now can proceed as in the proof of Lemma \ref{lem:Lmatrix10}:

Replacing $\oddtfS_k$ with $\evfS_k$ when $\ellmin(m) \leq k<\ellmid(m)-1$ and thus identifying $\oddtgamma_k$ with $\evgamma_k$ for these $k$, we obtain that the restriction of $\sigma^{\ell+1}_\ell $ to the subspace
$$
\bigoplus_{k=\ellmin(m)}^{\ellmid(m)-1} \cK^{\ell+1}_{\xi^k}
$$
is the identity. For $\ellmid(m-1)-1 < k < \ellmin(m) $, we have
$$
\evgamma_k: \evconn{m}{k}{k-m} \text{\quad and \quad} \oddtgamma_k: \evconn{m-1}{k}{k-m+1}.
$$
Comparing with \eqref{eq:alphaevS}, we see that the appropriate sector of the annulus reproduces the same situation as in the proof of Lemma \ref{lem:Lmatrix10}, \cf Figure \ref{figure:lmaxmid1new}.

\begin{figure}[htb]
\begin{center}
\begin{picture}(80,40)(0,0)
\put(0,3){\includegraphics[scale=0.38]{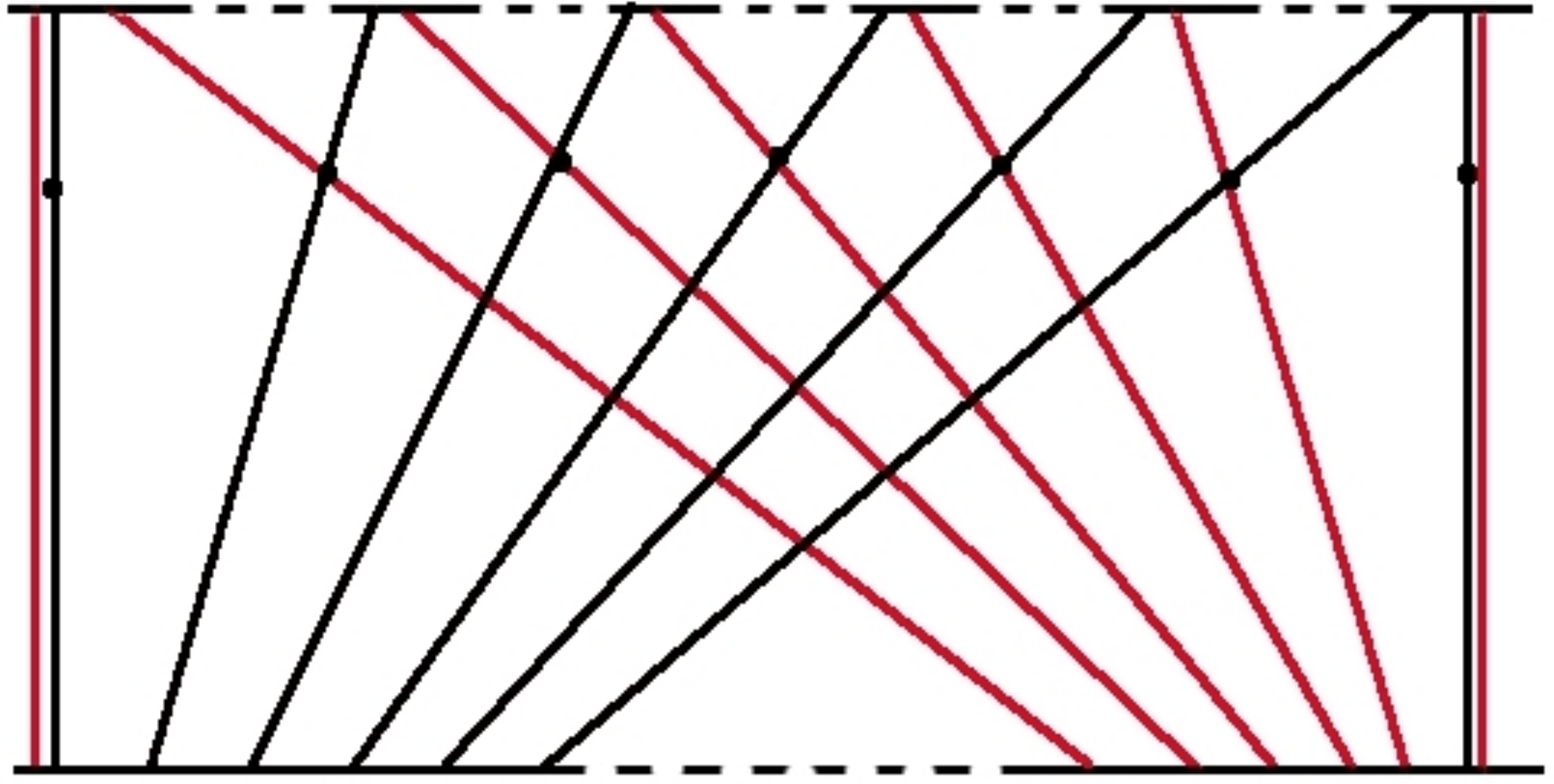}}

\put(-10,30.5){\footnotesize $\ellmin(m)$}
\put(70,30.5){\footnotesize $\ellmid(m-1)-1$}
\put(17,31){\footnotesize $k$}

\put(3,20){\footnotesize $\evgamma_k$}
\qbezier(13,26)(11,26)(6,22)
\put(6,28){\footnotesize $\oddtgamma_k$}
\qbezier(8,30.5)(9,34)(10,34.5)

\put(10,0){\footnotesize $\evDin_m$}
\put(60,0){\footnotesize $\evDin_{m-1}$}
\end{picture}
\end{center}
\caption{The sector $\fS$ and the curves $\evgamma_k$ and $\oddtgamma_k$ in the range $[\ellmid(m-1)-1, \ellmin(m)] $.}\label{figure:lmaxmid1new}
\end{figure}

Defining
$$
s:=\kmin(m)-\kmid(m-1)+1 \ge 2
$$
the resulting matrix is the matrix ${}^\odd\bS(m)\in \mathrm{Mat}_{(s+1)\times(s+1)}(\End(V)) $ with exactly the same design as the one defined in \eqref{eq:Smatrix} before. Therefore, we finally obtain the analogous result to Lemma \ref{lem:Lmatrix10}:

\begin{lemme}\label{lem:Lmatrix10odd}
The isomorphism $\sigma^{\ell+1}_\ell:\bigoplus_{k=0}^{p+q-1} V \otimes \mathbf{1}_k \to\bigoplus_{k=0}^{p+q-1} V \otimes \mathbf{1}_k$ does not depend on $\ell $ if $\ell $ is odd and is the linear map decomposing into diagonal blocks as follows:
\begin{enumerate}
\item $\id_{[\kmin(m), \kmid(m)-1)}$ (maybe empty), $m=0, \ldots, q-1 $,
\item $^\odd\bS(m)_{[\kmid(m-1)-1,\kmin(m)] }$ (of size at least $3$), $m=1, \ldots, q-1 $,
\item $\diag(1,\ldots, 1, T) \cdot {}^\odd\bS(0)_{[\kmid(q-1)-1, \kmin(0)]} \cdot \diag(1, \ldots, 1, T^{-1}) $
\end{enumerate}
\end{lemme}

\begin{remarque} The case $q<p $ can again be obtained in an analogous way interchanging the role of the inner and outer boundary (\cf Remark \ref{rem:elloddev}). With the same arguments as in Remark \ref{rem:elloddev} we verify that $\sigma^{\ell+1}_\ell $ for $\ell $ odd coincides with $\sigma_\odd^\ev $ of \eqref{eq:defL01}.
\end{remarque}

\subsubsection{The case $p=q=1 $} \label{subsec:pq1}

Then the Leray coverings $\evfS $ and $\oddfS $ are given as in Figure \ref{fig:pq1}. The curves $\evtgamma_k$ and $\oddtgamma_k$ obtained by the appropriate isotopy are also given in the picture.

\begin{figure}[htb]
\begin{center}
\begin{picture}(122,40)(0,0)
\put(5,0){\includegraphics[scale=0.3]{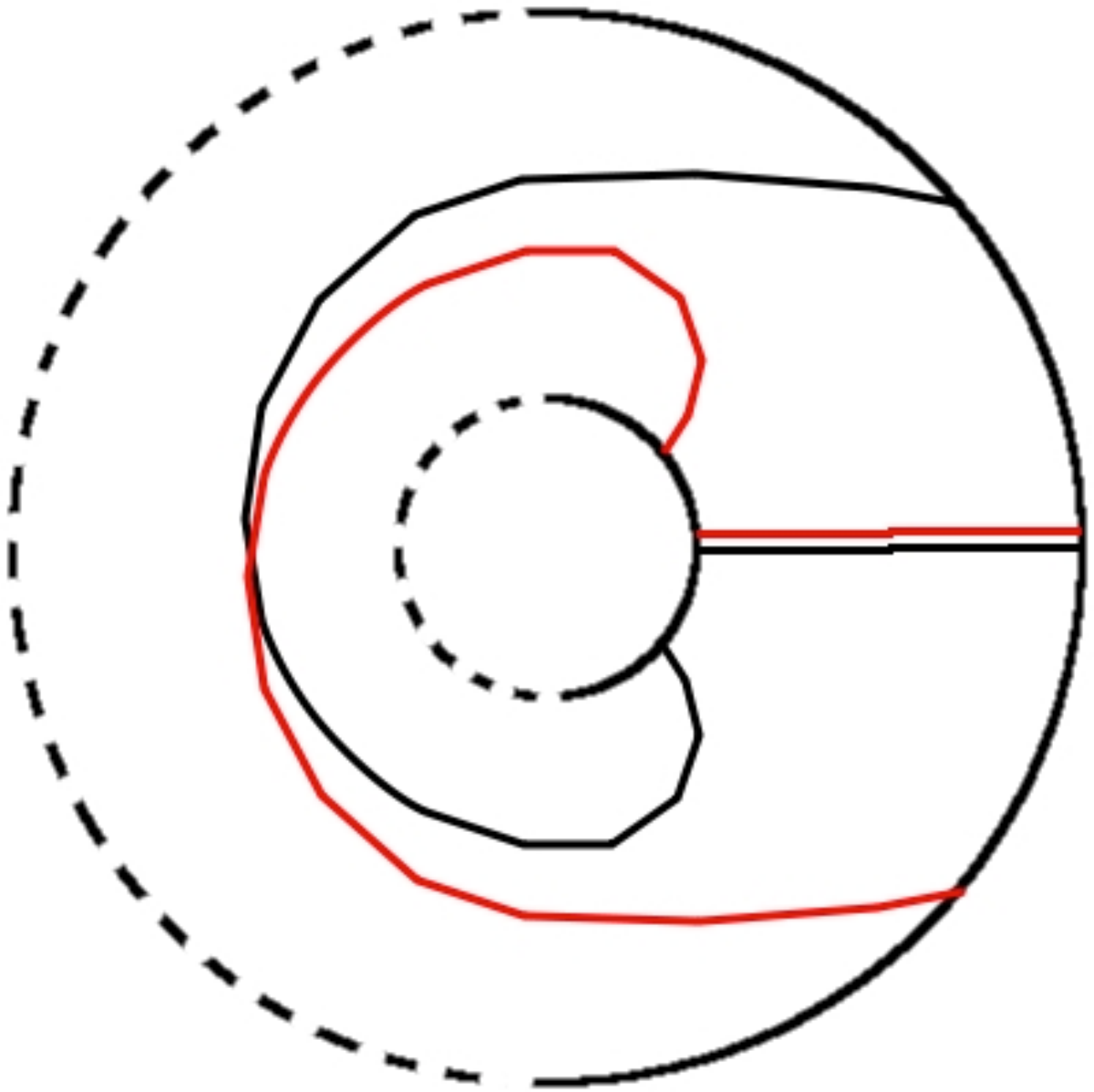}}
\put(75,0){\includegraphics[scale=0.3]{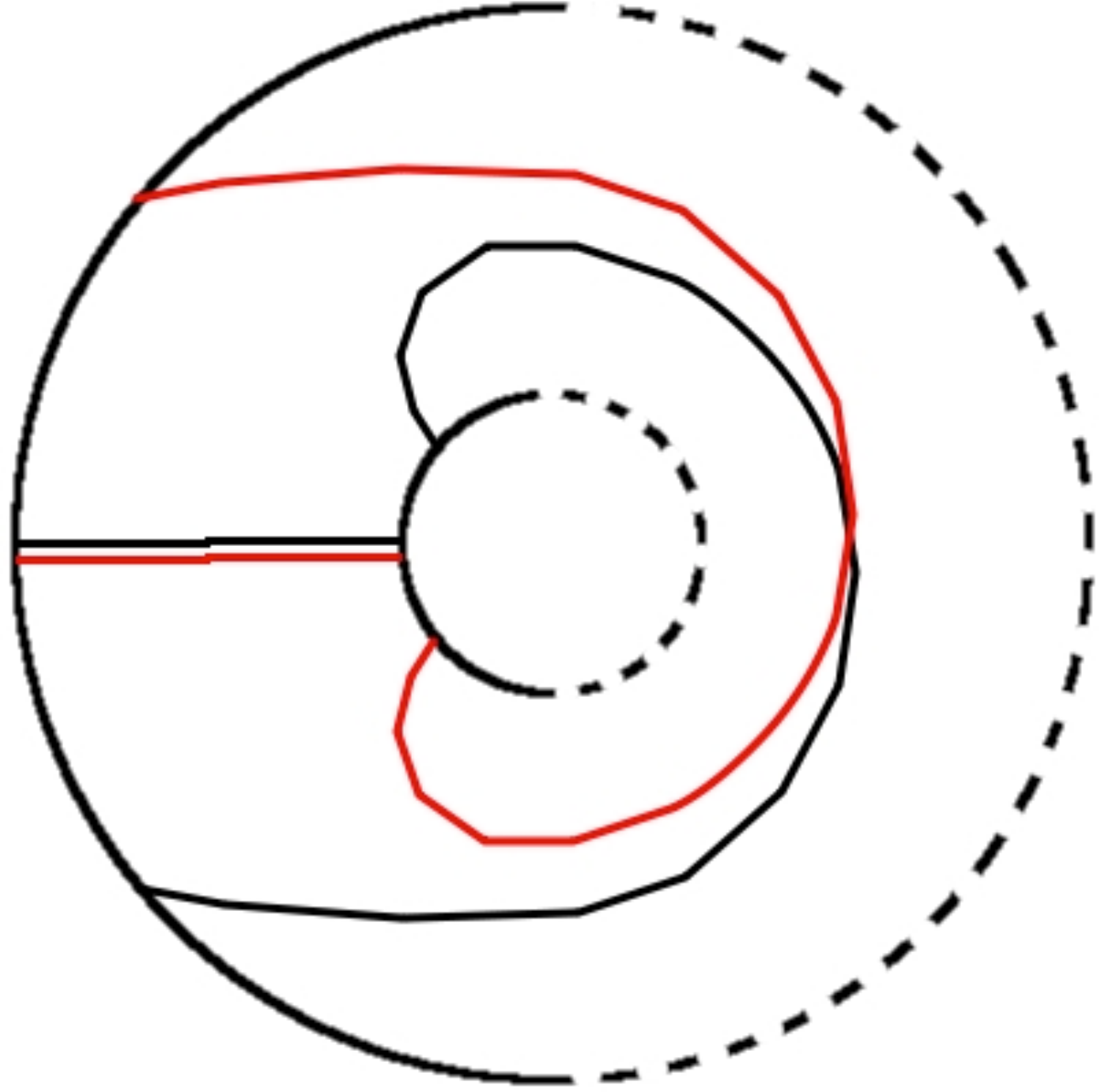}}

\put(25,37){\footnotesize $\evgamma_1 $}
\put(38,18){\footnotesize $\evgamma_0 $}
\put(38,23){\footnotesize $\oddtgamma_0 $}
\put(25,4){\footnotesize $\oddtgamma_1 $}
\put(35,30){\footnotesize $B_0 $}
\put(35,12){\footnotesize $B_1 $}
\put(16,20){\footnotesize $B_+ $}
\put(8,20){\footnotesize $B_- $}

\put(92,37){\footnotesize $\evtgamma_0 $}
\put(78,17){\footnotesize $\evtgamma_1 $}
\put(78,23){\footnotesize $\oddgamma_1 $}
\put(92,4){\footnotesize $\oddgamma_0 $}
\put(85,30){\footnotesize $B_0 $}
\put(85,12){\footnotesize $B_1 $}
\put(111,20){\footnotesize $B_- $}
\put(102,20){\footnotesize $B_+ $}

\put(45,3){\footnotesize $\evB $}
\put(73,3){\footnotesize $\oddB $}
\end{picture}
\end{center}
\caption{The Leray covering $\evfS $ (black) and the curves $\oddtgamma_k$ (red) on the left hand side, the Leray covering $\oddfS $ (black) and the curves $\evtgamma_k$ (red) on the right hand side.}\label{fig:pq1}
\end{figure}

By Remark \ref{rem:heur}, the heuristic rule provides an expectation on the form of $\sigma_\ev^\odd $ and $\sigma_\odd^\ev $ (namely to be the one defined in \ref{def:L}). These expectations can be verified by the analogous procedure used in the proof of Lemma \ref{lem:Lmatrix10}. Let us denote by $\fA:=(A_0, A_1) $ the closed covering induced by the curves $\evgamma_k$ and by $\fC:=(C_0,C_1)$ the one induced by the curves $\oddtgamma_k$. Furthermore, let $\fB $ be the canonical common (ordered) refinement $\fB=(B_0, B_1, B_-, B_+) $ (\cf Figure \ref{fig:pq1}), such that $A_0=B_0 \cup B_+ $, $A_1=B_1 \cup B_- $, $C_0=B_0 \cup B_- $ and $C_1=B_1 \cup B_+ $. This refinement allows to compute the change of the basis from the \Cech complex associated to $\fA $ to the one associated to $\fC $ as we did in the proof of Lemma \ref{lem:Lmatrix10}. The resulting presentation of the base change gives the map $\sigma^\ev_\odd $ as in Definition \ref{def:L}. The same arguments apply for $\sigma^\odd_\ev $ using the picture on the right hand side of Figure~\ref{fig:pq1}.

\subsubsection{Standard Stokes data}

We have now obtained the following result:
\begin{proposition}\label{prop:standStokes}
Let $p,q\in\N $ be two co-prime numbers and let $(\bV, \bT)$ be a vector space with an automorphism. The topological model of Definition \ref{def:topmodel} is isomorphic as a set of linear Stokes data to the standard linear Stokes data of Definition \ref{def:standstokesintro}.\qed
\end{proposition}

\subsection{Explicit computation of \texorpdfstring{$\protect\wh T_\top$}{Ttop}}\label{subsec:topmono}
We now prove Proposition \ref{prop:topmonointro}. In principle, the topological monodromy of the Laplace transform $\wrho\ccF^{(0,\infty)}\ccM$ around $\wh\infty $ can be obtained once we determined its Stokes data (see Remark \ref{rem:topmdrmy}). However, using the techniques developed above in this section, it is possible to compute in a straightforward way the monodromy. In particular, we can work on the space $\P^1_u \times \Af^1_\eta $ and do not have to use the blow-up of $(0, \wh\infty)$ -- see sections \ref {subsec:simplifassumpt} and \ref{sec:caseofelement} for the notation.

Let $\wt\PP^1_u$ be the closed annulus obtained as the real blow up of $0,\infty$ in $\PP^1_u$. We set $\vt=\arg u$ as in \S\ref{subsec:GGpsi}. The fiber at $\tau_o\!\neq\!0$ of $\wh M$ (\cf\S\ref{subsec:general}) is obtained~as
\begin{equation}\label{eq:H1F}
H^1\bigl(\wt\PP^1_u,\cH^0\DR^{\rmod{(0,\infty)}}(E^{-\vi(u)-u^p/\tau_o}\otimes R)\bigr),
\end{equation}
where $\wh\cF_{\tau_o}:=\cH^0\DR^{\rmod{(0,\infty)}}(E^{-\vi(u)-u^p/\tau_o}\otimes R)$ is the extension $j_*\cK$ of the sheaf~$\cK$ (corresponding to~$R$, \ie defined by the monodromy $\bT$) on the open annulus $\CC^*_u$ to the open part of the boundary defined by
\refstepcounter{equation}\label{eq:inout}
\begin{align}
-q\vt+\arg\vi_q&\in(\pi/2,3\pi/2)\bmod2\pi,\tag*{\eqref{eq:inout}$_\mathrm{in}$}\label{eq:in}\\
p\vt-\arg\tau_o&\in(\pi/2,3\pi/2)\bmod2\pi,\tag*{\eqref{eq:inout}$_\mathrm{out}$}\label{eq:out}
\end{align}
and extended by zero to the remaining part of the boundary. The action of the monodromy $\tau_o\mto e^{2\pi i}\tau_o$ consists only in moving each of the $p$ open intervals \ref{eq:out} (of length $\pi/p$) of the outer boundary in the positive direction to the next one. The intervals \ref{eq:in} can be written as $\evDin_m+(\arg\varphi_q)/q$, $m=0,\dots,q-1$, and it is possible to choose $\tau_o$ so that the intervals \ref{eq:out} are the intervals \hbox{$\evDout_n+(\arg\varphi_q)/q$}, $n=0,\dots,p-1$. For simplicity we assume below that $\arg\varphi_q=0$. We can therefore use the topological model of the annulus $\wA$ with the subset $\evB$ defined by \eqref{eq:evoddbeta} and the sheaf $\cK$ on $A$, and obtain the isomorphism
\[
H^1(\wt\PP^1_u,\wh\cF_{\tau_o})\simeq H^1(\wA,(\evbeta)_!\cK)=H^1_\mathrm{c}(\evB,\cK).
\]
It will be simpler here, since we do not have to control Stokes filtrations, to work with Borel-Moore homology, for which we refer to \cite{B-M60}, \cite{B-H61} and \cite{Bredon97}. We have an isomorphism $H^1_\mathrm{c}(\evB,\cK)\simeq \HBM(\oddB, \cK)$: On the one hand, by Poincaré duality, $H^1_\mathrm{c}(\evB,\cK)\simeq H^1_\mathrm{c}(\oddB,\cK^\vee)^\vee$; on the other hand, the cap product (\cf\cite[Th.\,V.10.4]{Bredon97} -- note that the result in \loccit\ is stated for $\cK$ being of rank one but generalizes to local systems of arbitrary rank) induces an isomorphism $H_1^\BM(\oddB,\cK)\simeq H^1_\mathrm{c}(\oddB,\cK^\vee)^\vee$.

The elements of $\HBM(\oddB,\cK)$ are represented by Borel-Moore cycles of the form
$$
\Big[ \sum^{\mathrm{loc.fin.}}_{\sigma \in \oldDelta_1(\oddB)} \sigma \otimes w_\sigma \Big]
$$
where $\oldDelta_1(\oddB)$ denotes the set of closed simplicial $1$-chains, $w_\sigma \in H^0(\sigma, \cK)$, and the sum being locally finite. Note that a curve $\gamma:(0,1) \to\evB$ such that $\lim_{t\to0} \gamma(t),\,\lim_{t\to 1} \gamma(t) \in \partial(\wA)$ defines a Borel-Moore cycle $\gamma \otimes w\in\HBM(\oddB, \cK)$ if $w \in H^0(\gamma, \cK)$. In particular, there is a well-defined morphism
\begin{equation}\label{eq:BMbasis3a}
\bigoplus_{k=0}^{p+q-1} \cK_{\xi^k} \lto \HBM(\oddB, \cK),
\end{equation}
defined by $w \mto \big[ \evgamma_k \otimes w \big]$ for $w \in \cK_{\xi^k}$. The long exact sequence for the open embedding $\oddB \subset \wA $ (\cite[Th.\,3.8]{B-M60} or \cite[\S V.5]{Bredon97}) yields that this is a basis for $\HBM(\oddB, \cK) $. Let us denote by
\begin{equation}\label{eq:BMbasis3}
\Psi^\BM: \HBM(\oddB, \cK) \isom \bigoplus_{k=0}^{p+q-1} \cK_{\xi^k}
\end{equation}
its inverse. We can therefore use the curves $\evgamma_k$ to compute this homology by way of the isomorphism $\Psi^\BM$.

We fix an isotopy $\wA \times [0,1] \to \wA \times [0,1] $ of the identity to a diffeomorphism $\psi $ of $\wA $ to itself by lifting the vector field $\partial_t $ on $[0,1] $ in a way that the lift equals $\partial_t $ outside a neighbourhood of the outer boundary of $\wA $ and that the diffeomorphism restricted to the outer boundary maps $\evDout_n $ to $\evDout_{n+1}$. Let $\evtgamma_k$ be the images of the curves $\evgamma_k$ by the diffeomorphism $\psi $. We obtain another isomorphism
$$
\wt \Psi^\BM: H_1^\BM(\oddB, \cK) \to\bigoplus_{k=0}^{p+q-1} \cK_{\xi^k} .
$$
The topological monodromy is then represented by $\Psi^\BM \cdot (\wt \Psi^\BM)^{-1}$.

We see that $\evgamma_k$ and $\evtgamma_k$ connect the boundary intervals in the following way:
$$
\evgamma_k: \evconn{m}{k}{n} \quad\text{and}\quad
\evtgamma_k: \evconn{m}{k}{n+1},
$$
with $(m,n)=(\evin(k), \evout(k)) $ as defined in \eqref{eq:cev}.
Fixing $m \in \{0, \ldots, q-1 \}$, we have (\cf \eqref{eq:evkm})
$$
\evin(k)=m \Longleftrightarrow k \in\big[ \ellmid(m-1), \ellmid(m)-1 \big],
$$
where $\ellmid(m)= \lceil \tfrac{(2m+1)(p+q)}{2q} \rceil $. Therefore, $\evtgamma_k$ and $\evgamma_{k+1}$ connect the same boundary intervals $\evDin_m $ and $\evDout_{n+1}$ for all
\begin{equation}\label{eq:intervl}
k \in\big[ \ellmid(m-1), \ellmid(m)-2\big]
\end{equation}
(which is possibly empty). Note, that if $p+q \ge 3 $, there is at least one $m $ for which this interval is non-empty (\cf Lemma \ref{lem:3xis}). The case $p=q=1 $ was completely determined in \S\ref{subsec:pq1}. The situation is sketched in Figure \ref{leraymdrmy} and \ref{leraymdrmy2}.

\begin{figure}[htb]
\centerline{The Borel-Moore cycles $\evgamma_k$ (black) and $\evtgamma_k$ (red).}\par\smallskip
\begin{minipage}[t]{.46\textwidth}
\begin{center}\setlength{\unitlength}{.9mm}
\begin{picture}(65,35)
\put(0,5){\includegraphics[scale=0.28]{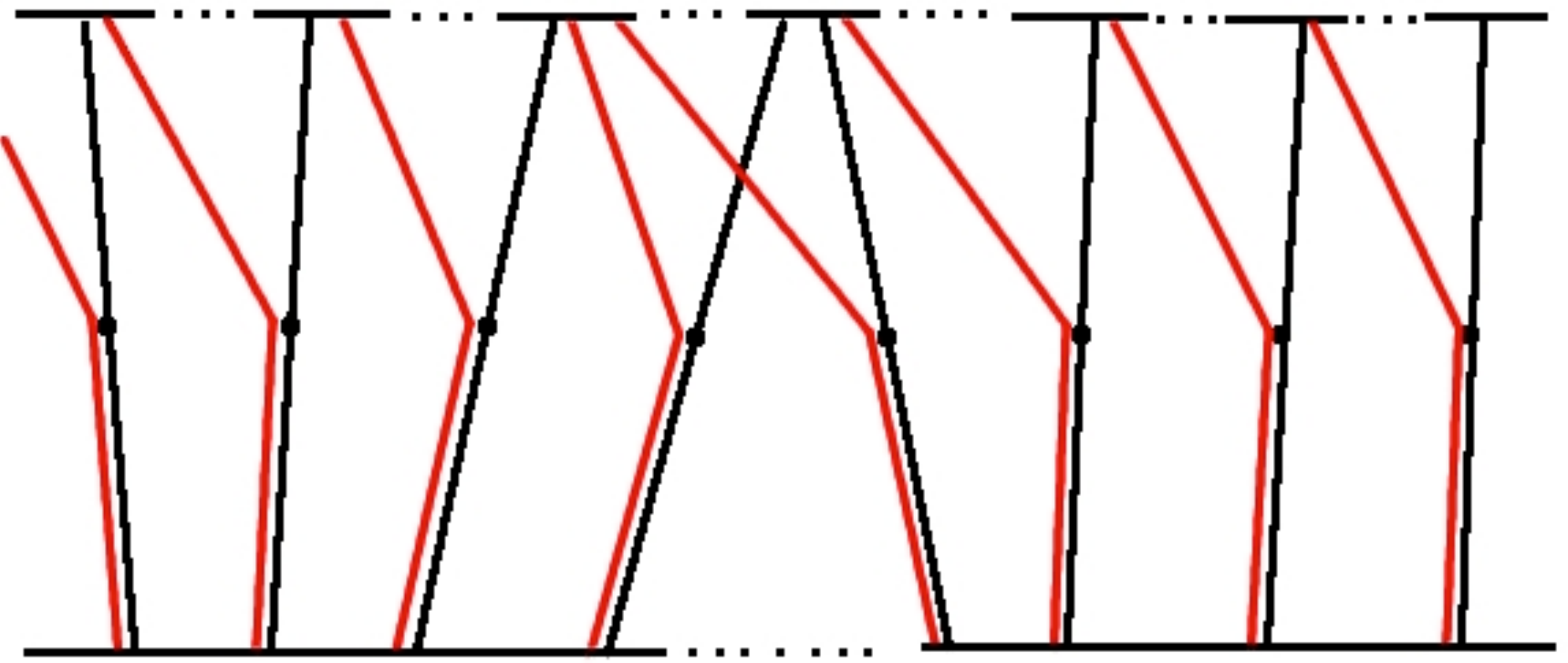}}

\put(10,0){\footnotesize $\evDin_{m+1}$}
\put(45,0){\footnotesize $\evDin_{m}$}

\put(44,17){\footnotesize $k$}
\put(36,32){\footnotesize $\evtgamma_k$}
\qbezier(36,28)(37,28)(38,31.5)
\put(46,32){\footnotesize $\evgamma_k$}
\qbezier(44,25)(45,27)(47,32)
\put(25,0){\footnotesize $\ellmid(m) $}
\qbezier(28,3)(29,10)(28,17)
\end{picture}
\caption{The case $q<p $, the interval \eqref{eq:intervl} always containing at least one integer.}\label{leraymdrmy}
\end{center}
\end{minipage}
\hfill
\begin{minipage}[t]{.46\textwidth}
\begin{center}\setlength{\unitlength}{.9mm}
\begin{picture}(65,35)
\put(0,4){\includegraphics[scale=0.28]{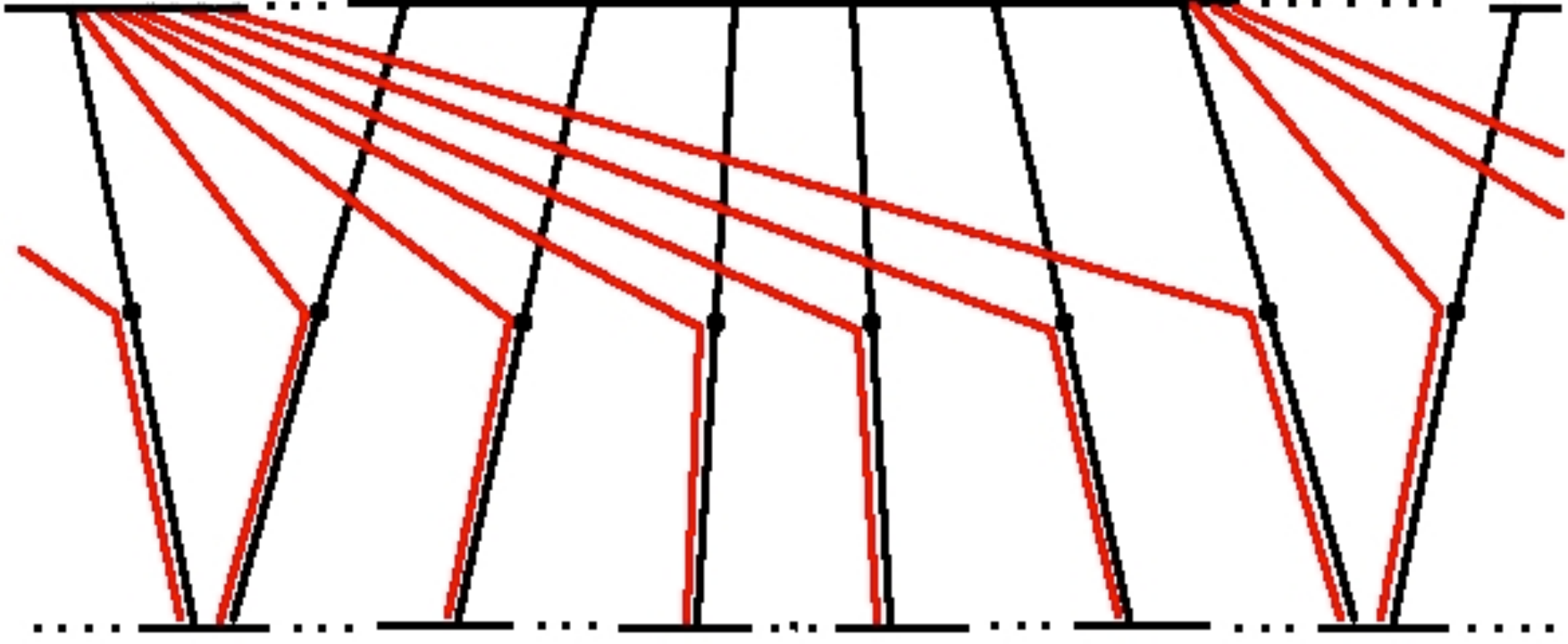}}

\put(12,0){\footnotesize $\evDin_{m+1}$}
\put(24,0){\footnotesize $\evDin_{m}$}
\put(6,1){\footnotesize $(\ast)$}
\put(54,1){\footnotesize $(\ast)$}

\put(44,16){\footnotesize $k$}
\put(25,32){\footnotesize $\evtgamma_k$}
\qbezier(31,32)(37,28)(38,19)
\put(42,32){\footnotesize $\evgamma_k$}
\qbezier(44,32)(44,27)(43,23)
\put(-8,7){\footnotesize $\max\nolimits_\mathrm{out}(k) $}
\qbezier(-3,10)(5,15)(13,17)
\end{picture}
\caption{The case $q>p $, the interval \eqref{eq:intervl} containing at most one integer. Here, it is non-empty only for the values of $m $ where $\evDin_m $ is marked with an $(\ast) $ .}\label{leraymdrmy2}
\end{center}
\end{minipage}
\end{figure}
Now, the assertion of Proposition \ref{prop:topmonointro} can directly be read off from these figures, the index
$\max\nolimits_\mathrm{out}(k)$ denoting the index $k' \ge k$ such that $\mathrm{out}(k')=\mathrm{out}(k) $, but $\mathrm{out}(k'+1)\neq \mathrm{out}(k) $. Note, that in the case $q<p $, we have
$$
\max\nolimits_\mathrm{out}(k) =
\begin{cases}
k & \text{if } k \neq \ellmid(m)-1 \\
\ellmid(m) & \text{if } k=\ellmid(m)-1
\end{cases}
$$
with $m=\mathrm{in}(k)$ (\cf Figure \ref{leraymdrmy}).

In case $q>p $, we have $\max\nolimits_\mathrm{out}(k) = k$ if and only if $\mathrm{in}(k+1)=\mathrm{in}(k)$ (where $k$ is the lower index of the two curves starting at the marked $(\ast)$ intervals in Figure~\ref{leraymdrmy2}).
\qed

\section{Stokes data for \texorpdfstring{$\protect\wrho\ccF^{(0,\infty)}\ccM$}{FM}}\label{finalsection}

We will now prove Theorem \ref{thm:mainintro} asserting that the linear Stokes data attached to $\wrho\ccF^{(0,\infty)}\ccM$ are isomorphic to the standard model of \S\ref{subsec:modelStokes}. We are left to proving that for each $\ell\in\{0,\dots,2q-1\}$, the isomorphism $\Psi_\ell$ in \eqref{eq:Labstrsimple} is compatible with the filtrations of \eqref{eq:Filell} and of Definitions \ref{def:standfilev} and \ref{def:standfilodd}. For that purpose, we will modify the construction of the Leray coverings $\evfS$ \resp $\oddfS$ in such a way that the new Leray coverings moreover induce Leray coverings on each $\evB_{\omega_{\wzeta>0}}$ \resp $\oddB_{\omega_{\wzeta>0}}$ ($\wzeta\in\mu_{p+q}$).

\subsection{The even case}
Throughout this subsection, we use the even order $\lev$ on $\mu_{p+q}$ (\cf Definition \ref{def:standardorder}) and we simply denote it by $<$. The key observation is the following Lemma:

\begin{lemme}\label{lem:alphaxisneu}
There exists a family $\{\gamma_{\wzeta} \mid \wzeta\in\mu_{p+q}\}$ of smooth curves in $\wA$ with the following properties:
\begin{enumerate}
\item\label{lem:alphaxisneu1}
$\gamma_{\wzeta}$ passes through $\wzeta$,
\item\label{lem:alphaxisneu2}
$\gamma_{\wzeta} \subset \evB_{\omega_{\wzeta}>0}$,
\item\label{lem:alphaxisneu3}
setting $\wzeta=\exp(\frac{2k\pi i}{p+q})$ with $k\in\{0,\dots,p+q-1\}$, the end points of $\gamma_{\wzeta}$ belong to the intervals $\evIin{m},\evIout{n}$ with $(m\bmod q,n\bmod p)=\cev(k)$,
\item\label{lem:alphaxisneu4}
for $\wzeta'\neq \wzeta$, we have $\gamma_{\wzeta'}\cap\gamma_{\wzeta}=\emptyset $.
\end{enumerate}
\end{lemme}

\begin{proof}
For $\ell$ even, we have $B_{\omega_{\wzeta}>0}^{\vtn_\ell}=\evB_{\omega_{\wzeta}>0}$, according to \eqref{eq:Bfiltre}. Lemma \ref{lem:G1} shows that there exist exactly two connected components of $G^{-1}\bigl((-\frac{\pi}{2}, \frac{\pi}{2})\bigr)$ containing the point $x=1$ in their closure: these are the components containing $\evDin_0$ and $\evDout_0$ (see Figures \ref{figure:fibreDelta} and \ref{figurerd1}). Let us denote by $B$ the union of their closures.

We will construct paths $\gamma_{\wzeta}$ of the form $\gamma_{\wzeta}=\wzeta \cdot \ogamma_{\wzeta}$ for some $\ogamma_{\wzeta}$ connecting $\evDin_0$ and $\evDout_0$ via $x=1$ (so that \eqref{lem:alphaxisneu1} will be fulfilled) inside the region $B$ and contained in the fiber $G^{-1}(\theta_{\wzeta})$ for some \angle $\theta_{\wzeta}\in (- \frac{\pi}{2}, \frac{\pi}{2})\cap (q\arg(\wzeta)+(-\frac{\pi}{2}+\ve, \frac{\pi}{2}+\ve))$ (so that \eqref{lem:alphaxisneu2} will be fulfilled, according to \eqref{eq:Bfiltre}; note that this intersection is not empty). We will then have
\begin{align*}
\gamma_{\wzeta}\cap\bin\ov A&\subset \wzeta\cdot\evDin_0\cap\bin(\evB_{\omega_{\wzeta}>0}),\\
\tag*{and}
\gamma_{\wzeta}\cap\bout\ov A&\subset \wzeta\cdot\evDout_0\cap\bout(\evB_{\omega_{\wzeta}>0}),
\end{align*}
so \eqref{lem:alphaxisneu3} will be fulfilled according to \eqref{eq:dB} and Remark \ref{rem:cev}\eqref{rem:cev2}.

We are thus reduced to find $(\theta_{\wzeta})_{\wzeta\in\mu_{p+q}}$ in order to fulfill \eqref{lem:alphaxisneu4}. In the following, we will identify the paths $\gamma$ and their images, and we will consider the set
$$
\arg(\gamma):=\{\arg(x) \mid x\in\gamma\}.
$$
According to the ordering at $\vtn_o$ (\cf\S\ref{subsec:Bwzeta}), we will inductively find $\theta_{\wzeta}$ with the slightly stronger condition (preparing for a variant necessary later) that
\begin{equation}\label{eq:thetachi}
\theta_{\wzeta}\in \Bigl({-} \frac{\pi}{2}+\ve, \frac{\pi}{2}\Bigr)\cap \Bigl(q\arg(\wzeta) +\ve +\Bigl({-}\frac{\pi}{2}, \frac{\pi}{2}- \frac{\pi}{p+q}\Bigr)\Bigr),
\end{equation}
assuming that $q\arg({\wzeta}) \neq \pi$ (in which case the intersection is non-empty). The case $q\arg({\wzeta}) =\pi $ appears exactly when $p+q $ is even and $\wzeta=-1=:\ximaxodd$ is the maximal element in $\mu_{p+q}$ with respect to the even ordering. Therefore, we can also use Condition \eqref{eq:thetachi} in the inductive argument in the case $p+q \equiv 0\bmod2 $ if we apply the induction hypotheses to all $\wzeta < \ximaxodd$ only (recall that $<$ means $\lev$), and it is enough to construct an appropriate $\theta_{\ximaxodd}$ at the end of the induction process.

Assuming we already have found $\gamma_{\wzeta'}$ of the above form for $\wzeta'\in\{\wzeta_0,\dots, \wzeta_{k-1}\}$, the task then is to find an \angle $\theta=\theta_{\wzeta}$ with $\wzeta=\wzeta_k$, giving rise to the path $\ogamma_{\wzeta} \subset B$ (whose image is then uniquely determined by $\ogamma_{\wzeta}=G^{-1}(\theta)\cap B$) with the properties:\refstepcounter{equation}\label{eq:thetaBig}
\begin{align*}
\tag*{\eqref{eq:thetaBig}(i)}\label{eq:thetaBigi}
\theta&\in q\arg(\wzeta)+\ve+\Bigl({-} \frac{\pi}{2}, \frac{\pi}{2}- \frac{\pi}{p+q}\Bigr),\text{ and}\\
\tag*{\eqref{eq:thetaBig}(ii)}\label{eq:thetaBigii}
\theta&\in \Bigl({-}\frac{\pi}{2}+\ve,\frac{\pi}{2}\Bigr) \smallsetminus\bigcup_{{\wzeta'}\in\{\wzeta_0,\dots, \wzeta_{k-1}\}}\arg(\wzeta^{-1} \cdot \gamma_{\wzeta'}).
\end{align*}
Condition \ref{eq:thetaBigi} ensures that $G^{-1}(\theta) \subset \wzeta^{-1} \cdot \evB_{\omega_{\wzeta}>0}$ (\cf \eqref{eq:Bfiltre}) and Condition \ref{eq:thetaBigii} that~$\theta $ is such that $\ogamma_{\wzeta}:=G^{-1}(\theta)\cap B$ does not intersect any of the curves $\wzeta^{-1} \gamma_{\wzeta'}=\wzeta^{-1} {\wzeta'}\,\ogamma_{\wzeta'}$ for ${\wzeta'}<\wzeta$, which forces $\gamma_{\wzeta}\cap\gamma_{\wzeta'}=\emptyset $ for all these ${\wzeta'}$.

To this end, consider a ${\wzeta'}\in\{\wzeta_0,\dots, \wzeta_{k-1}\}$. Then by induction, we have $\ogamma_{\wzeta'}=G^{-1}(\theta_{\wzeta'})\cap B$ for some $\theta_{\wzeta'}$ as desired. In particular, each $x\in\ogamma_{\wzeta'}$ satisfies
$$
G(x)=\arg\bigl(g(x)\bigr)=\theta_{\wzeta'}\in \Bigl({-}\frac{\pi}{2}+\ve, \frac{\pi}{2} \Bigr)\qquad (g(x):=f(x)/x^q,\text{ \cf\eqref{eq:fX}}).
$$
Consequently, each point in $\wzeta^{-1} {\wzeta'}\,\ogamma_{\wzeta'}$ is of the form $\wzeta^{-1}{\wzeta'}x$ for such an $x$ and satisfies
\begin{equation} \label{mult:Gxi}
G(\wzeta^{-1}{\wzeta'}x)=\arg\bigl(g(\wzeta^{-1}{\wzeta'}x)\bigr)\\
=q\arg(\wzeta)+\arg\bigl(\wzeta^{\prime p}g(x)+(p+q) (\wzeta^{\prime p}-\wzeta^p)\bigr).
\end{equation}

\begin{remarque}\label{rem:argpi2}
Since ${\wzeta'}<\wzeta$, we know that $\arg(\wzeta^{\prime p} - \wzeta^p)\in (-\frac{\pi}{2}, \frac{\pi}{2}] $. Using the notation of Definition \ref{def:standardorder}, the pairs $(\wzeta',\wzeta)$ such that $\wzeta'<\wzeta$ and $\arg(\wzeta^{\prime p}-\wzeta^p)=\sfrac{\pi}{2}$ are the pairs $(\xi^{ka},\xi^{-ka})$ with $k\in[1,\frac{p+q}2)\cap\NN$. Indeed, set $\wzeta'=\xi^{k'a}$ and $\wzeta=\xi^{-ka}$ with $k'\in(-\frac{p+q}2,\frac{p+q}2)\cap\ZZ$ and $k\in[-\frac{p+q}2,\frac{p+q}2)\cap\ZZ$. Then $\wzeta^{\prime p}-\wzeta^p=\xi^{k'}-\xi^{-k}$, which has argument $\sfrac{\pi}{2}$ if and only if the conditions above are fulfilled. We then have $p\arg\wzeta\in(\pi,2\pi-\frac{2\pi}{p+q}]$.
\end{remarque}

Additionally, since $x\in\ogamma_{\wzeta'}$, we know from \ref{eq:thetaBigi} that
$$
p\arg(\wzeta')+\theta_{\wzeta'}\in\ve+\Bigl({-}\frac{\pi}{2}, \frac{\pi}{2}- \frac{\pi}{p+q}\Bigr).
$$
Notice also that if $\arg(\wzeta^{\prime p}-\wzeta^p)\in (-\frac{\pi}{2}, \frac{\pi}{2})$, we have more accurately
$$
\arg(\wzeta^{\prime p}-\wzeta^p)\in \Big[-\frac{\pi}{2}+\frac{\pi}{p+q}, \frac{\pi}{2}- \frac{\pi}{p+q}\Big].
$$
In such a case, the set
$
\{\wzeta^{\prime p}g(x)+(p+q) (\wzeta^{\prime p}-\wzeta^p) \mid x\in\ogamma_{\wzeta'}\}
$
is contained in a half-line starting at the point $(p+q)(\wzeta^{\prime p}-\wzeta^p)$ -- with fixed \angle inside $[-\frac{\pi}{2}+\frac{\pi}{p+q}, \frac{\pi}{2}- \frac{\pi}{p+q}] $ -- and having direction $p\arg(\wzeta')+\theta_{\wzeta'}\in\ve+(-\frac{\pi}{2}, \frac{\pi}{2}- \frac{\pi}{p+q})$.

Let us set
\begin{equation}\label{eq:Thetadef}
\Theta_{\wzeta'}^{\wzeta}:=\{
\arg\bigl(\wzeta^{\prime p} r e^{i\theta_{{\wzeta'}}}+(p+q)(\wzeta^{\prime p}-\wzeta^p)\bigr) \mid r\in\R_{\ge 0}\}.
\end{equation}
We deduce that there exists $\ve'>0$ such that
\begin{equation}\label{eq:Theta}
\Theta_{\wzeta'}^{\wzeta}\subset
\begin{cases}\dpl
\ve+\Bigl({-}\frac{\pi}{2}+\ve', \frac{\pi}{2}- \frac{\pi}{p+q}- \ve'\Bigr)&\text{if }\arg(\wzeta^{\prime p}-\wzeta^p)\neq\sfrac\pi2,\\[8pt]
\dpl\Bigl({-}\frac{\pi}{2}+\ve+\ve', \frac{\pi}{2}\Big]&\text{if }\arg(\wzeta^{\prime p}-\wzeta^p)=\sfrac\pi2.
\end{cases}
\end{equation}
Let us set $\Theta^{\wzeta}:=\bigcup_{\wzeta'<\wzeta}\Theta_{\wzeta'}^{\wzeta}$. Now, there are two cases to distinguish. Let us consider the interval
\[
\Sigma:= \Bigl(p\arg(\wzeta)+\Bigl({-} \frac{\pi}{2}+\ve, \frac{\pi}{2} \Bigr)\Bigr)\cap\Bigl(\ve+ \Bigl({-}\frac{\pi}{2}, \frac{\pi}{2}- \frac{\pi}{p+q} \Bigr) \Bigr).
\]
Then $\Sigma\neq\emptyset$, and only two cases can occur for the relative position of the two intervals.
\subsubsection*{Case 1}
Assume that
$$
\Sigma=\ve+\Bigl(p\arg(\wzeta) - \frac{\pi}{2}, \frac{\pi}{2}- \frac{\pi}{p+q}\Bigr).
$$
In this case we have $p\arg\wzeta\in(0,\pi-\frac{\pi}{p+q})$, hence Remark \ref{rem:argpi2} implies that there is no $\wzeta'<\wzeta$ such that $\arg(\wzeta^{\prime p}-\wzeta^p)=\sfrac\pi2$, so that the first line of \eqref{eq:Theta} applies for all $\wzeta'<\wzeta$ and thus there exists $\ve'>0$ such that
\[
\Theta^{\wzeta}\subset\ve+\Bigl({-}\frac{\pi}{2}+\ve', \frac{\pi}{2}- \frac{\pi}{p+q}- \ve'\Bigr).
\]
It follows that $\Sigma\not\subset\Theta^{\wzeta}$ and we can find a
$\theta\in q\arg(\wzeta)+\bigl(\Sigma\smallsetminus\Theta^{\wzeta}\bigr)$,
as desired (\cf \eqref{eq:thetaBig}).

\subsubsection*{Case 2}
Assume that
\[
\Sigma=\Bigl({-}\frac{\pi}{2}+\ve, p\arg(\wzeta)+\frac{\pi}{2}\Bigr),
\]
that is, $\pi+\ve<p\arg\wzeta<2\pi-\frac\pi{p+q}+\ve$. From the second line of \eqref{eq:Theta} we deduce that $\Theta^{\wzeta}\subset(-\frac\pi2+\ve+\ve',\frac\pi2]$, so that $\Sigma\not\subset\Theta^{\wzeta}$, and we conclude as in Case 1.

It remains to consider the special case when $p+q$ is even and $\wzeta=\ximaxev=-1$. By the first part of the proof, we can assume that $\theta_{\wzeta'}$ satisfies \ref{eq:thetaBigi} and \ref{eq:thetaBigii} for each $\wzeta'<\wzeta$. The first line of \eqref{eq:Theta} applies to $\Theta^{\wzeta}$. Now, we do not need the stronger condition \eqref{eq:thetaBig} to hold for $\wzeta$ since we do not have to continue with the induction. Therefore, we consider the interval
\begin{equation}\label{eq:speccasepi}
\Sigma' := \underbrace{p \arg(\wzeta)}_{=\pi} + \Bigl(-\frac{\pi}{2}, \frac{\pi}{2} \Bigr) \cap \Bigl(\ve+ \Bigl(-\frac{\pi}{2}, \frac{\pi}{2}\Bigr) \Bigr) = \Bigl(\frac{\pi}{2}, \frac{\pi}{2} +\ve \Bigr)
\end{equation}
instead of $\Sigma$. We conclude again that $\Sigma' \not\subset \Theta^{\wzeta}$ and hence we will find $\theta_{\wzeta} \in q\arg(\wzeta)+(\Sigma' \smallsetminus \Theta) $.
\end{proof}

\begin{remarque}
The case $q=1$ and $p \equiv 1 \text{ mod } 2 $ is covered by the special case above. Due to \eqref{eq:speccasepi}, the resulting picture looks as the one from the topological model in Figure \ref{fig:casq=1} (if $p \neq 1 $) or Figure \ref{fig:pq1} (if $p=1 $).
\end{remarque}

We will now set $\gamma_k:=\gamma_{\xi^k}$ for $k=0,\dots,p+q-1$.

\begin{proposition}\label{prop:leray}
The collection of curves $\{\gamma_k \mid k=0,\dots,p+q-1\}$ from Lemma \ref{lem:alphaxisneu} decomposes $\wA$ into $p+q$ pieces, each being homeomorphic to a closed disc. This decomposition is a Leray covering by closed subsets for each of the sheaves $(\beta^{\vtn_\ell}_{\omega_{\wzeta}>0})_!\cK^{\vtn_\ell}$ with ${\wzeta}\in\nobreak\mu_{p+q}$. It induces the isomorphism \eqref{eq:Psiev}.
\end{proposition}

\begin{proof}
According to Lemma \ref{lem:alphaxisneu}, it is clear that the family of curves $(\gamma_k)_k$ decomposes $\wA$ into $p+q$ pieces homeomorphic to closed sectors of $\wA$. Let $\fS$ be one of these. Then $\fS$ is bounded on each side by $\gamma_k$ and $\gamma_{k+1}$ for some $k$, and additionally bounded by an interval in the circle $\SS^1_{r=0}$ and one of the circle $\SS^1_{r=\infty}$ in $\wA=\{(r,\vt_x) \mid r\in [0,\infty], \ \vt_x\in\SS^1\}$, see Figure \ref{fig:leray}.

Now, for any $\wzeta\in\mu_{p+q}$, we know the shape of $\evB_{\omega_{\wzeta}>0}$ due to Proposition \ref{prop:G1}: namely, $\evB_{\omega_{\wzeta}>0}$ emerged from discs with part of the boundary contained in it (namely the intervals $\evIin{m}$ and $\evIout{n}$) which have been glued according at various points $\wzeta\in\mu_{p+q}$ (depending on the relation $<$). The shape of $\evB_{\omega_{\wzeta}>0}\cap\fS$ is therefore easily determined once we know the shape of the boundary component $\gamma_k\cap\evB_{\omega_{\wzeta}>0}$ for $\xi^k\in\fS$.

If $\xi^k\le_{\vtn_o} {\wzeta}$, it is clear that $\gamma_k \subset \evB_{\omega_{\xi^k}>0} \subset \evB_{\omega_{\wzeta}>0}$. If on the other hand ${\wzeta} <\xi^k$, we proceed as in Lemma \ref{lem:alphaxisneu}. For $x\in\ogamma_k=\xi^{-k}\gamma_k$, we have
\begin{equation}\label{eq:Gxik}
G(\wzeta^{-1}\xi^k x)=q\arg\wzeta+\arg\bigl(\xi^{pk}g(x)+(p+q)(\xi^{pk}-\wzeta^p)\bigr).
\end{equation}
The term $g(x)=:r e^{i\theta_k}$ has fixed argument $\theta_k\in (-\frac{\pi}{2},\frac{\pi}{2})$ (see the construction of $\ogamma_k$ in Lemma \ref{lem:alphaxisneu}). From \eqref{eq:Bfiltre}, we deduce that
\begin{equation}\label{eq:Gchi2}
\arg\bigl(\underbrace{\xi^{pk} \cdot r e^{i\theta_k}}_{\textup{(i)}}+\underbrace{(p+q) (\xi^{pk}- \wzeta^p)}_{\textup{(ii)}}\bigr)\in \ve+\Bigl({-}\frac{\pi}{2}, \frac{\pi}{2}\Bigr)\Longleftrightarrow\xi^k x\in\evB_{\omega_{\wzeta}>0}.
\end{equation}
The summand (ii) in \eqref{eq:Gchi2} is a point with argument in $(\frac{\pi}{2}, \frac{3\pi}{2})$ (since ${\wzeta} <\xi^k$) and therefore $\mathfrak{h} :=$ (i)+(ii) describes a half-line starting from (ii) with direction
$$
\textup{(i)}=\frac{2pk\pi}{p+q}+\theta_k.
$$
Note that $r\to\infty $ both for $|x|\to 0$ or $|x|\to\infty$ and since the starting and endpoint of~$\gamma_k$ is contained in $\evB_{\omega_{\wzeta}>0}$, we see that the latter direction is contained in~\hbox{$\ve+(-\tfrac{\pi}{2}, \tfrac{\pi}{2})$}.
Let $t \mto \ogamma_k(t) $, $t \in [0,2] $ be a regular parametrization of $\ogamma_k$ with $\ogamma_k(1)=1 $ and consider the resulting parametrization
$$
t \mto h(t):= r(t) \cdot \xi^{pk} e^{i\theta_k} + (p+q) (\xi^{pk}- \wzeta^p)
$$
of the half-line $\mathfrak{h}$ with $r(t):=|g(\ogamma_k(t))|$.

\begin{claim*}
The parametrization $h $ of the half-line $\mathfrak{h}$ satisfies
$$
h^{-1}\Bigl(\ve + \Bigl(-\frac{\pi}{2}, \frac{\pi}{2}\Bigr) \Bigr) = \bigl(0, t_0 \bigr) \cup \bigl(t_1,2 \bigr)
$$
for some $0<t_0<1<t_1<2 $.
\end{claim*}

If this claim is proved, we deduce by \eqref{eq:Gchi2} that
$$
\xi^k \ogamma_k(t) \in \evB_{\omega_{\wzeta}>0} \Longleftrightarrow t \in \bigl(0, t_0 \bigr) \cup \bigl(t_1,2 \bigr),
$$
and therefore $\gamma_k\cap\evB_{\omega_{\wzeta}>0}$ is homeomorphic to the union $\mathcal{I}_{\text{in}} \cup \mathcal{I}_{\text{out}}$ of two half-closed intervals, $\mathcal{I}_{\text{in}}$ starting at $\evIin{m}$, and $\mathcal{I}_{\text{out}}$ starting at $\evIout{n}$ with \hbox{$(m\bmod q,n\bmod p)=\cev(k)$}.

\begin{proof}[Proof of the claim]
Due to the construction of $\ogamma_k$, we know that
$G(\ogamma_k(t))=\arg(g(\ogamma_k(t)))=\theta_k$ for all $t $ with $\ogamma_k(t) \neq 1 $ (recall that $G(x) $ is not defined for $x=1 $). We deduce that
$$
\im\bigl(g(\ogamma_k(t)) e^{-i \theta_k}\bigr)=0 \quad\text{and}\quad \reel\bigl(g(\ogamma_k(t)) e^{-i \theta_k}\bigr) \ge 0
$$
for all $t $. Consequently, $r(t)=|g(\ogamma_k(t))|=|g(\ogamma_k(t)) e^{-i \theta_k}|= g(\ogamma_k(t)) e^{-i \theta_k}$. Now
$$
\frac{d}{dt} r(t)= \frac{d}{dt}g(\ogamma_k(t)) e^{-i \theta_k} = g'(\ogamma_k(t)) \cdot \frac{d}{dt} \ogamma_k(t) e^{-i\theta_k}
$$
and since $g'(x)=0 $ if and only if $x \in \mu_{p+q}$, the real function $r(t) $ is strictly monotone on $(0,1) $ and $(1,2) $ with $\lim_{t\to0} r(t)=\lim_{t\to2} r(t)=\infty $ and $\lim_{t\to1} r(t)=0 $.
\end{proof}

We can now finally understand the shape of $\evB_{\omega_{\wzeta}>0}\cap\fS$. Let $\gamma_k, \gamma_{k+1} \subset \fS$ be the boundary curves. Due to Lemma \ref{lem:3xis}, we know that, up to interchanging the inner and outer boundaries, the sector $\fS$ looks as in Figure \ref{fig:fS}.

\begin{figure}[htb]
\begin{center}
\begin{picture}(60,45)
\put(10,0){\includegraphics[scale=0.4]{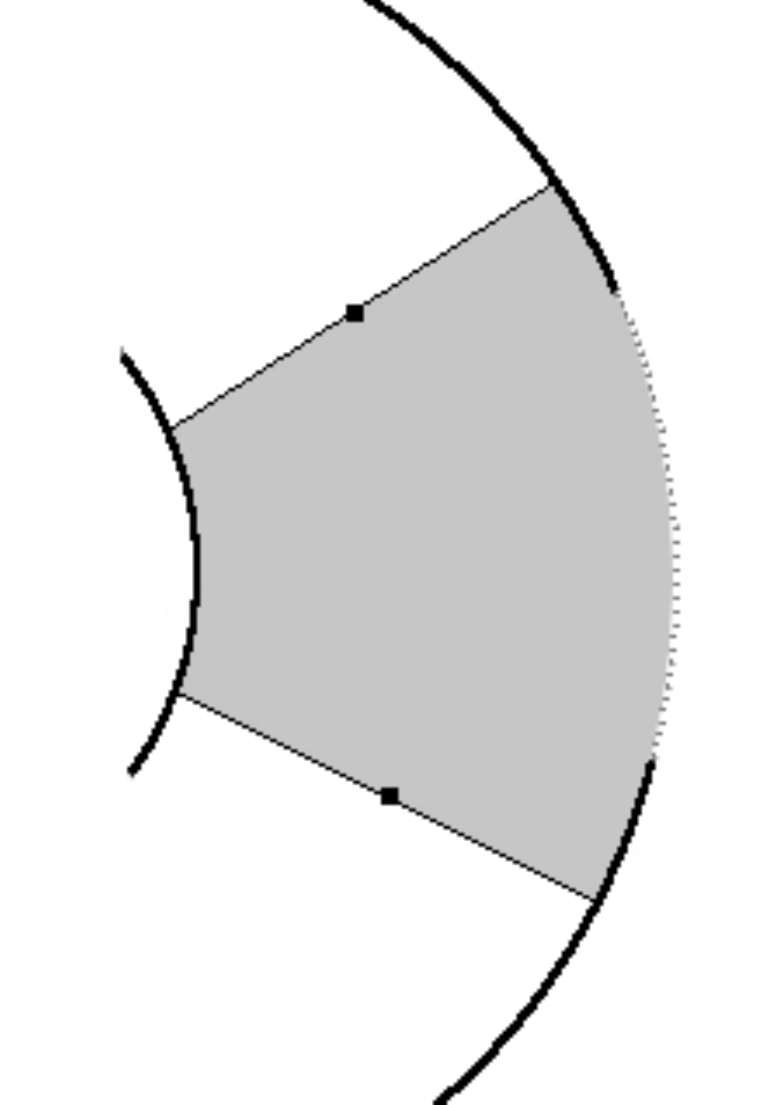}}

\put(7,23){\footnotesize $\evIin{m}$}
\put(22,34){\footnotesize $\xi^{k+1}$}
\put(24,9){\footnotesize $\xi^k$}
\put(33,40){\footnotesize $\evIout{n+1}$}
\put(35,4){\footnotesize $\evIout{n}$}

\qbezier(30,22)(35,25)(40,30)
\put(40,30){\footnotesize $\fS$}

\end{picture}
\end{center}
\caption{The typical shape of $\fS$.}\label{fig:fS}
\end{figure}

We have two cases for the boundaries: either $\gamma_k \subset \evB_{\omega_{\wzeta}>0}$ or $\gamma_k\cap\evB_{\omega_{\wzeta}>0}=\mathcal{I}_{\text{in}} \cup\nobreak \mathcal{I}_{\text{out}}$, and the same for $\gamma_{k+1}$. From our knowledge of the homotopy type of $\evB_{\omega_{\wzeta}>0}$, we deduce that $\evB_{\omega_{\wzeta}>0}\cap\fS$ is homeomorphic to one of the three spaces sketched in Figure \ref{fig:fS2}.

\begin{figure}[htb]
\begin{center}
\begin{picture}(125,45)
\put(10,0){\includegraphics[scale=0.4]{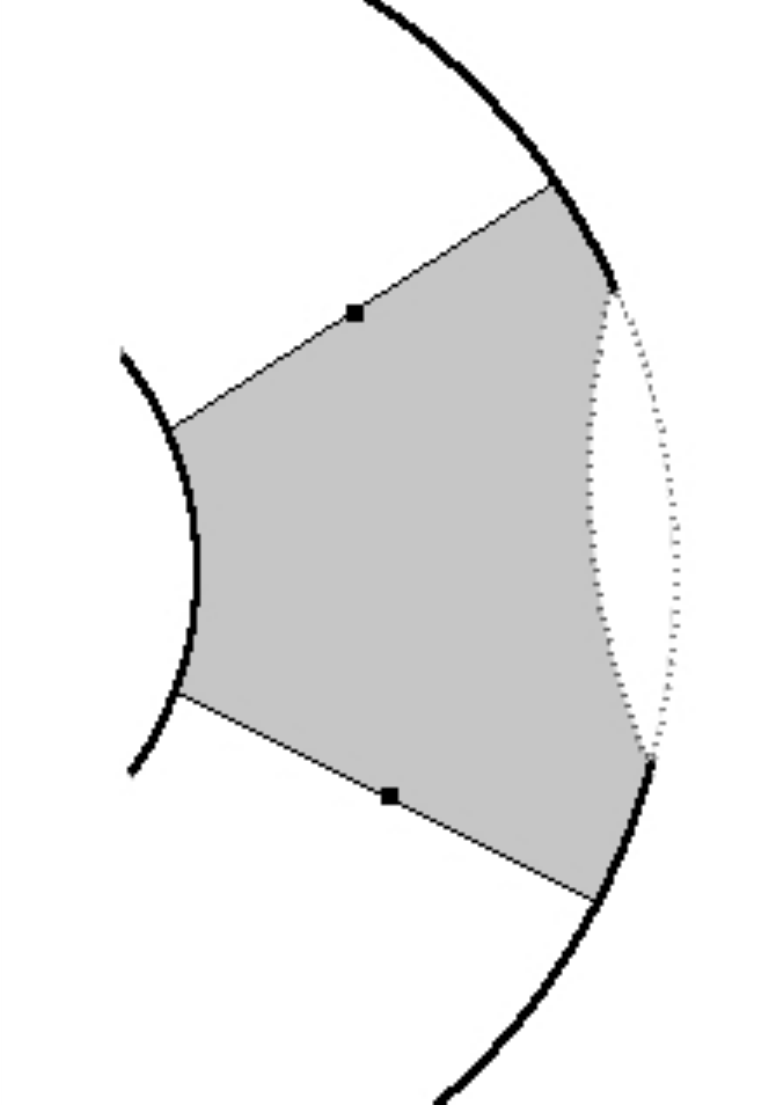}}
\put(50,0){\includegraphics[scale=0.4]{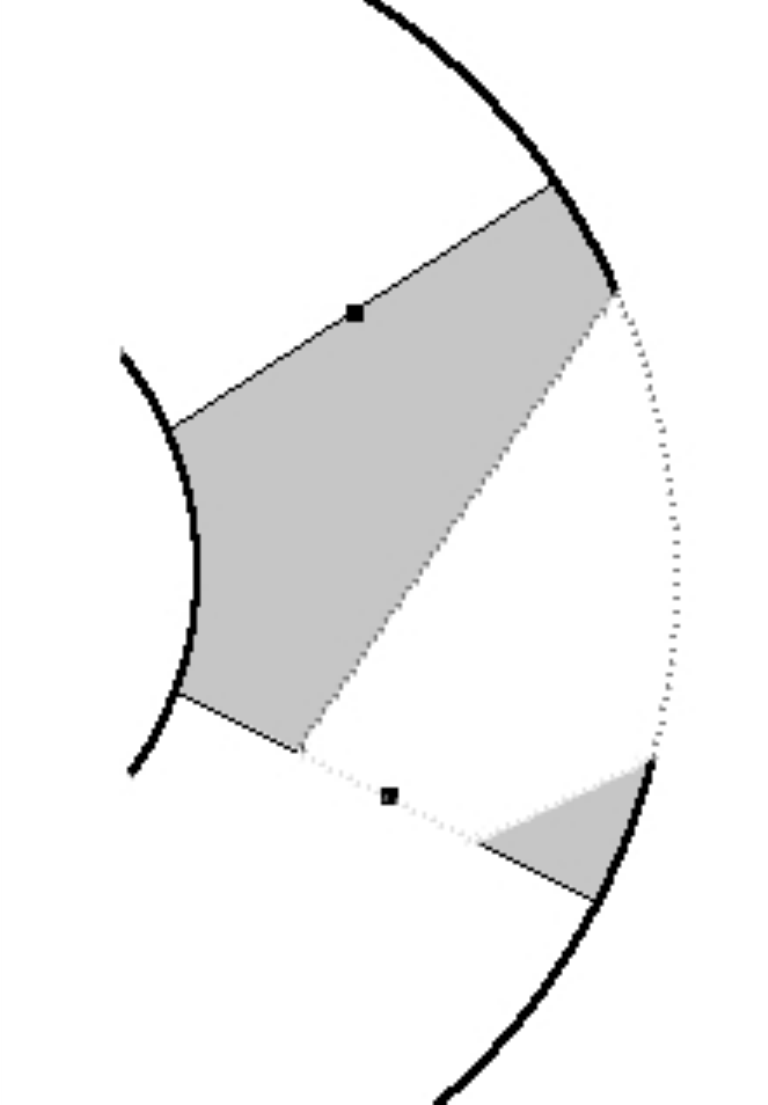}}
\put(85,0){\includegraphics[scale=0.4]{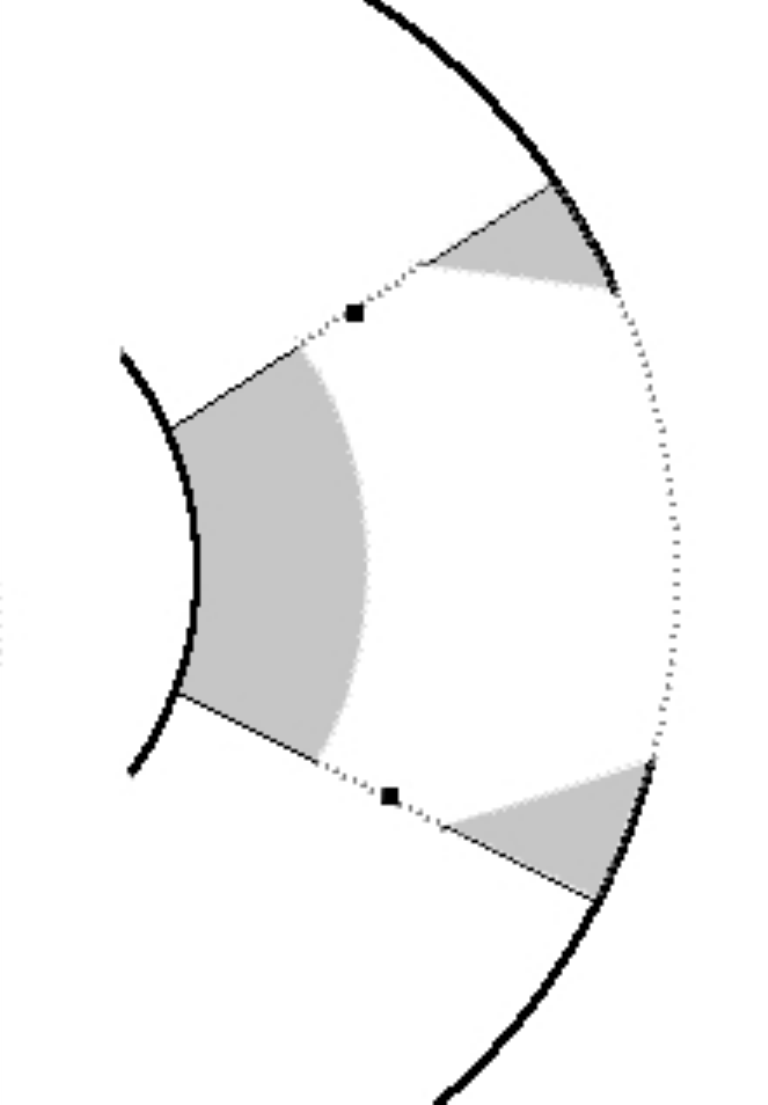}}

\put(7,23){\footnotesize $\evIin{m}$}
\put(22,34){\footnotesize $\xi^{k+1}$}
\put(24,9){\footnotesize $\xi^k$}
\put(33,40){\footnotesize $\evIout{n+1}$}
\put(35,4){\footnotesize $\evIout{n}$}

\qbezier(30,22)(35,25)(40,30)
\put(40,32){\footnotesize $\fS\cap\evB_{\omega_{\wzeta}>0}$}
\qbezier(49,28)(58,20)(60,20)
\qbezier(45,29)(55,0)(73,10)

\put(85,39){\footnotesize $\fS\cap\evB_{\omega_{\wzeta}>0}$}
\qbezier(89,37)(90,35)(96,25)
\qbezier(95,37)(97,36)(105,35)
\qbezier(92,37)(102,20)(108,12)
\end{picture}
\end{center}
\caption{Up to homeomorphism, the three possible shapes of $\evB_{\omega_{\wzeta}>0}\cap\fS$.}\label{fig:fS2}
\end{figure}

We conclude that each connected component of $\evB_{\omega_{\wzeta}>0}\cap\fS$ is homeomorphic to a closed disc with a closed interval deleted in its boundary, and each connected component of an intersection $\evB_{\omega_{\wzeta}>0}\cap\fS\cap\fS'$ is homeomorphic to a closed or a semi-closed interval. There are no triple intersections for this covering. It is thus clear that this covering is Leray for $(\beta_{\omega_{\wzeta}>0}^{\vtn_\ell})_!\cK^{\vtn_\ell}$ ($\ell$ even).

Finally, the properties of the curves $\gamma_k$ proven in Lemma \ref{lem:alphaxisneu} correspond to the behaviour of the standard Leray covering $\evfS $ described in Proposition \ref{prop:standtopmodel}, and we obtain the final assertion.
\end{proof}

\subsection{The odd case}
We now have $B^{\vtn_\ell}_{\omega_{\wzeta}>0}=\oddB_{\omega_{\wzeta}>0}$ (\cf\eqref{eq:Bfiltre}). In analogy to the procedure in the standard case as in section \ref{sec:stdodd}, we deduce the odd case from the even one by rotation by $\ximinodd $ as defined before \eqref{eq:oddalpha}. Note that
\begin{equation}\label{eq:ximin}
q\arg(\ximinodd)=
\left\{
\begin{array}{ll}
\pi & \text{if $p+q \equiv 0$ mod $2$}\\[3pt]
\pi-\dfrac{\pi}{p+q} & \text{ if $p+q \equiv 1$ mod $2$.}
\end{array}\right.
\end{equation}
The minimal subset to be considered is $\oddB_{\omega_{\ximinodd}>0}$ which is contained in $\oddB_{\omega_{\wzeta}>0}$ for all $\wzeta\in\mu_{p+q}$:
$$
\oddB_{\omega_{\ximinodd}>0}=
\begin{cases}
\ximinodd\cdot G^{-1}\Bigl(\Bigl(- \dfrac{\pi}{2}+\ve, \dfrac{\pi}{2}+\ve\Bigr)\Bigr) & \text{if $p+q \equiv 0$ mod $2$}\\[.3cm]
\ximinodd\cdot G^{-1}\Bigl(- \dfrac{\pi}{p+q}+\Bigl(- \dfrac{\pi}{2}+\ve, \dfrac{\pi}{2}+\ve\Bigr)\Bigr) & \text{if $p+q \equiv 1$ mod $2$.}
\end{cases}
$$
Since there is no critical value of $G $ inside $-\frac{\pi}{p+q}+(-\frac{\pi}{2}+\ve, \frac{\pi}{2}+\ve)$, the subset
$$
G^{-1} \Bigl(- \frac{\pi}{p+q}+\Bigl(- \frac{\pi}{2}+\ve, \frac{\pi}{2}+\ve \Bigr) \Bigr)
$$
is homeomorphic to $B^{\vtn_o}_{\omega_1<0}$. Therefore, $\oddB_{\omega_{\ximinodd}>0}$ is either equal to $\evB_{\omega_1>0}$ rotated by $\ximinodd$ (in the first case), or a small perturbation of it in the sense of applying the flow of the gradient vector field of the Morse function considered for a time interval without critical points.

Furthermore, considering arbitrary $\wzeta\in\mu_{p+q}$, we have
\begin{multline}\label{eq:oddBtwist}
\oddB_{\omega_{\wzeta}>0}\\
\shoveright{=\wzeta G^{-1} \Bigl(q\arg{\wzeta}-q\arg(\ximinodd)+q\arg(\ximinodd)+ \Bigl(\frac{\pi}{2}+\ve, \frac{3\pi}{2}+\ve \Bigr)\Bigr)}\\
=\wzeta G^{-1}\Bigl(q\arg(\wzeta (\ximinodd)^{-1}) - \sigma \frac{\pi}{p+q}+ \Bigl(-\frac{\pi}{2}+\ve, \frac{\pi}{2}+\ve \Bigr)\Bigr)
\end{multline}
with $\sigma=0$ if $p+q \equiv 0 \bmod 2$ and $\sigma=1$ otherwise. This space is again either equal to $\ximinodd \cdot \evB_{\omega_{\wzeta (\ximinodd)^{-1}}>0}$ or a perturbation of it in the same sense as above.

Let us define
$
\wt\gamma_{\wzeta}:=\ximinodd \cdot \gamma_{\wzeta (\ximinodd)^{-1}} \subset \oddB_{\omega_{\wzeta}>0}
$,
with $\gamma_{\wzeta}$ as in Lemma \ref{lem:alphaxisneu}. If $p+q $ is even, the whole situation to be considered is obtained from the one in the even case by rotation by $\pi $. Therefore, the statements of Lemma \ref{lem:alphaxisneu} and Proposition \ref{prop:leray} remain true for the curves $\wt\gamma_{\wzeta}$ and using $\oddB $, $\oddDin $, $\oddDout $ and $\codd $ instead of their even counterparts. We have to verify that this is still true for $p+q $ being odd.

Properties \eqref{lem:alphaxisneu1} and \eqref{lem:alphaxisneu4} of Lemma \ref{lem:alphaxisneu} for the family $(\wt\gamma_{\vtn})_{\vtn}$ are trivially satisfied. Furthermore, we have
$$
\wt\gamma_{\vtn} = \ximinodd \cdot \vtn \cdot (\ximinodd)^{-1} \ogamma_{\vtn \cdot (\ximinodd)^{-1}} = \vtn \cdot \ogamma_{\vtn \cdot (\ximinodd)^{-1}}.
$$
Now, each $z \in \wt\gamma_{\vtn}$ is of the form $z=\vtn \cdot x $ for some $x \in \ogamma_{\vtn \cdot (\ximinodd)^{-1}}$ and we deduce from \eqref{eq:thetachi} that
$$
G(\vtn^{-1} \cdot z)=
G(x) = \theta_{\vtn \cdot (\ximinodd)^{-1}} \in q\arg(\vtn \cdot (\ximinodd)^{-1}) +\ve +\Bigl({-}\frac{\pi}{2}, \frac{\pi}{2}- \frac{\pi}{p+q}\Bigr),
$$
and \eqref{eq:oddBtwist} yields $z \in \oddB_{\omega_{\vtn}>0}$. This proves property \eqref{lem:alphaxisneu2} of Lemma \ref{lem:alphaxisneu} for the family $\wt\gamma_{\vtn}$.

Property \eqref{lem:alphaxisneu3} of Lemma \ref{lem:alphaxisneu} follows due to the fact that \eqref{lem:alphaxisneu4} holds and therefore
\begin{align*}
\wt\gamma_{\wzeta}\cap\bin\ov A&\subset \ximinodd\cdot\evDin_m\cap\bin(\oddB_{\omega_{\wzeta}>0}),\\
\tag*{and}
\wt\gamma_{\wzeta}\cap\bout\ov A&\subset \ximinodd\cdot\evDout_n\cap\bout(\oddB_{\omega_{\wzeta}>0}),
\end{align*}
with $(m,n)=\cev(k) $ for $\vtn=\xi^k$ together with \eqref{align:inDoddD}.

We also have the analogous statement to Proposition \ref{prop:leray}:

\begin{proposition}\label{prop:lerayoddodd}
The collection of curves $\{\wt\gamma_{\vtn} \mid \vtn \in \mu_{p+q} \}$ decomposes $\wA$ into $p+q$ pieces, each being homeomorphic to a closed disc. This decomposition is a Leray covering by closed subsets for each of the sheaves $(\beta^{\vtn_\ell}_{\omega_{\wzeta}>0})_!\cK^{\vtn_\ell}$ with ${\wzeta}\in\nobreak\mu_{p+q}$. It induces the isomorphism \eqref{eq:Psiodd} ($\ell$ odd).
\end{proposition}
\begin{proof}
We only have to consider the case $p+q $ being odd. The proof goes along the same lines as the one of Proposition \ref{prop:leray} with the following changes. Again, it suffices to understand the shape of $\wt\gamma_{\xi^k} \cap \oddB_{\omega_{\wzeta}>0}$ for $\xi^k, \wzeta \in \mu_{p+q}$.

If $\xi^k \le_{\vtn_\ell} \wzeta $, then $\wt\gamma_{\xi^k} \subset \oddB_{\omega_{\xi^k}>0} \subset \oddB_{\omega_{\wzeta}>0}$. Otherwise, we proceed in the same way. For $\xi^k x \in \xi^k \cdot \ogamma_{\xi^k(\ximinodd)^{-1}} = \wt\gamma_{\xi^k}$, we have the equality \eqref{eq:Gxik}:
$$
G(\wzeta^{-1}\xi^k x)=q\arg\wzeta+\arg\bigl(\xi^{pk}g(x)+(p+q)(\xi^{pk}-\wzeta^p)\bigr).
$$
The term $g(x)=:r e^{i\theta}$ has fixed argument $\theta:=\arg\bigl(\xi^k(\ximinodd)^{-1}\bigr)$. From \eqref{eq:oddBtwist}, we deduce that
\begin{equation}\label{eq:halflineodd}
\arg\bigl(\underbrace{\xi^{pk} \cdot r e^{i\theta_k}}_{\textup{(i)}}+\underbrace{(p+q) (\xi^{pk}- \wzeta^p)}_{\textup{(ii)}}\bigr)\!\in\! \ve- \frac{\pi}{p+q}+\Bigl({-}\frac{\pi}{2}, \frac{\pi}{2}\Bigr)
\Leftrightarrow\xi^k x\in\evB_{\omega_{\wzeta}>0}.
\end{equation}
The sum (i)+(ii) again describes a half-line in the complex plane. The argument used in the proof (Claim) of Proposition \ref{prop:leray} shows that choosing a regular parametrization of the curve $\wt\gamma_{\xi^k}$, the half-line $\mathfrak{h}$ inherits a parametrization $[0,2] \to \mathfrak{h}$, $t \mto \textup{(ii)}+ r(t)e^{i\theta}$ with $r(1)=0 $ such that $r(t) $ is monotone on $(0,1) $ and $(1,2) $ and satisfies $\lim_{t\to0} r(t)=\lim_{t\to2} r(t)=\infty $. Since the half-line has to be inside the region determined by \eqref{eq:halflineodd} for $r $ big enough (due to the fact that the endpoints of $\wt\gamma_{\xi^k}$ lie in $\oddB_{\omega_{\wzeta}>0}$) and since (ii) cannot be inside this region, we obtain the same conclusion as in the proof of Proposition \ref{prop:leray}.
\end{proof}

Since the Leray coverings of Proposition \ref{prop:leray} ($\ell $ even) and \ref{prop:lerayoddodd} ($\ell $ odd) respect the filtrations (\cf \eqref{eq:Filell}) given by the subspaces $\evB_{\omega_{\wzeta}>0}$ (\resp odd), we finally obtain the following corollary concluding the proof of Theorem \ref{thm:mainintro}:
\begin{corollaire}\label{cor:isofil}
The isomorphisms $\Psi_\ell $ \eqref{eq:Labstrsimple} are compatible with the filtration \eqref{eq:Filell} and the filtration defined in \ref{def:standfilev} for $\ell $ even (\resp \ref{def:standfilodd} for $\ell $ odd).
\end{corollaire}

\backmatter
\def\ieme{rd\xspace}
\providecommand{\SortNoop}[1]{}\providecommand{\eprint}[1]{\href{http://arxiv.org/abs/#1}{\texttt{arXiv\string:\allowbreak#1}}}
\providecommand{\bysame}{\leavevmode ---\ }
\providecommand{\og}{``}
\providecommand{\fg}{''}
\providecommand{\smfandname}{\&}
\providecommand{\smfedsname}{\'eds.}
\providecommand{\smfedname}{\'ed.}
\providecommand{\smfmastersthesisname}{M\'emoire}
\providecommand{\smfphdthesisname}{Th\`ese}

\end{document}